\documentclass[11pt]{amsart}

\usepackage{amsmath}
\usepackage{amssymb}
\usepackage{amsrefs}
\usepackage{stmaryrd}
\usepackage{graphicx}

\newtheorem{theorem}{Theorem}[section]
\newtheorem{proposition}[theorem]{Proposition}
\newtheorem{lemma}[theorem]{Lemma}
\newtheorem{cor}[theorem]{Corollary}

\theoremstyle{definition}
\newtheorem{definition}[theorem]{Definition}
\newtheorem{example}[theorem]{Example}

\theoremstyle{definition}
\newtheorem{remark}[theorem]{Remark}
\newtheorem{assumption}[theorem]{Assumption}

\numberwithin{equation}{section}

\newcommand{\rawFilt}{{\mathbb{F}^0}} 
\newcommand{\Filt} {{\mathbb{F}^0_+}}

\newcommand{\rawField}[1]{{\mathcal{F}^0_{#1}}}
\newcommand{\Field}[1]{{\mathcal{F}^0_{{#1}+}}}

\newcommand{\bfnull}{{\mathbf{0}}}
\newcommand{\bfone}{{\mathbf{1}}}

\newcommand{\abs}[1]{\left|{#1}\right|}
\newcommand{\norm}[1]{\lVert{#1}\rVert}

\newcommand{\D}{{\mathbb{D}}}
\newcommand{\F}{{\mathbb{F}}}

\newcommand{\N}{{\mathbb{N}}}

\newcommand{\R}{{\mathbb{R}}}
\newcommand{\Q}{{\mathbb{Q}}}

\newcommand{\Mean}{{{\mathbb{E}}}}
\newcommand{\Prob}{{{\mathbb{P}}}}

\newcommand{\llb}{{\llbracket}}
\newcommand{\rrb}{{\rrbracket}}

\newcommand{\cA}{{\mathcal{A}}}

\newcommand{\cF}{{\mathcal{F}}}
\newcommand{\cH}{{\mathcal{H}}}
\newcommand{\cL}{{\mathcal{L}}}

\newcommand{\cT}{{\mathcal{T}}}

\newcommand{\overh}{{\overline{h}}}

\newcommand{\om}{{\omega}}

\newcommand{\loc}{{\text{loc}}}

\newcommand{\eps}{{\varepsilon}}

\newcommand{\tiom}{{\tilde{\omega}}}

\newcommand{\ttau}{{\tilde{\tau}}}

\newcommand{\esssup}{\operatornamewithlimits{esssup}}

\newcommand{\dist}{{\mathbf{d}}}

\def\ch{\textsc{h}}

\newcommand{\sint}{\stackrel{\mbox{\tiny$\bullet$}}{}}

\begin{document}
\title[PPIDE]{Viscosity Solutions of\\Path-de\-pendent\\ Integro-Differential Equations}
\author{Christian Keller}
\address{Department of Mathematics, 
University of Southern California}
\email{kellerch@usc.edu}

\date{12/29/2014}



\thanks{The author would like
to thank Remigijus Mikulevicius
and Jianfeng Zhang for very valuable
discussions.}

\begin{abstract}
We extend the notion of viscosity solutions for
path-de\-pendent PDEs introduced by
Ekren et al.~[\textit{Ann.~Probab.~}\textbf{42} (2014), no.~1, 204-236]
to path-dependent integro-differential equations and
 establish well-posedness, i.e., existence, uniqueness,
and stability, for a class
of semilinear path-dependent
integro-differential equations with uniformly continuous data.
Closely related are non-Markovian
backward SDEs with jumps, which provide a probabilistic
representation for solutions of our equations. 
The results are potentially useful for  applications using
non-Markovian jump-diffusion models.
\end{abstract}

\maketitle

\section{Introduction}
The goal of this paper is to extend the
theory of viscosity solutions
(in the sense of \cite{EKTZ11}
and \cite{ETZ12a})
for  path-dependent partial differential
equations (PPDEs) to path-dependent 
integro-differential equations. In particular,
we investigate semilinear path-dependent integro-differential equations of
the form
\begin{equation}\label{E:PPIDE}
\begin{split}
\mathcal{L}u(t,\om)-f_t(\om,u(t,\om),
\partial_\om u(t,\om),\mathcal{I}u(t,\om))=0, \\(t,\om)\in
[0,T)\times\mathbb{D}([0,T],\R^d),
\end{split}
\end{equation}
where $\mathbb{D}([0,T],\R^d)$ is the space of right-continuous
functions with left limits from $[0,T]$ to $\R^d$,  $\mathcal{L}$ is a linear integro-differential operator of the form
\begin{align*}
\mathcal{L}u(t,\om)&=
-\partial_t u(t,\om)-\sum_{i=1}^d b^i_t(\om)\partial_{\om^i}u(t,\om)
-\frac{1}{2}\sum_{i,j=1}^d c^{ij}_t(\om)\partial^2_{\om^i\om^j}u(t,\om)\\ &\qquad
-\int_{\R^d} \left[u(t,\om+z.\bfone_{[t,T]})-u(t,\om)-
\sum_{i=1}^d z^i\,\partial_{\om^i}u(t,\om)\right]\,K_t(\om,dz),
\end{align*}
and $\mathcal{I}$ is an integral operator of the form
\begin{align*}
\mathcal{I}u(t,\om)=\int_{\R^d} \left[u(t,\om+z.\bfone_{[t,T]})-u(t,\om)
\right]\,\eta_t(\om,z)\,K_t(\om,dz).
\end{align*}

Well-posedness for semilinear PPDEs 
has been first established by 
Ekren, Keller, Touzi, and Zhang
\cite{EKTZ11}, where
also the here used notion of viscosity solutions
has been introduced. Subsequent work by Ekren, Touzi and Zhang
deals
with fully nonlinear PPDEs 
(\cite{ETZ12a}
and \cite{ETZ12b}), by Pham and Zhang with
path-dependent Isaacs equations 
 (\cite{PhamZhang12}), 
 by Ekren with obstacle PPDEs (\cite{Ekren13}),
  and by Ren with
 fully nonlinear elliptic PPDEs
 (\cite{Ren14}).

Initial motivation for this line of research
came from Peng \cite{PengICM}, who considered
non-Markovian backward stochastic differential
equations (BSDEs) as PPDEs analogously 
to the relationship between Markovian BSDEs
and (standard) PDEs,
from Dupire \cite{dupirefunctional}, who
introduced new derivatives on $\mathbb{D}([0,T],\R^d)$
so that for smooth functionals on
$[0,T)\times\mathbb{D}([0,T],\R^d)$ a functional counterpart
to It\^os formula holds,
and from Cont and Fourni\'e
(see \cite{ContFournie10_CompteRendus},
\cite{ContFournie10_JFA},  and
\cite{ContFournier13_AOP}),
who extended Dupire's seminal work.
However, fully nonlinear PPDEs of first order
have been studied earlier by Lukoyanov
(see, for example, \cite{Lukoyanov00},
\cite{Lukoyanov03},
\cite{Lukoyanov07}).
He used derivatives introduced by Kim
\cite{KimBook} and adapted first 
the notion of so-called minimax solution
and then of viscosity solutions from PDEs
to PPDEs. Minimax solutions for PDEs have
been introduced by Subbotin 
(see, e.g., \cite{Subbotin80} and \cite{Subbotin84})
motivated by the study of differential games.
In the case of PDEs of first order, 
minimax and viscosity solutions are equivalent
(see \cite{SubbotinBook}). Another approach for 
generalized solution for first-order PPDEs
can be found in work by Aubin and Haddad \cite{AubinHaddad02}, where
so-called Clio derivatives for path-dependent
functionals are introduced in order to 
study certain path-dependent Hamiltion-Jacobi-Bellman
equations that occur in portfolio theory.

Possible applications of path-dependent integro-differential
equations are 
non-Markovian problems in
control, differential games, 
and financial mathematics that 
involve jump processes.  

Some comments about differences between PDEs and PPDEs
seem to be in order.
Contrary to standard PDEs, even linear PPDEs have rarely classical solutions
in most relevant situations. Hence, one needs to consider a weaker
forms of solutions. In the case of PDEs, the notion
of viscosity solutions introduced by Crandall and Lions
\cite{CrandallLions83TAMS} turned out
to be extremely successful.
 The main difficulty in the path-dependent
case compared to the standard PDE case is the lack of local compactness
of the state space, e.g., $[0,T]\times\mathbb{D}$  vs.~$[0,T]\times\R^d$.
Local compactness is essential for proofs of uniqueness of viscosity solutions
to PDEs, i.e., PDE standard methods can, in general, not easily adapted to
the path-dependent case. The main contribution of \cite{EKTZ11} was
to replace the pointwise supremum/ infimum occuring in the definition
of viscosity solutions to PDEs via test functions by an optimal
stopping problem. The lack of local compactness could be
circumvented by the existence of an optimal stopping time.
This is crucial in establishing the comparison principle.
 In this paper, additional
intricacies caused by the jumps have to be faced.
 For example, it turns  out
that in contrast to the PPDE case
 the uniform topology is not always
appropriate. In order to prove the comparison principle, it seems
necessary to equip $\mathbb{D}$ with one of Skorohod's nonuniform
topologies.

The general methodology to establish well-posedness
of viscosity solutions for \eqref{E:PPIDE}
follows \cite{EKTZ11} and \cite{ETZ12a}.
Existence will be proven by
using a stochastic representation.
An intermediate result is
a so-called partial comparison principle,
which is a comparison
principle, where one the involved solutions
is a viscosity subsolution
(resp.~a viscosity supersolution),
and the other one
a classical super- (resp.~a classical 
subsolution).
The partial comparison principle is essential to
prove the comparison principle. 


The rest of the paper is organized as follows.
In Section~2, we introduce most of the notation
and preliminaries. In Section~3, viscosity solutions for semilinear 
path-dependent integro-differential equation
are defined and the main results are stated.
In Section~4, we prove consistency of classical
solutions with viscosity solutions as 
well as existence of viscosity solutions.
In Section~5, the partial comparison principle
and a stability result is proved. This section
also contains some auxiliary results about
backward SDEs and optimal stopping.
In Section~6, the comparison principle is proved.
Appendix~A deals with conditional probability
distributions and their applications to 
martingale problems.
In Appendix~B, Skorohod's topologies are
defined.
Appendix~C contains additional auxiliary results.
\section{Setup}
\subsection{Notation and preliminaries}
For unexplained notation, we refer to
\cite{JacodShiryaevBook}
and \cite{RogersWilliamsI}.

Let $\N=\{1,2,\ldots\}$ be the set of all strictly positive integers,
$\N_0:=\N\cup\{0\}$, $\Q$ be the set of rational numbers, and
$\R$ be the set of real numbers. 
Given $d^\prime\in\N$, we denote by
$\mathbb{S}^{d^\prime}$  the set of all symmetric real-valued 
 $d^\prime\times d^\prime$-matrices. For any matrix $A$, we denote by 
 $A^\top$ its transpose. Given a topological space $E$, let
 $\mathcal{B}(E)$ its Borel $\sigma$-field.
 We write $\bfnull$ for zero vectors, zero matrices, constant functions
 attaining only the value $0$, etc. The meaning should be clear from 
 context. We write $\bfone$ for indicator functions.
 The expected value with respect to some probability measure $\Prob$
 is denoted by $\Mean^\Prob$.  On $\R^{d^\prime}$, $d^\prime\in\N$,
 denote the $\ell^p$-norms by $\abs{\cdot}_p$, $p\in\N\cup\{\infty\}$.
 Also set $\abs{\cdot}:=\abs{\cdot}_2$.
  
 Fix $T>0$ and $d\in\N$.
 Let $\Omega:=\mathbb{D}([0,T],\R^d)$ be the canonical space,
$X$ the canonical process on $\Omega$, i.e.,
$X_t(\om)=\om_t$, and
$\F^0=\{\cF^0_t\}_{t\in [0,T]}$ the (raw) filtration generated by $X$.
Denote the right-limit of $\F^0$ by 
$\F^0_+=\{\cF^0_{t+}\}_{t\in [0,T]}$.
 Given $t\in [0,T]$, let $\Lambda^t:=[t,T)\times\Omega$ and
$\bar{\Lambda}^t:=[t,T]\times\Omega$.  Also, put
$\Lambda:=\Lambda^0$ and $\bar{\Lambda}:=\bar{\Lambda}^0$.
Given random times $\tau_1$, $\tau_2:\Omega\to [0,\infty]$, put
\begin{align*}
\llb \tau_1,\tau_2\rrb:=\{(t,\om)\in\bar{\Lambda}:\tau_1(\om)\le t\le\tau_2(\om)\}.
\end{align*}
The other \emph{stochastic intervals} $\llb \tau_1,\tau_2\llb$, etc., are defined
similarly.
We equip $\Omega$ with the uniform norm
$\norm{\cdot}_\infty$ and $\bar{\Lambda}$
with the pseudometric $\dist_\infty$ defined by
\begin{align*}
\dist_\infty((t,\om),(t^\prime,\om^\prime)):=
\abs{t-t^\prime}+\norm{
\om_{\cdot\wedge t}
-\om^\prime_{\cdot\wedge t^\prime}
}_\infty.
\end{align*}

Often, we consider a functional $u:\bar{\Lambda}\to\R$ as a stochastic
process, in which case, we write $u_t$ instead of $u(t,X)$.
\begin{definition}
Let $E_1$ and $E_2$ be  nonempty sets. Let $A$ be
a nonempty subset of $\bar{\Lambda}\times E_1$.
Consider a mapping $u=u(t,\om,x):A\to E_2$. We call $u$
\emph{non-anticipating} if, for every 
$(t,\om,x)\in A$,
\begin{align*}
u(t,\om,x)=u(t,\om_{\cdot\wedge t},x).
\end{align*}
We call $u$ \emph{deterministic} if it does not depend on $\om$.
\end{definition}

Given a nonempty subset $A$ of $\bar{\Lambda}$ and a topological space $E$,
we denote by $C(A,E)$ the set of all functionals from $A$ to $E$ that are
continuous under $\dist_\infty$. If $E=\R$, we just write $C(A)$ instead.
 
\begin{remark}\label{R:C(Lambda)}
 Note that any $u\in C(\bar{\Lambda})$
satisfies the following:

(i)  $u$ is non-anticipating.
This follows immediately from the definition
of $\dist_\infty$.

(ii) The trajectories
 $t\mapsto u(t,\om)$ are c\`adl\`ag
and the trajectories 
$t\mapsto u(t,\om_{\cdot\wedge t-})$
are left-continuous (Proposition~1 of \cite{ContFournie10_JFA}).
Also, for fixed $t\in (0,T]$,the path
$\tiom:=\om_{\cdot \wedge t-}$ is
continuous at $t$, which again, by
Proposition~1 of \cite{ContFournie10_JFA},
implies that
\begin{align*}
u_{t-}(\om)=\lim_{s\uparrow t}
 u(s,\tiom)=u(t,\tiom)=u(t,\om_{\cdot\wedge t-}).
\end{align*}
That is, $u_-=(u(t,\om_{\cdot\wedge t-}))_t$.
Moreover, considered as processes,
$u$ and $u_-$ are $\F^0$-adapted
(Theorem~2 of \cite{ContFournie10_JFA}).

(iii) $X$ jumps whenever $u$ jumps
(Lemma~\ref{L:XjumpsWhenUjumps}).
\end{remark}

Often, we write $H\sint S$ for stochastic integrals with
respect to semimartingales, i.e.,
$H\sint S_t=\int_{s}^t H_r\,dS_r$. The initial time $s$
is usually clear from context. We also write sometimes
$H\sint t$ instead of $\int H_t\,dt$.  Similarly, we write
$W\ast \mu$ for stochastic integrals with respect to random measures
(see Chapter~II of \cite{JacodShiryaevBook}).
Given a probability measure $\Prob$, denote by
$\F^\Prob$ its induced filtration satisfying the usual
conditions. If $S$ is an $(\F^\Prob,\Prob)$-semimartingale,
write $L^2_{\loc}(S,\Prob)$ for the set of all $\F^\Prob$-predictable 
processes $H$ such that $H^2\sint\langle X,X\rangle$ is locally integrable
(cf.~I.439 in \cite{JacodShiryaevBook}). Similarly, given a random measure
$\mu$, the set $G_\loc(\mu,\Prob)$ is defined (see~Definition~II.1.27
in \cite{JacodShiryaevBook}).
Given a process $S$ with left limits, define $\Delta S$ by
$\Delta S_t:=S_t-S_{t-}$.
If $S$ is a semimartingale
under a probability measure $\Prob$, then we denote
by $S^c=S^{c,\Prob}$ the 
 continuous local martingale part
  of $S$ 
(p.~45 in \cite{JacodShiryaevBook}) and by
$\mu^S$ the random measure 
 associated to the jumps of $S$
(p.~69 in \cite{JacodShiryaevBook}). It is defined by
\begin{align*}
\mu^S(\cdot;dt,dx):=\sum_s
\bfone_{\{\Delta S_s\neq 0\}}\,
\delta_{(s,\Delta S_s)} (dt,dx),
\end{align*}
where $\delta$ denotes the Dirac measure.

Given a nonempty domain $D$ in $\R^{d^\prime}$, $d^\prime\in\N$,
denote by $\norm{\cdot}_{n+\alpha,D}$, $n\in\N_0$, $\alpha\in (0,1)$,
the standard H\"older norms. 
Similarly, denote by  $\norm{\cdot}_{n+\alpha,Q}$, $Q=(t_1,t_2)\times D$,
$t_1<t_2$, the corresponding parabolic H\"older norms. We refer to \cite{Mikulevicius1994classical}
for the definition.
The corresponding H\"older spaces are denoted by
$C^{n+\alpha}(\bar{D})$ and $C^{n+\alpha}(\bar{Q})$, resp., and the corresponding
local H\"older spaces by
$C_{\text{loc}}^{n+\alpha}({D})$ and
 $C_{\text{loc}}^{n+\alpha}({Q})$, resp.
Also, put $\norm{\cdot}_D:=\norm{\cdot}_{0,D}$ and 
$\norm{\cdot}_Q:=\norm{\cdot}_{0,Q}$.
 
\subsection{Standing assumptions}
The assumptions in this section are always in force unless explicitly stated
otherwise.

Let $b=(b^i)_{i\le d}$
be a $d$-dimensional,
non-anticipating, and
 $\F^0_+$-predictable process,
$c=(c^{ij})_{i,j\le d}$  a
non-anticipating and
$\F^0_+$-predictable
process with values in the set of  nonnegative definite
real $d\times d$-matrices, 
and  $K=K_t(dz)$  a
non-anticipating and $\F^0_+$-predictable
process with values  in the set of 
$\sigma$-finite measures on $\mathcal{B}(\R^d)$.

\begin{assumption}\label{A:diffChar}
Let $(b,c,K)$ satisfy 
$c^{ij}=\sum_{k\le d} \sigma^{ik}\sigma^{jk}$,
$i$, $j\le d$, and
$K_t(A)=\int \bfone_{A\setminus\{\bfnull\}}
(\delta_t(z))\,F(dz)$, $A\in\mathcal{B}(\R^d)$,
where $\sigma=(\sigma^{i,j})_{i,j\le d}$ is
a non-anticipating and
 $\F^0_+$-predictable process with values in the set of
real $d\times d$-matrices,
 $\delta=(\delta^i)_{i\le d}$ is
a $d$-dimensional,
non-anticipating, and
 $\F^0_+$-predictable
 random field on $\R^d$, and $F$ is a nonnegative
 $\sigma$-finite
 measure on $\mathcal{B}(\R^d)$.
Let $b$, $\sigma$, and $\delta$ be right-continuous in $t$.
Let $b$ and $\sigma$ be
bounded  by a common constant $C^\prime_0\ge 1$
and  Lipschitz continuous in $\om$ with common
Lipschitz constant $L_0\ge 1$.
Let $\delta_t(\cdot,z)$ be bounded by $\abs{z}\wedge C^\prime_0$
and Lipschitz continuous with
Lipschitz constant $L_0(\abs{z}\wedge C^\prime_0)$.
Also, assume that $\int_{\R^d} \abs{z}^2\wedge C^\prime_0\,F(dz)\le C^\prime_1$
for some constant $C^\prime_1\ge 1$.
Moreover, let $d=1$ or let $\sigma_s(\om)$ be
invertible for every $(s,\om)\in\Lambda$.
\end{assumption}

Let
$\eta=\eta_t(\om,z):\bar{\Lambda}\times\R^d\to\R$ and
$f=f_t(\om,y,z,p):\bar{\Lambda}\times\R\times\R^d\times\R\to\R$
be functions that are non-anticipating in $(t,\om)$. 
\begin{assumption}\label{A:SemilinCoeff}
Let $\xi$ and $f$ be bounded from above by $C^\prime_0$.
Let $\xi$ be uniformly continuous
under $\norm{\cdot}_\infty$
 with modulus of continuity
$\rho_0$.
Let $\eta$ and $f$ be
uniformly continuous in $(t,\om)$
under $\dist_\infty$
uniformly in $(y,z,p)$
with modulus of continuity $\rho_0$.
Also let, uniformly in $(t,\om)$,
\begin{align*}
\abs{f_t(\om,y,z,p)-
f_t(\om,y^\prime,z^\prime,p^\prime)}\le 
L_0\left[\abs{y-y^\prime}+
\abs{\sigma(t,\om)^\top (z-z^\prime)}_1+
\abs{p-p^\prime}\right].
\end{align*}
Let $0\le \eta(t,\om,z)\le C^\prime_0(1\wedge \abs{z})$.
\end{assumption}

\begin{remark}
Note that our Lipschitz condition of $f$ in $z$ is
the same as in \cite{CarboneEtAl08_BSDE}.
\end{remark}

To be able to use the comparison principle for backward SDEs
with jumps, we  also need
the following assumption.
\begin{assumption}\label{A:f_increasing}
Let $f$ be nondecreasing in $p$.
\end{assumption}

\subsection{Canonical setup}
We introduce probability measures $\Prob_{s,\om}$,
which will be employed in the rest of
this paper. To this end, 
let $(B,C,\nu)$ be a candidate for a characteristic
triplet of $X$
(see \S III.2.a in \cite{JacodShiryaevBook})
such that 
\begin{align*}
dB_t=b_t\,dt,\quad dC_t=c_t\,dt,\quad
\nu(dt,dz)=K_t(z)\,dt.
\end{align*}
For every $s\in [0,T]$, define (cf.~\S III.2.d~in \cite{JacodShiryaevBook})
\begin{align*}
&p_s B:[s,T]\times\Omega\to \R^d,\,
 (p_sB)_t:=B_t-B_s,\\
 &p_s C:[s,T]\times\Omega\to \mathbb{S}^d,\,
 (p_s C)_t:=C_t-C_s,\\
&p_s \mu:\mathcal{B}([s,T]\times \R^d)\to\R,\,
(p_s\mu)((s,t]\times A):=
\mu(((s,t]\times A).
\end{align*}
Then, by Assumption~\ref{A:diffChar},  for every 
$(s,\om)\in [0,T]\times\Omega$,
the martingale problem for $(p_sB,p_sC,p_s\nu)$
starting at $(s,\om)$ has a unique solution $\Prob_{s,\om}$
(cf.~Theorem~III.2.26 
in \cite{JacodShiryaevBook} for the Markovian case). That is, 
$X$ is an 
$(\F^0_+,\Prob_{s,\om})$-semimartingale on $[s,T]$ with characteristics
$(p_sB,p_sC,p_s\nu)$ and $X.\bfone_{[0,s]}=
\om.\bfone_{[0,s]}$, $\Prob_{s,\om}$-a.s.

We also write $\Mean_{s,\om}$
instead of $\Mean^{\Prob_{s,\om}}$ and  the continuous 
local 
martingal part  of $X$ under $\Prob_{s,\om}$
 on $[s,T]$ is denoted by $X^{c,s,\om}$.

\begin{remark}\label{R:jumpSize}
By Theorem~II.2.34
 in \cite{JacodShiryaevBook},
the canonical representation of $X$ on $[s,T]$
is given by
\begin{align*}
X=X_s+p_s B+X^{c,s,\om}+
z.\bfone_{\{\abs{z}\le C_0^\prime\}}\ast (\mu^X-\nu)+
z.\bfone_{\{\abs{z}> C_0^\prime\}}\ast \mu^X,\quad
\text{$\Prob_{s,\om}$-a.s.}
\end{align*}
Also note that, since $\delta$ is bounded from above by $C_0^\prime$,
the random measure $K$ assigns no mass to
$\{z\in\R^d:\abs{z}>C_0^\prime\}$. Consequently,
the jumps of $X$ on $[s,T]$ are bounded from
above by $C_0^\prime$, $\Prob_{s,\om}$-a.s., 
i.e., we have on $[s,T]$, 
\begin{align*}
X=X_s+p_sB+X^{c,s,\om}+
z\ast (\mu^X-\nu),\quad
\text{$\Prob_{s,\om}$-a.s.}
\end{align*}
and $X$ is a special
$(\Prob_{s,\om},\F^{0}_{+})$-semimartingale
on $[s,T]$.
\end{remark}

We  augment the raw filtration $\F^0$ 
similarly as in the theory of Markov processes
(see, e.g., \cite{RogersWilliamsI}).
To this end, let $\mathcal{N}_{s,\om}$ be the collection
of all $\Prob_{s,\om}$-null sets in $\cF^0_T$
and put,
for every $t\in [0,T]$,
\begin{align*}
\cF_t^{s,\om}:=
\sigma(\cF_{t+}^0,\mathcal{N}_{s,\om}),\quad
\cF_t:=\bigcap_{(s,\om)\in\bar{\Lambda}} 
\cF_t^{s,\om}.
\end{align*}
Now we can define the following filtrations:
\begin{align*}
\F:=\{\cF_t\}_{t\in [0,T]},\quad
\F^{s,\om}:=\{\cF^{s,\om}_t\}_{t\in [0,T]},\quad
(s,\om)\in\bar{\Lambda}.
\end{align*}
Note that $\F$ is right-continuous.

Next, we introduce several classes 
of stopping times.
\begin{definition}
Let $s\in [t,T]$. Given a filtration
$\mathbb{G}=\{\mathcal{G}_t\}_{t\in [0,T]}$ on $\Omega$,
denote by $\cT_s(\mathbb{G})$ the set of all
$\mathbb{G}$-stopping times $\tau$ such that 
$s\le \tau$.
Set $\cT_s:=\cT_s(\F)$.
Let $\om\in\Omega$.
Denote by $\mathcal{H}_s$ (resp.~$\mathcal{H}_{s,\om}$)  the set of 
all $\tau\in\cT_s$ for which there exist some
$d^\prime\in\N$, a right-continuous,
non-anticipating, $\F$-adapted, $d^\prime$-dimensional
process $Y=(Y^i)_{i\le d^\prime}$, and a closed subset $E$
of $\R^{d^\prime}$
such that, for every $\tiom\in\Omega$ 
(resp.~for $\Prob_{s,\om}$-a.e.~$\tiom\in\Omega$),
\begin{align*}
\tau(\tiom)=\inf\{t\ge s: Y_t(\tiom)\in E\}\wedge T.
\end{align*}
\end{definition}

Given a stopping time $\tau$ and
a path $\om\in\Omega$, we often write $(\tau,\om)$
instead of $(\tau(\om),\om)$ if there is no danger
of confusion.

\begin{lemma}\label{L:ShiftingHittingtime}
Fix $(s,\om)\in\Lambda$ and $t\in [s,T)$.
Let $\tau\in\cH_{s,\om}$. If $\tau(\om)>t$ and $X$
coincides with $\om$ on $[0,t]$, $\Prob_{s,\om}$-a.s., then $\tau>t$,
$\Prob_{s,\om}$-a.s.
\end{lemma}
\begin{proof}
Let $Y$ be the corresponding process,
$E$ be the corresponding closed set, and $\Omega^\prime$ the
corresponding subset of $\Omega$ with $\Prob_{s,\om}(\Omega^\prime)=1$
in the definition of $\cH_{s,\om}$ such
that, for every $\tiom\in\Omega^\prime$, 
$\tau(\tiom)=\inf\{t\ge s:Y_t(\tiom)\in E\}\wedge T$
and $\om$ coincides with $\tiom$ on $[0,s]$.
Since $\tau(\om)>t$,
we have $Y_r(\tiom)=Y_r(\om)\in E^c$ for every $r\in [s,t]$. This yields $\tau(\tiom)>t$
because $Y$ is right-continuous and $E$ is 
closed.
\end{proof}

\subsection{Path-dependent stochastic analysis}


First, we introduce a new space of 
continuous functionals. The reason 
is that we want the 
trajectories 
$t\mapsto u(t,\om+x.\bfone_{[t,T]})$ to be 
right-continuous, which,
in general, is not the case if 
 the functional $u$ is only in $C(\bar{\Lambda})$
 as Example~\ref{Ex:C_0} below demonstrates.

\begin{definition}
Let $s\in [0,T]$.
Denote by $C^0(\bar{\Lambda}^s)$ the set of all  $u\in C(\bar{\Lambda}^s)$ such that, for every $x\in\R^d$,
the map $(t,\om)\mapsto u(t,\om+x\bfone_{[t,T]})$ is continuous under $\dist_\infty$.
Denote by $C^0_b(\bar{\Lambda}^s)$ the set of all bounded 
functionals in $C^0(\bar{\Lambda}^s)$ and by $UC^0_b(\bar{\Lambda}^s)$
the set of all uniformly continuous functionals in $C^0_b(\bar{\Lambda}^s)$.
\end{definition}

\begin{example}\label{Ex:C_0}
Consider
$u=u(t,\om):=\sup_{0\le s\le t} \abs{\om_s}$.
Fix $t>0$. Let $\om=-2.\bfone_{[t,T]}$.
Then $u(t,\om+\bfone_{[t,T]})=1$
but $u(t+n^{-1},\om+\bfone_{[t+n^{-1},T]})=2$
for every $n\in\N$.
\end{example}

Next, we give an implicit definition of 
our path-dependent derivatives.
\begin{definition}
Let $(s,\om)\in\Lambda$
and let $\ch\in\cH_{s,\om}$ with $\ch>s$, $\Prob_{s,\om}$-a.s.
Denote by $C_b^{1,2}(\llb s,\ch\rrb)$ 
 the set of all bounded functionals
 $u\in C(\bar{\Lambda}^s)$ for which 
there exist bounded, right-continuous, non-anticipating,
$\F^{s,\om}$-adapted functionals
$\partial_t u:\bar{\Lambda}^s\to\R$,
$\partial_\om u=(\partial_{\om^i} u)_{i\le d}:\bar{\Lambda}^s\to\R^d$,
and $\partial^2_{\om\om} u=(\partial_{\om^i\om^j} u)_{i,j\le d}:
\bar{\Lambda}^s\to\mathbb{S}^d$ such that
$\partial_t u\in C(\llb s,\ch\llb)$,
$\partial_\om u\in C(\llb s,\ch\llb,\R^d)$, 
$\partial^2_{\om\om} u
\in C(\llb s,\ch\llb,\mathbb{S}^d)$, and that,
 for every $\tau\in\cT_s$,
\begin{align*}
&u_{\tau\wedge\ch}=u_s+
\int_s^{\tau\wedge\ch}\partial_t u_t \,dt+
\sum_{i=1}^d\int_s^{\tau\wedge\ch} \partial_{\om^i} u_{t-} 
 \,dX^i_t\\&\quad+\frac{1}{2}
\sum_{i,j=1}^d\int_s^{\tau\wedge\ch} \partial^2_{\om^i\om^j} u_{t-}
\,d\langle X^{i,s,\om,c},X^{j,s,\om,c}\rangle_t
\\ &\quad+\int_s^{\tau\wedge\ch}
 \int_{\R^d}\left[
 u_t(X_{\cdot\wedge t-}+z.\bfone_{[t,T]})-u_{t-}
-\sum_{i=1}^d z^i
\partial_{\om^i} u_{t-}\right]
\mu^X(dt,dz),\,\text{$\Prob_{s,\om}$-a.s.}
\end{align*}
\end{definition}

Given $z\in\R^d$, we sometimes use the 
operator $\nabla^2_z$ defined by
\begin{align*}
\nabla^2_z u(t,\om):=
u(t,\om+z.\bfone_{[t,T]})-u(t,\om)-
\sum_{i=1}^d z^i\partial_{\om^i} u(t,\om).
\end{align*}
\begin{remark}\label{R:L_Ito}
 If $u\in C_b^{1,2}(\llb s,\ch\rrb)$, then, for every 
 $\tau\in\cT_s$,
\begin{align*}
u_{\tau\wedge\ch}=u_s-\int_s^{\tau\wedge\ch}\mathcal{L}u_t\,dt
+\,\text{local martingale part, $\Prob_{s,\om}$-a.s.}
\end{align*}
\end{remark}

\section{Viscosity solutions and main results} \label{SS:ViscSol}
In this section, we introduce the notion of
viscosity solutions for equations of the
 form \eqref{E:PPIDE}. A minimal requirement
 for those solutions is consistency with
 classical solutions. They are defined
 as follows:
\begin{definition}
If $u\in C_b^0(\bar{\Lambda})\cap C_b^{1,2}(\Lambda)$
and
\begin{align*}
\cL u-f(\cdot,u,\partial_\om u,
\mathcal{I} u)\le\,\text{(resp.~$\ge$, $=$) $0$} &&\text{in $\Lambda$},
\end{align*}
 then
$u$ is a   \textit{classical subsolution} (resp.~\textit{classical supersolution},
 \textit{classical solution}) of \eqref{E:PPIDE}.
\end{definition}

To state the actual definition of viscosity 
solutions, we need first to
introduce two nonlinear expectations and spaces
of test functionals.

Fix $(s,\om)\in\Lambda$ and $L\ge0$.
Given a process 
$H\in L^2_{\mathrm{loc}}(X^{c,s,\om},
\Prob_{s,\om})$
and a random field 
$W\in G_{\mathrm{loc}}(p_s\mu^X,\Prob_{s,\om})$,
denote by $\Gamma^{H,W}$  the solution to 
\begin{align*}
\Gamma=1+(\Gamma_-H)\sint X^{c,s,\om}+
(\Gamma_-W)\ast(\mu^X-\nu)
\end{align*}
on $[s,T]$ with $\Gamma=1$ on $[0,s)$,
$\Prob_{s,\om}$-a.s.

\begin{definition}
Let $L\ge 0$ and let $(s,\om)\in\bar{\Lambda}$.
Denote by $\mathcal{P}^L(s,\om)$ the set of 
all probability measures $\Prob$ on $(\Omega,\cF_T^0)$
for which there exists a process 
$H\in L^2_{\mathrm{loc}}(X^{c,s,\om},
\Prob_{s,\om})$ with $\abs{\sigma^\top H}_\infty\le L$
and a random field $W\in G_{\mathrm{loc}}(p_s\mu^X,\Prob_{s,\om})$
with $0\le W\le L\eta$ such that, for every $A\in\cF_T^0$,
\begin{align*}
\Prob(A)=\int_{A}
\Gamma^{H,W}_T(\tiom)\,d\Prob_{s,\om}(\tiom).
\end{align*}
\end{definition}
Now we can define the following
 nonlinear expectations:
\begin{align*}
\underline{\mathcal{E}}^L_{s,\om}:=
\inf_{\Prob\in\mathcal{P}^L(s,\om)}\Mean^\Prob,
\qquad
\overline{\mathcal{E}}^L_{s,\om}:=
\sup_{\Prob\in\mathcal{P}^L(s,\om)}\Mean^\Prob.
\end{align*}

\begin{definition}
Let  
$u:\bar{\Lambda}\to\R$ be an $\F$-adapted process,
let $L\ge 0$, and let $(s,\om)\in\Lambda$. Denote by
$\underline{\cA}^Lu(s,\omega)$ 
(resp.~$\overline{\cA}^Lu(s,\omega)$) 
the set of all functionals $\varphi\in C_b^0(\bar{\Lambda}^s)$
for which there exists a hitting time 
$\ch\in\mathcal{H}_{s,\om}$ with $\ch>s$,
$\Prob_{s,\om}$-a.s., such that
$\varphi\in C^{1,2}_b(\llb s, \ch\rrb)$ and that
\begin{align*}
0=(\varphi-u)(s,\omega)=\inf_{\tau\in\cT_s} 
\underline{\mathcal{E}}^L_{s,\om} 
\left[(\varphi-u)_{\tau\wedge\ch}\right]\,
\text{(resp.~$=\sup_{\tau\in\cT_s} 
\overline{\mathcal{E}}^L_{s,\om} 
\left[(\varphi-u)_{\tau\wedge\ch}\right] $).}
\end{align*}
\end{definition}

\begin{definition}
Let $u$ be a bounded, right-continuous,
non-anticipating,
$\F$-adapted process that is
 $\Prob_{s,\om}$-quasi-left-continuous
 on $[s,T]$
for every $(s,\om)\in\bar{\Lambda}$.

(i) Given $L\ge 0$, we say $u$ is a \textit{viscosity $L$-sub\-solution}
(resp.~\textit{viscosity $L$-super\-solution})
 of \eqref{E:PPIDE} if,
for every $(t,\om)\in\Lambda$  and every
 $\varphi\in\underline{\cA}^Lu(t,\omega)$
 (resp.~$\overline{\cA}^Lu(t,\omega)$),
\begin{align*}
  \cL \varphi(t,\om)-
 f_t(\om,\varphi(t,\om),\partial_\om 
 \varphi(t,\om),\mathcal{I}\varphi(t,\om))
 \le\,\text{(resp.~$\ge$) } 0.
\end{align*}
(ii) We say $u$ is a \textit{viscosity subsolution}
(resp.~\textit{viscosity supersolution})
of \eqref{E:PPIDE} if it
is a viscosity $L$-subsolution
(resp.~viscosity $L$-supersolution)
of \eqref{E:PPIDE} for some $L\ge 0$.

(iii) We say $u$ is a \textit{viscosity solution} of
\eqref{E:PPIDE} if it is both a viscosity subsolution and 
a viscosity supersolution of \eqref{E:PPIDE}.
\end{definition}

\begin{remark}
If $u\in C(\bar{\Lambda})$, then $u$ is 
$\Prob_{s,\om}$-quasi-left-continuous 
on $[s,T]$
for every $(s,\om)\in\Lambda$.

Indeed, since $X$ is $\Prob_{s,\om}$-quasi-left-continuous on $[s,T]$,
 there exists, by
Proposition~I.2.26 in  \cite{JacodShiryaevBook},
a sequence of totally inaccessible stopping times
exhausting the jumps of $X$. 
By Remark~\ref{R:C(Lambda)}~(iii),
this sequence also exhausts the jumps of $u$.
 Hence, again by
Proposition~I.2.26 
 in \cite{JacodShiryaevBook}, $u$ is
$\Prob_{s,\om}$-quasi-left-continuous.
\end{remark}

\begin{theorem}[Consistency with 
classical solutions]\label{T:Cons}
Let $u\in C^0(\bar{\Lambda})\cap 
C^{1,2}(\Lambda)$. Then $u$ is a classical
subsolution (classical supersolution,
classical solution) 
of \eqref{E:PPIDE} if and only if 
$u$ is a
viscosity subsolution 
(viscosity supersolution, viscosity solution)
of \eqref{E:PPIDE}.
\end{theorem}

Our semilinear path-dependent integro-differential
equation is closely connected to a family of
non-Markovian BSDEs with jumps.
To introduce this family,  
fix first $(s,\om)\in\bar{\Lambda}$.
  Denote by 
  $(Y^{s,\om},Z^{s,\om},U^{s,\om})$
  the unique solution
  to the BSDE
  \begin{align*}
  Y_t^{s,\om}&=
  \xi+\int_t^T f_r\left(X,Y_r^{s,\om},Z_r^{s,\om},
  \int_{\R^d} U_r^{s,\om}(z)\,\eta_r(z)
  \,K_r(dz)\right)\,dr\\
  &\qquad -\int_t^T Z_r^{s,\om}\,dX_r^{c,s,\om}
  -\int_t^T \int_{\R^d} U_r^{s,\om}(z)\,
  (\mu^X-\nu)(dr,dz),\, t\in [s,T],\,
  \text{$\Prob_{s,\om}$-a.s.}
  \end{align*}
Without loss of generality, we assume that 
$Y^{s,\om}$ is
right-continuous and $\F_+^0$-adapted.

\begin{remark}
By Theorem~III.4.29
 in \cite{JacodShiryaevBook}
 every $(\Prob_{s,\om},\F^0_+)$-local martingale
 has the representation property relative to $X$
 (see Definition~III.4.22
  in \cite{JacodShiryaevBook}). 
  Therefore, one can prove well-posedness of
  the BSDE above by standard methods.
  For related results for BSDEs driven
by c\`adl\`ag martingales see
\cite{ElKarouiHuang97}
and \cite{CarboneEtAl08_BSDE}. 
In \cite{Xia00}, which deals
with BSDEs driven by c\`adl\`ag martingales
and random measures,  a special case
of our BSDE is covered. 
Moreover, BSDEs driven by random measures
in a general setting are treated in
\cite{ConfortolaFuhrman13}.
\end{remark}

Next, define a functional $u^0:\bar{\Lambda}\to\R$ by
\begin{align*}
u^0(t,\om):=
\Mean_{t,\om} [Y^{t,\om}_t].
\end{align*}
It will turn out that under additional assumptions $u^0$
is the unique solution to \eqref{E:PPIDE} satisfying $u^0_T=\xi$.

\begin{theorem}[Existence]\label{T:Existence}
If $(B,C,\nu)$ and $\nu$ are deterministic,
then $u^0$ is  a viscosity solution of \eqref{E:PPIDE} and $u^0\in UC_b(\bar{\Lambda})$.
\end{theorem}

\begin{theorem}[Partial comparison I]
\label{T:PartialComp}
Fix $(s,\om)\in\Lambda$. Let $u^1$ be 
a viscosity subsolution of \eqref{E:PPIDE}
on $\Lambda^s$ and let $u^2$ be a 
classical supersolution of \eqref{E:PPIDE} on
$\Lambda^s$. Suppose that $u^1_T\le u^2_T$,
$\Prob_{s,\om}$-a.s. Then $u^1(s,\om)\le 
u^2(s,\om)$.
\end{theorem}

\begin{theorem}[Stability]\label{T:Stab}
For every $\eps>0$, let
$(b^\eps,c^\eps,K^\eps)$ together
with some process $\sigma^\eps$, some
random field $\delta^\eps$, and some 
$\sigma$-finite measure $F^\eps$
satisfy Assumption~\ref{A:diffChar}
in place of $(b,c,K)$ together with 
$\sigma$, $\delta$, and $F$, and denote
the corresponding linear integro-differential 
operator by $\mathcal{L}^\eps$ (see Section~3).
Also, for every $\eps>0$, let
$\eta^\eps=\eta^\eps_t(\om,z)$ and
 $f^\eps=f^\eps_t(\om,y,z,p)$ 
satisfy Assumption~\ref{A:SemilinCoeff}
and Assumption~\ref{A:f_increasing}
in place of $\eta$ and $f$, respectively,
and denote the corresponding integral operator
by $\mathcal{I}^\eps$, where
$(\eta,K)$ is replaced with $(\eta^\eps,K^\eps)$.
 Suppose that
$b^\eps\to b$, $c^\eps\to c$, 
$K^\eps\to K$,
 $\eta^\eps\to\eta$, and
$f^\eps\to f$ uniformly as $\eps\downarrow 0$.
Fix $L>0$. For every $\eps>0$, let
$u^\eps=u^\eps(t,\om)$ be a
 viscosity $L$-supersolution
of \eqref{E:PPIDE} with
$\mathcal{L}$ replaced by
$\mathcal{L}^\eps$ and
 $f$ replaced by $f^\eps$.
Suppose that $u^\eps$ converges to some
functional $u=u(t,\om)$ on $\bar{\Lambda}$
uniformly as $\eps\downarrow 0$. Then
$u$ is a viscosity $L$-supersolution of
\eqref{E:PPIDE}.
\end{theorem}

For the comparison principle, we have to employ the
subsequent set of assumptions.
\begin{assumption}\label{A:xi}
Let $\xi$ be uniformly continuous with respect to the weak $M_1$-topology,
i.e., with respect to the metric $d_p$ defined by
$d_p(\om,\tiom):=\max_{i\le d} d_{M_1}(\om^i,\tiom^i)$,
$\om=(\om^i)_{i\le d}$, $\tiom=(\tiom^i)_{i\le d}\in\Omega$
(see Theorem~12.5.2 in \cite{WhittBook}).
\end{assumption}

\begin{remark}
If $d=1$, then the weak $M_1$-topology coincides with the $M_1$-topology.
For its definition, see 
Appendix~\ref{S:Skorohod}.  For more details,
we  refer the reader to \cite{WhittBook}.
\end{remark}

\begin{assumption}\label{A:L}
The triple $(b,c,K)$ is constant, the random field $\eta$ is
deterministic and does not depend on $z$, and there exist
positive constants $(L_\eps)_{\eps\in (0,1)}$ and
$\check{\nu}\in (0,1)$ such that the following holds:
\begin{enumerate}
\renewcommand{\labelenumi}{(\roman{enumi})}
\item For every $\zeta=(\zeta^i)_{i\le d}\in\R^d$,
\begin{align*}
\check{\nu}\abs{\zeta}^2\le \frac{1}{2}\sum_{i,j=1}^d c^{ij}\zeta^i\zeta^j
\le \check{\nu}^{-1}\abs{\zeta}^2.
\end{align*}
\item For every $\alpha\in (0,1)\cap\mathbb{Q}$, there exists
a positive constant $L_2(\alpha)$ such that
\begin{align*}
\abs{\eta}_{\alpha/2,[0,T]}\le L_2(\alpha).
\end{align*}
\item \label{I3:A:L} For every $\eps\in (0,1)$, there exist nonnegative
$\sigma$-finite measures
$K_{1,\eps}$ and $K_{2,\eps}$ 
on $\mathcal{B}(\R^d)$ such that
\begin{align*}
K(dz)&\le K_{1,\eps}(dz)+K_{2,\eps}(dz),\\
\int_{\R^d} \left(\abs{z}^2+\abs{z}\right)\,K_{1,\eps}(dz)&\le \eps,\\
K_{2,\eps}(\R^d\setminus\{\bfnull\})&\le L_\eps.
\end{align*}
\end{enumerate}
\end{assumption}
\begin{assumption}\label{A:jumps}
Suppose that there exists a constant
$c_0^\prime\in (0,C_0^\prime)$ such
that $K(\{z\in\R^d:\abs{z}<c_0^\prime\})=0$
and that $K(\R^d)<\infty$.
\end{assumption}
\begin{assumption}\label{A:f}
For every $\om\in\Omega$ and every 
$\alpha\in (0,1)\cap\mathbb{Q}$, there exists a positive
constant $L_1(\om,\alpha)$ such that the following holds:
\begin{enumerate}
\renewcommand{\labelenumi}{(\roman{enumi})}
\item For every $(y,z,p)\in\R\times\R^d\times\R$,
\begin{align*}
\abs{f_{\cdot}(\om,y,z,p)}_{\alpha/2,[0,T]}\le 
L_1(\om,\alpha)\cdot\left[\abs{(y,z,p)}+1\right].
\end{align*}
\item For every $(t,p)\in [0,T]\times\R$, we have
$f_t(\om,\cdot,p)\in C^\infty(\R\times\R^d)$.
\end{enumerate}
\end{assumption}
The following sets are also used in the proof of the 
partial comparison principle and the
comparison
principle.
\begin{definition}[The sets $\Pi^t_\infty$ and $\Pi^t_i$]\label{D:TheSets}

Let $t\in[0,T]$ and $i\in\N$.

Denote by $\Pi^t_\infty$ the set of all
$\pi_\infty=(t_0,x_0;t_1,x_1;\ldots;x_\infty)$
such that
\begin{enumerate}
\renewcommand{\labelenumi}{(\roman{enumi})}
\item $t=t_0\le t_1\le \ldots \le T$,
\item $t_i= T$ for all but finitely many $i\in\N_0$,
\item $x_i\in\R^d$ for all $i\in\N_0\cup\{\infty\}$.
\end{enumerate}

Denote by $\Pi^t_i$ the set of all
$\pi_i=(s_0,y_0;\ldots;s_{i-1},y_{i-1})$
such that
\begin{enumerate}
\renewcommand{\labelenumi}{(\roman{enumi})}
\item $t=s_0\le \ldots\le s_{i-1}\le T$,
\item $y_j\in\R^d$ for all $j\in\{0,\ldots,i-1\}$.
\end{enumerate}
\end{definition}

The following assumption is only needed for measurability
issues in the proof of the comparison principle.
\begin{assumption}\label{A:pi}
For every $(s,\om)\in\Lambda$, $i\in\N$, and $\tiom\in\Omega$,
let the functions
\begin{align*}
&\Pi_i\to\R, \pi_i=(s_0,y_0;\ldots;s_{i-1},y_{i-1})\mapsto
\xi\left(
\om.\bfone_{[0,s_0)}+\sum_{j=0}^{i-1} y_j.\bfone_{[s_j,s_{j+1})}+
\tiom.\bfone_{[s_i,T]}
\right),\\
&\Pi_i\to\R, \pi_i\mapsto f_t\left(
\om.\bfone_{[0,s_0)}+\sum_{j=0}^{i-1} y_j.\bfone_{[s_j,s_{j+1})}+
\tiom.\bfone_{[s_i,T]},y,z,p
\right)
\end{align*}
be continuous uniformly in $(t,y,z,p)$.
\end{assumption}

\begin{theorem}[Comparison]\label{T:Comparison}
Suppose that Assumptions~\ref{A:xi},
\ref{A:L},
\ref{A:f}, \ref{A:jumps},
and \ref{A:pi}
are satisfied.
If $u^1$ is a viscosity subsolution
of \eqref{E:PPIDE},
if $u^2$ is a viscosity supersolution
of \eqref{E:PPIDE}, and if
$u^1_T\le u^2_T$, then $u^1\le u^2$.  
\end{theorem}

Theorems~\ref{T:Existence} and \ref{T:Comparison}
immediately yield our final main result.
\begin{theorem}[Well-posedness]
Suppose that Assumptions~\ref{A:xi},
\ref{A:L},
\ref{A:f}, \ref{A:jumps},
and \ref{A:pi}
are satisfied. Then
$u^0$ is the unique viscosity solution
of \eqref{E:PPIDE} with $u^0_T=\xi$.
\end{theorem}

\section{Consistency and Existence}
\begin{proof}[Proof of Theorem~\ref{T:Cons}]
Clearly, if $u$ is a viscosity subsolution
of \eqref{E:PPIDE}, then it is a classical
subsolution of \eqref{E:PPIDE}.

Let us now assume that $u$ is not a viscosity
$L_0$-subsolution of \eqref{E:PPIDE} 
 but a classical
subsolution of \eqref{E:PPIDE}. Then
there exist $(s_0,\om)\in\Lambda$ and
$\varphi\in\underline{\mathcal{A}}^{L_0}
 u(s_0,\om)$
with corresponding hitting time 
$\ch\in\cH_{s_0}$ (see the definition of
$\underline{\mathcal{A}}^{L_0}$) such that
\begin{align*}
c^\prime:=\mathcal{L}\varphi(s_0,\om)-
f_{s_0}(\om,\varphi(s_0,\om),\partial_\om 
\varphi(s_0,\om),\mathcal{I}\varphi(s_0,\om))>0.
\end{align*}
Without loss of generality, $s_0=0$. Put
\begin{align*}
\tau:=\inf\left\{t\ge 0:
\mathcal{L}\varphi_t-
f_t(X,\varphi_t,\partial_\om\varphi_t,
\mathcal{I}\varphi_t)\le \frac{c^\prime}{2}
\right\}\wedge T.
\end{align*}
By right-continuity of the involved process,
$\tau>t$, $\Prob_{0,\om}$-a.s.
Let $H=(H^i)_{i\le d}$
 be a stochastic process with
$\abs{\sigma^\top H}_\infty\le L_0$, let
$W$ be a random field with 
$0\le W\le L_0\eta$, and let $\Gamma=\Gamma^{H,W}$
(see Section~\ref{SS:ViscSol}).
Then integration-by-parts
 (Lemma~\ref{L:IntegrByParts})
 yields
 \begin{equation}\label{E0:Cons}
 \begin{split}
 &\Gamma(\varphi-u)\\&\qquad=\Gamma\Biggl[
 -\mathcal{L}\varphi+\mathcal{L} u+
 \sum_{k,l} H^k\partial_{\om^l} (\varphi-u)c^{kl}
+\int_{\R^d} W\nabla^2_z(\varphi-u)K(dz)
 \Biggr]\sint t\\&\qquad\qquad+\text{ martingale.}
\end{split}
 \end{equation}
Given a strictly positive 
stopping time $\ttau\le\tau\wedge\ch$
such that 
\begin{align*}
\abs{f_t(X,\varphi_t,\partial_\om \varphi_t,
\mathcal{I}\varphi_t)-
f_t(X,u_t,\partial_\om \varphi_t,
\mathcal{I}\varphi_t)}\le \frac{c^\prime}{4}
\end{align*}
on $\llb 0,\ttau\llb$,
taking expectations yields
\begin{equation}\label{E1:Cons}
\begin{split}
&\Mean_{0,\om}[\Gamma_\ttau(\varphi-u)_\ttau]
\\&\qquad\le 
\Mean_{0,\om}\Biggl[
\int_0^\ttau \Gamma_t\Biggl[
-\frac{c^\prime}{2}-
f_t(X,\varphi_t,\partial_\om\varphi_t,
\mathcal{I}\varphi_t)+
f_t(X,u_t,\partial_\om u_t,
\mathcal{I}u_t)\\
&\qquad\qquad\qquad\qquad +\sum_{k,l} H^k_t
\partial_{\om^l}(\varphi-u)_t c^{kl}\\
&\qquad\qquad\qquad\qquad +
\int_{\R^d} W_t\nabla^2_z(\varphi-u)_t\,
K_t(dz)
\Biggr]\,dt
\Biggr]\\
&\qquad\le \Mean_{0,\om}\Biggl[
\int_0^\ttau \Gamma_t\Biggl[
-\frac{c^\prime}{4}-
f_t(X,u_t,\partial_\om\varphi_t,
\mathcal{I}\varphi_t)+
f_t(X,u_t,\partial_\om u_t,
\mathcal{I}u_t)\\
&\qquad\qquad\qquad\qquad +\sum_{k,l} H^k_t
\partial_{\om^l}(\varphi-u)_t c^{kl}\\
&\qquad\qquad\qquad\qquad +
\int_{\R^d} W_t\nabla^2_z(\varphi-u)_t\,
K_t(dz)
\Biggr]\,dt
\Biggr].
\end{split}
\end{equation}
Define a function $\tilde{f}:\bar{\Lambda}\times\R
\times\R^d\times\R\to\R$ by
\begin{align*}
\tilde{f}_t(\tiom,y,z,p):=
f_t(\tiom,y,(\sigma^\top_t(\tiom))^{-1}z,p).
\end{align*}
Then $f_t(\tiom,y,z,p)=
\tilde{f}_t(\tiom,u,\sigma_t^\top(\tiom)z,p)$
and
\begin{align*}
\abs{\tilde{f}_t(\tiom,y,z,p)-
\tilde{f}_t(\tiom,y,z^\prime,p)}\le 
L_0\abs{z-z^\prime}_1.
\end{align*}
Consequently,
\begin{align*}
&-f_t(X,u_t,\partial_\om\varphi_t,
\mathcal{I}\varphi_t)+
f_t(X,u_t,\partial_\om u_t,
\mathcal{I} u_t)\\
&\quad=
\tilde{f}_t(X,u_t,\sigma^\top_t\partial_\om u_t,
\mathcal{I} u_t)-
\tilde{f}_t(X,u_t,
\sigma^\top_t\partial_\om \varphi_t,
\mathcal{I} \varphi_t)\\
&\quad =
[\tilde{f}_t(X,u_t,\sigma^\top_t\partial_\om u_t,
\mathcal{I} u_t)\\ &\quad\quad\quad\quad-
\tilde{f}_t(X,u_t,
(\sigma^\top_t\partial_\om u_t)^1,
\ldots,
(\sigma^\top_t\partial_\om u_t)^{d-1},
(\sigma^\top_t\partial_\om \varphi_t)^d
),
\mathcal{I} u_t)]\\
&\quad\quad +[\tilde{f}_t(X,u_t,
(\sigma^\top_t\partial_\om u_t)^1,
\ldots,
(\sigma^\top_t\partial_\om u_t)^{d-1},
(\sigma^\top_t\partial_\om \varphi_t)^d
),
\mathcal{I} u_t)\\
&\quad\quad\quad\quad-
\tilde{f}_t(X,u_t,
(\sigma^\top_t\partial_\om u_t)^1,
\ldots,
(\sigma^\top_t\partial_\om u_t)^{d-2},
(\sigma^\top_t\partial_\om \varphi_t)^{d-1},
(\sigma^\top_t\partial_\om \varphi_t)^d
),
\mathcal{I} u_t)]\\&\quad\quad +\cdots\\
&\quad\quad +
\tilde{f}_t(X,u_t,
(\sigma_t^\top \partial_\om u_t)^1,
(\sigma_t^\top \partial_\om \varphi_t)^2,
\ldots,
(\sigma_t^\top \partial_\om\varphi_t)^d,
\mathcal{I} u_t)\\
&\quad\quad\quad\quad -
\tilde{f}_t(X,u_t,\sigma_t^\top\partial_\om 
\varphi_t,\mathcal{I}u_t)]\\
&\quad\quad +
[\tilde{f}_t(X,u_t,\sigma_t^\top\partial_\om 
\varphi_t,\mathcal{I}u_t)-
\tilde{f}_t(X,u_t,\sigma_t^\top\partial_\om 
\varphi_t,\mathcal{I}\varphi_t)]\\
&\quad =
\kappa^d_t(\sigma_t^\top
\partial_\om(u-\varphi)_t)^d+\cdots
\kappa^1_t(\sigma_t^\top
\partial_\om(u-\varphi)_t)^1\\ &\quad\quad+
\int_{\R^d}\lambda_t\eta_t(z)
\nabla_z^2(u-\varphi)_t\,K_t(dz),
\end{align*}
where, for every $i\in\{1,\ldots,d\}$,
\begin{align*}
\kappa^i_t&:=[(\sigma_t^\top\partial_\om (u-\varphi)_t)^i]^{-1}.\bfone_{
\{(\sigma_t^\top\partial_\om (u-\varphi)_t)^i
\neq 0\}
}\\ &\qquad\cdot
[\tilde{f}_t(X,u_t,
(\sigma_t^\top\partial_\om u_t)^1,\ldots,
(\sigma_t^\top\partial_\om u_t)^{i},
(\sigma_t^\top\partial_\om \varphi_t)^{i+1},\ldots,
(\sigma_t^\top\partial_\om \varphi_t)^{d})\\
&\qquad\qquad
-\tilde{f}_t(X,u_t,
(\sigma_t^\top\partial_\om u_t)^1,\ldots,
(\sigma_t^\top\partial_\om u_t)^{i-1},
(\sigma_t^\top\partial_\om \varphi_t)^{i},\ldots,
(\sigma_t^\top\partial_\om 
\varphi_t)^{d})]
\end{align*}
and 
\begin{align*}
\lambda_t&:=
[\mathcal{I}(u_t-\varphi_t)]^{-1}.\bfone_{
\{
\mathcal{I}(u_t-\varphi_t)\neq 0
\}
}\\ &\qquad \cdot
[\tilde{f}_t(X,u_t,\sigma_t^\top\partial_\om 
\varphi_t,\mathcal{I}u_t)-
\tilde{f}_t(X,u_t,\sigma_t^\top\partial_\om 
\varphi_t,\mathcal{I}\varphi_t)].
\end{align*}
Note that, by Assumption~\ref{A:SemilinCoeff},
 we have $\abs{\kappa^i}\le L_0$
and that, by 
Assumption~\ref{A:SemilinCoeff}
and Assumption~\ref{A:f_increasing},
 we have $0\le \lambda\le L_0$.
Our goal now is to establish
$H^\top c\, \partial_\om(\varphi-u)=
\kappa^\top [\sigma^\top \partial_\om 
(\varphi-u)]$, which, by putting
$\tilde{H}=\sigma^\top H$ and
$\tilde{Z}=\sigma^\top 
\partial_\om (\varphi-u)$, is equivalent
to $\tilde{H}^\top \tilde{Z}=
\kappa^\top\tilde{Z}$. 
Hence, if $H$ is given by $H=(\sigma^\top)^{-1}
\tilde{H}$
in the case $d>1$ and
by $H=\sigma^{-1}\tilde{H}.\bfone_{
\{\sigma\neq 0\}
}$ in the case $d=1$, where $\tilde{H}^i=\kappa^i$,
and if $W$ is given by $W_t(z)=\lambda_t
\eta_t(z)$, then
$\abs{\sigma^\top H}_\infty=
\abs{\kappa}_\infty\le L_0$,
$0\le W\le L_0\eta$, and,
provided $\ttau$ is sufficiently small,
which is possible because $\Gamma$ is 
right-continuous, we have, by \eqref{E1:Cons},
\begin{align*}
\Mean_{0,\om}[\Gamma_T(\varphi-u)_\ttau]=
\Mean_{0,\om}[\Gamma_\ttau(\varphi-u)_\ttau]&\le 
\Mean_{0,\om}\Biggl[
\int_0^\ttau \Gamma_t\Biggl[
-\frac{c^\prime}{4}
\Biggr]\,dt
\Biggr]\le 
\Mean_{0,\om}\left[-
\frac{\ttau c^\prime}{8}
\right]<0,
\end{align*}
i.e., $\underline{\mathcal{E}}_{0,\om}^{L_0}
[(\varphi-u)_\ttau]<0$,
which contradicts
$\varphi\in\underline{\mathcal{A}}^{L_0} u
(0,\om)$. Thus $u$ is a 
viscosity $L_0$-subsolution.

Similarly, one can show the corresponding statement
for supersolutions.
\end{proof}

Next, we prove regularity of $u^0$.
To this end, we need the following result.

\begin{lemma}\label{L:BSDE-ChangeMeasure}
Suppose that $(B,C,\nu)$ and $\eta$ are deterministic.
Fix $(s,\om)\in\bar{\Lambda}$.
Define a process $\tilde{Y}$ on $[s,T]$ by
\begin{align*}
\tilde{Y}_t:=Y^{s,\om}_t
(\om.\bfone_{[0,s)}+(X+\om_s).\bfone_{[s,T]}).
\end{align*}
Then there exist processes $\tilde{Z}$ and
$\tilde{U}(z)$ in the appropriate spaces such that
$(\tilde{Y},\tilde{Z},\tilde{U})$ is the solution to the BSDE
\begin{equation}\label{E1:BSDE-ChangeMeasure}
\begin{split}
&\tilde{Y}_t=
\xi(\om.\bfone_{[0,s)}+(X+\om_s).\bfone_{[s,T]})\\
&\qquad +\int_t^T f_r(
\om.\bfone_{[0,s)}+(X+\om_s).\bfone_{[s,T]},
\tilde{Y}_r,\tilde{Z}_r,
\int_{\R^d} \tilde{U}_r(z)\,\eta_r(z)\,K_r(dz))\,dr\\
&\qquad -\int_t^T \tilde{Z}_r\,dX_r^{c,s,\bfnull}
-\int_t^T\int_{\R^d} \tilde{U}_r(z)\,(\mu^X-\nu)(dr,dz),\,
t\in[s,T],\,\text{$\Prob_{s,\bfnull}$-a.s.}
\end{split}
\end{equation}
\end{lemma}
\begin{proof}
Put
\begin{align*}
Z^0&:=Z^{s,\om}(\om.\bfone_{[0,s)}+(X+\om_s).\bfone_{[s,T]}),\\
U^0&:=U^{s,\om}(\om.\bfone_{[0,s)}+(X+\om_s).\bfone_{[s,T]}).
\end{align*}
Since  $(B,C,\nu)$ and $\eta$ are deterministic,
the process $M$ on $[s,T]$ defined by
\begin{align*}
M_t&:=\tilde{Y}_t-\tilde{Y}_s-
\int_s^t f_r(\om.\bfone_{[0,s)}+(X+\om_s).\bfone_{[s,T]},
\tilde{Y}_r,Z^0_r,\int_{\R^d}U_r^0(z)\,\eta_r(z)\,K_r(dz))
\end{align*}
is an $(\F^{s,\bfnull},\Prob_{s,\bfnull})$-martingale. Thus, we can,
for every $n\in\N_0$, define a pair $(Z^{n+1},U^{n+1})$
inductively by
\begin{align*}
&\tilde{Y}_t=
\xi(\om.\bfone_{[0,s)}+(X+\om_s).\bfone_{[s,T]})\\
&\qquad +\int_t^T f_r(
\om.\bfone_{[0,s)}+(X+\om_s).\bfone_{[s,T]},
\tilde{Y}_r,Z^n_r,
\int_{\R^d} U^n_r(z)\,\eta_r(z)\,K_r(dz))\,dr\\
&\qquad -\int_t^T Z^{n+1}_r\,dX_r^{c,s,\bfnull}
-\int_t^T\int_{\R^d} U^{n+1}_r(z)\,(\mu^X-\nu)(dr,dz),\,
t\in[s,T],\,\text{$\Prob_{s,\bfnull}$-a.s.}
\end{align*}
Since $(Z^n,U^n)$ converges to some limit $(\tilde{Z},\tilde{U})$,
the triple $(\tilde{Y},\tilde{Z},\tilde{U})$ is a solution
to  \eqref{E1:BSDE-ChangeMeasure}.
\end{proof}

\begin{proposition}\label{P:Regularity}
If $(B,C,\nu)$ and $\nu$
are deterministic, 
then $u^0\in UC_b(\bar{\Lambda})$.
\end{proposition}
\begin{proof}
Let $\rho_1$ be an increasing modulus of continuity of $\xi$ and of 
$(t,\tiom)\mapsto f_t(\tiom,y,z,p)$, uniformly in $t$, $y$, $z$, and $p$,
with upper bound $\norm{\rho_1}_\infty>0$.
Let $(s,\om)$, $(s^\prime,\om^\prime)\in\bar{\Lambda}$ with $s\le s^\prime$.
Then
\begin{equation*}\label{E1:Reg}
\begin{split}
\abs{u^0(s,\om)-u^0(s^\prime,\om^\prime)}&\le
\abs{u^0(s,\om)-u^0(s^\prime,\om_{\cdot\wedge s})}+
\abs{u^0(s^\prime,\om_{\cdot\wedge s})-u^0(s^\prime,\om^\prime)}\\ &=:A_1+A_2.
\end{split}
\end{equation*}

Let us start with estimating $A_2$. To this end, put
\begin{align*}
\tilde{Y}^1_t&:=Y^{s^\prime,\om_{\cdot\wedge s}}_t
(\om_{\cdot\wedge s}.\bfone_{[0,s^\prime)}+(X+\om_s).\bfone_{[s^\prime,T]}),\\
\tilde{\xi}^1&:=\xi
(\om_{\cdot\wedge s}.\bfone_{[0,s^\prime)}+(X+\om_s).\bfone_{[s^\prime,T]}),\\
\tilde{f}^1_t(X,y,z,p)&:=f_t
(\om_{\cdot\wedge s}.\bfone_{[0,s^\prime)}+(X+\om_s).\bfone_{[s^\prime,T]},y,z,p),\\
\end{align*}
and
\begin{align*}
\tilde{Y}^2_t&:=Y^{s^\prime,\om^\prime}_t
(\om^\prime_{\cdot\wedge s^\prime}.\bfone_{[0,s^\prime)}
+(X+\om^\prime_{s^\prime}).\bfone_{[s^\prime,T]}),\\
\tilde{\xi}^2&:=\xi
(\om^\prime_{\cdot\wedge s^\prime}.\bfone_{[0,s^\prime)}
+(X+\om^\prime_{s^\prime}).\bfone_{[s^\prime,T]}),\\
\tilde{f}^2_t(X,y,z,p)&:=f_t
(\om^\prime_{\cdot\wedge s^\prime}.\bfone_{[0,s^\prime)}
+(X+\om^\prime_{s^\prime}).\bfone_{[s^\prime,T]},y,z,p).
\end{align*}
Since $(B,C,\nu)$ and $\eta$ are deterministic, there exists, 
by Lemma~\ref{L:BSDE-ChangeMeasure}, for every $i\in\{1,2\}$,
a pair $(\tilde{Z}^i,\tilde{U}^i)$ such that the triple $(\tilde{Y}^i,\tilde{Z}^i,\tilde{U}^i)$
is the solution to the BSDE
\begin{align*}
\tilde{Y}^i_t&=\tilde{\xi}^i
+\int_t^T \tilde{f}^i_r(X,
\tilde{Y}^i_r,\tilde{Z}^i_r,
\int_{\R^d} \tilde{U}^i_r(z)\,\eta_r(z)\,K_r(dz))\,dr\\
&\qquad -\int_t^T \tilde{Z}^i_r\,dX_r^{c,s^\prime,\bfnull}
-\int_t^T\int_{\R^d} \tilde{U}^i_r(z)\,(\mu^X-\nu)(dr,dz),\,
t\in[s^\prime,T],\,\text{$\Prob_{s^\prime,\bfnull}$-a.s.}
\end{align*}
Therefore and using again the fact  that $(B,C,\nu)$ is
deterministic, we have 
\begin{equation}\label{E2:Reg}
\begin{split}
A_2&=\abs{\Mean_{s^\prime,\om_{\cdot\wedge s}}
[Y^{s^\prime,\om_{\cdot\wedge s}}_{s^\prime}]-
\Mean_{s^\prime,\om^\prime}
[Y^{s^\prime,\om^\prime}_{s^\prime}]}
=\abs{\Mean_{s^\prime,\bfnull}[
\tilde{Y}^1_{s^\prime}-\tilde{Y}^2_{s^\prime}
]}.
\end{split}
\end{equation}
Now, note that, for every $t\in [s^\prime,T]$,
\begin{align*}
&\tilde{f}_t^1(X,\tilde{Y}^1_t,\tilde{Z}^1_t,
\int_{\R^d}\tilde{U}^1_t(z)\,\eta_t(z)\,K_t(dz))-
\tilde{f}_t^2(X,\tilde{Y}^2_t,\tilde{Z}^2_t,
\int_{\R^d}\tilde{U}^2_t(z)\,\eta_t(z)\,K_t(dz))\\
&\qquad = \gamma_t[\tilde{Y}_t^1-\tilde{Y}_t^2]+
\sum_{j=1}^d \kappa^j_t\,[\sigma_t^\top
(\tilde{Z}_t^1-\tilde{Z}_t^2)
]^j+
\int_{\R^d}\lambda_t\,\eta_t(z)\,
(\tilde{U}^1_t(z)-\tilde{U}^2_t(z))\,K_t(dz)\\
&\qquad\qquad +
\tilde{f}_t^1(X,\tilde{Y}^2_t,\tilde{Z}^2_t,
\int_{\R^d}\tilde{U}^2_t(z)\,\eta_t(z)\,K_t(dz))\\
&\qquad\qquad\qquad\qquad-
\tilde{f}_t^2(X,\tilde{Y}^2_t,\tilde{Z}^2_t,
\int_{\R^d}\tilde{U}^2_t(z)\,\eta_t(z)\,K_t(dz)),
\end{align*}
where
\begin{align*}
&\gamma_t:=[\tilde{Y}_t^1-\tilde{Y}_t^2]^{-1}.\bfone_{
\{\tilde{Y}_t^1-\tilde{Y}_t^2\neq 0\}
}\\
&\quad \cdot[
\tilde{f}_t^1(X,\tilde{Y}^1_t,\tilde{Z}^1_t,
\int_{\R^d}\tilde{U}^1_t(z)\,\eta_t(z)\,K_t(dz))-
\tilde{f}_t^1(X,\tilde{Y}^2_t,\tilde{Z}^1_t,
\int_{\R^d}\tilde{U}^1_t(z)\,\eta_t(z)\,K_t(dz))
],
\end{align*}
and the processes $\kappa^j$, $j=1$, $\ldots$, $d$, and
$\lambda$ are defined similarly (with the obvious changes)
as in the proof of Theorem~\ref{T:Cons}. Also, as in said proof,
define a $d$-dimensional process $H$ by
$H^j:=[(\sigma^\top)^{-1}\kappa]^j$ in the case $d>1$
and by $H:=\sigma^{-1}\kappa$ in the case $d=1$, define
a random field $W$ by $W_t(z):=\lambda_t\,\eta_t(z)$, and consider the solution
$\tilde{\Gamma}$ of 
\begin{align*}
\tilde{\Gamma}=1+(\tilde{\Gamma}_-\gamma)\sint t+
(\tilde{\Gamma}_-H)\sint X^{c,s^\prime,\bfnull}+
(\tilde{\Gamma}_-W)\ast(\mu^X-\nu)
\end{align*}
on $[s^\prime,T]$ with $\tilde{\Gamma}=1$ on $[0,s^\prime)$,
$\Prob_{s^\prime,\bfnull}$-a.s.
Integration-by-parts (Lemma~\ref{L:IntegrByParts}) yields
\begin{align*}
\tilde{\Gamma}(\tilde{Y}^1-\tilde{Y}^2)&=
(\tilde{Y}^1-\tilde{Y}^2)_{s^\prime}-\tilde{\Gamma}
[\tilde{f}^1(X,\tilde{Y}^2,\tilde{Z}^2,
\int_{\R^d}\tilde{U}^2(z)\,\eta(z)\,K(dz))\\ &\qquad\qquad\qquad\qquad\qquad-
\tilde{f}^2(X,\tilde{Y}^2,\tilde{Z}^2,
\int_{\R^d}\tilde{U}^2(z)\,\eta(z)\,K(dz))]
\sint t \\ &\qquad\qquad+\text{martingale, $\Prob_{s^\prime,\bfnull}$-a.s.}
\end{align*}
Since $\xi$, $f$, and $\gamma$ are bounded, we get, together with \eqref{E2:Reg}, 
\begin{equation*}\label{E3:Reg}
\begin{split}
A_2&\le e^{(T-s^\prime)L_0}
\Mean_{s^\prime,\bfnull}\left[\abs{\tilde{\xi}^1-\tilde{\xi}^2}\right]
+ \int_{s^\prime}^T e^{(t-s^\prime)L_0}
\rho_1((t,\om_{\cdot\wedge s}),(t,\om^\prime_{\cdot\wedge s^\prime}))\,dt\\
&\le C^\prime \rho_1(\norm{\om_{\cdot\wedge s}-
\om^\prime_{\cdot\wedge s^\prime}}_\infty),
\end{split}
\end{equation*}
where $C^\prime$ does not depend on $(s,\om)$ and $(s^\prime,\om^\prime)$.

To deal with $A_1$, let $\eps>0$. 
Our goal is to find a $\delta^\prime>0$ such that
\begin{align*}
\dist_\infty((s,\om),(s^\prime,\om^\prime))<\delta^\prime
\end{align*}
implies $A_1<\eps$.
In order to estimate $A_1$, put
for every $\tiom\in\Omega$, 
\begin{align*}
\tilde{Y}^{3,\tiom}_t&:=
Y^{s,\om}_t(
\tiom.\bfone_{[0,s^\prime)}+(X+\tiom_{s^\prime}).\bfone_{[s^\prime,T]}
),\\
\tilde{\xi}^{3,\tiom}&:=\xi(
\tiom.\bfone_{[0,s^\prime)}+(X+\tiom_{s^\prime}).\bfone_{[s^\prime,T]}
),\\
\tilde{f}^{3,\tiom}_t(X,y,z,p)&:=f_t(
\tiom.\bfone_{[0,s^\prime)}+(X+\tiom_{s^\prime}).\bfone_{[s^\prime,T]},
y,z,p).
\end{align*}
Since $(B,C,\nu)$ is deterministic and since,
for $\Prob_{s,\om}$-a.e.~$\tiom\in\Omega$,
\begin{align*}
Y^{s,\om}_{s^\prime}(\tiom)=
\Mean_{s,\om}[Y^{s,\om}_{s^\prime}\vert\cF^0_{s^\prime+}](\tiom)=
\Mean_{s^\prime,\tiom}[Y^{s,\om}_{s^\prime}],
\end{align*}
we have
\begin{equation}\label{E4:Reg}
\begin{split}
A_1&=\abs{
\Mean_{s,\om}[Y^{s,\om}_s]-
\Mean_{s^\prime,\om_{\cdot\wedge s}}
[Y^{s^\prime,\om_{\cdot\wedge s}}_{s^\prime}]
}\\
&\le 
\abs{
\Mean_{s,\om}\left[
\int_s^{s^\prime} f_t\left(
X,Y^{s,\om}_t,Z^{s,\om}_t,
\int_{\R^d}U^{s,\om}_t(z)\,\eta_t(z)\,K_t(dz)
\right)\,dt
\right]
}\\
&\qquad +
\abs{
\Mean_{s,\om}[Y^{s,\om}_{s^\prime}]-
\Mean_{s^\prime,\om_{\cdot\wedge s}}
[Y^{s^\prime,\om_{\cdot\wedge s}}_{s^\prime}]
}\\
&\le (s^\prime-s)C_0^\prime+
\abs{
\int_\Omega\Mean_{s^\prime,\tiom}[Y^{s,\om}_{s^\prime}]-
\Mean_{s^\prime,\om_{\cdot\wedge s}}
[Y^{s^\prime,\om_{\cdot\wedge s}}_{s^\prime}]\,\Prob_{s,\om}(d\tiom)
}\\
&= (s^\prime-s)C_0^\prime+
\abs{
\int_\Omega\Mean_{s^\prime,\bfnull}[\tilde{Y}^{3,\tiom}_{s^\prime}-
\tilde{Y}^{1}_{s^\prime}]\,\Prob_{s,\om}(d\tiom)
}.
\end{split}
\end{equation}
Similarly, as we estimated $A_2$ in \eqref{E3:Reg}, we get,
for every $\tiom\in\Omega$,
\begin{align}\label{E5:Reg}
\abs{
\Mean_{s^\prime,\bfnull}[\tilde{Y}^{3,\tiom}_{s^\prime}-
\tilde{Y}^{1}_{s^\prime}]
}\le C^\prime\rho_1(\norm{
\tiom_{\cdot\wedge s^\prime}-\om_{\cdot\wedge s}
}_\infty),
\end{align}
where $C^\prime>0$ does not depend on $s$, $s^\prime$,
$\om$, and $\tiom$. 
Note that, since
$\rho_1$ is continuous at $0$, there exists a 
$\delta^{\prime\prime}$ such that 
$C^\prime \rho_1(\delta^{\prime\prime})<\eps/2$.
Thus, by \eqref{E4:Reg} together with \eqref{E5:Reg},
\begin{equation}\label{E6:Reg}
\begin{split}
A_1&\le (s^\prime-s)C^\prime_0+
C^\prime\Mean_{s,\om}[
\rho_1(\norm{X_{\cdot\wedge s^\prime}-\om_{\cdot\wedge s}}_\infty
]\\
&= (s^\prime-s)C^\prime_0+
C^\prime\Mean_{s,\om}[
\rho_1(\sup_{t\in [s,s^\prime]}\abs{X_t-X_s})\\
&\qquad\qquad\cdot
\bfone_{
\{\sup_{t\in [s,s^\prime]}\abs{X_t-X_s}<\delta^{\prime\prime}\}
}+\bfone_{
\{\sup_{t\in [s,s^\prime]}\abs{X_t-X_s}\ge\delta^{\prime\prime}\}
}
]\\
&\le  (s^\prime-s)C^\prime_0+\frac{\eps}{2}+C^\prime
\norm{\rho_1}_\infty\,
\Prob_{s,\om}\left(
\sup_{t\in [s,s^\prime]}\abs{X_t-X_s}\ge\delta^{\prime\prime}
\right).
\end{split}
\end{equation}
Recall that on $[s,T]$,
\begin{align*}
X=X_s+p_s B+M,\qquad\text{$\Prob_{s,\om}$-a.s.,}
\end{align*}
where $M:=X^{c,s,\om}+z\ast(\mu^X-\nu)$ is
a $(\Prob_{s,\om},\F^0_+)$-martingale on $[s,T]$.
Without loss of generality, let
\begin{align}\label{E7:Reg}
\delta^\prime C_0^\prime<\frac{\delta^{\prime\prime}}{2}\qquad
\text{and}\qquad s^\prime-s<\delta^\prime.
\end{align}
Thus, since 
\begin{align*}
\sup_{t\in [s,s^\prime]}\abs{X_t-X_s}\ge\delta^{\prime\prime}
\end{align*}
implies 
\begin{align*}
\sup_{t\in [s,s^\prime]}\abs{p_s B_t}+
\sup_{t\in [s,s^\prime]}\abs{M_t}
\ge\delta^{\prime\prime}
\end{align*}
but, by \eqref{E7:Reg},
\begin{align*}
\sup_{t\in [s,s^\prime]}\abs{p_s B_t}\le 
\int_s^{s\prime}\abs{b_t}\,dt\le (s^\prime-s)C_0^\prime\le 
\frac{\delta^{\prime\prime}}{2},
\end{align*}
we have, by Doob's inequality and It\^o's lemma,
\begin{align*}
&\Prob_{s,\om}\left(
\sup_{t\in [s,s^\prime]}\abs{X_t-X_s}\ge\delta^{\prime\prime}
\right)\\&\qquad\le
\Prob_{s,\om}\left(
\sup_{t\in [s,s^\prime]}\abs{M_t}\ge\frac{\delta^{\prime\prime}}{2}
\right)\\&\qquad\le
\Prob_{s,\om}
\left(\sup_{t\in [s,s^\prime]}\abs{M_t}^2\ge
\frac{\abs{\delta^{\prime\prime}}^2}{4}
\right)\\&\qquad\le
\frac{4}{\abs{\delta^{\prime\prime}}^2}
\Mean_{s,\om}\left[\abs{M_{s^\prime}}^2\right]\\ &\qquad\le
\frac{4}{\abs{\delta^{\prime\prime}}^2}
\Mean_{s,\om}\left[ \sum_{i\le d} \int_s^{s^\prime} c_t^{ii}\,dt
+\int_s^{s^\prime}\int_{\R^d} 
\left(\abs{z}^2\wedge C_0^\prime\right) \,K_t(dz)\,dt
\right]\\ &\qquad\le
\frac{4}{\abs{\delta^{\prime\prime}}^2}\cdot
(s^\prime-s) (dC_0^\prime+C_0^{\prime\prime}).
\end{align*}
Together with \eqref{E6:Reg}, we get
\begin{align*}
A_1\le (s^\prime-s)C^{\prime\prime}+\frac{\eps}{2}
\end{align*}
for some constant $C^{\prime\prime}>0$ that does not
depend on $s$, $s^\prime$,  $\omega$, and $\om^\prime$ provided that
\eqref{E7:Reg} holds. I.e., if
\begin{align*}
\dist_\infty((s,\om),(s^\prime,\om^\prime))<\frac{\eps}
{2C^{\prime\prime}}\wedge \frac{\delta^{\prime\prime}}
{2C_0^\prime},
\end{align*}
then $A_1<\eps$.
\end{proof}

\begin{lemma}\label{L:BSDE-DPP}
Fix $(s,\om)\in\bar{\Lambda}$ and $t\in [s,T]$.
For $\Prob_{s,\om}$-a.e.~$\tiom\in\Omega$,
the processes $Y^{s,\om}$ and $Y^{t,\tiom}$
are $\Prob_{t,\tiom}$-indistinguishable on $[t,T]$.
\end{lemma}
\begin{proof}
Let $\Omega^\prime$ be the set of all
$\om^\prime\in\Omega$ such that the process
$M$ on $[s,T]$ defined by
\begin{align*}
M_r:=Y_r^{s,\om}-Y_t^{s,\om}-
\int_t^r f_\theta(X,Y^{s,\om}_\theta,
Z^{s,\om}_\theta,
\int_{\R^d}
 U^{s,\om}_\theta(z)\eta_\theta(z)K_\theta(dz))
\,d\theta
\end{align*}
is an $(\F^{t,\om^\prime},
\Prob_{t,\om^{\prime}})$-martingale.
By Proposition~\ref{L:shiftMG_augment},
$\Prob_{s,\om}(\Omega^\prime)=1$. Now, let 
$\tiom\in\Omega^\prime$. 
Put $(Z^0,U^0):=(Z^{s,\om},U^{s,\om})$.
Since
$M$ is an 
$(\F^{t,\tiom},\Prob_{t,\tiom})$-martingale,
we can, for every $n\in\N_0$, define $(Z^{n+1},U^{n+1})$
inductively by
\begin{align*}
Y_r^{s,\om}&=\xi+\int_r^T
f_\theta(X,Y_\theta^{s,\om},Z^n_\theta,
\int_{\R^d}
 U^n_\theta(z)\eta_\theta(z)\,K_\theta(dz))
\,d\theta
-\int_r^T Z^{n+1}_\theta\,dX^{c,t,\tiom}_\theta
\\ &\qquad
-\int_r^T\int_{\R^d} U^{n+1}_\theta(z)
\,(\mu^X-\nu)(d\theta,dz),
\quad r\in [t,T],\quad
\text{$\Prob_{t,\tiom}$-a.s.}
\end{align*}
Note that $(Z^n,U^n)$ converges 
to some limit $(Z,U)$ and that $(Y^{s,\om},Z,U)$
solves 
\begin{align*}
Y_r^{s,\om}&=\xi+\int_r^T
f_\theta(X,Y_\theta^{s,\om},Z_\theta,
\int_{\R^d} U_\theta(z)\eta_\theta(z)
\,K_\theta(dz))\,d\theta
-\int_r^T Z_\theta\,dX^{c,t,\tiom}_\theta
\\ &\qquad
-\int_r^T \int_{\R^d}
 U_\theta(z)\,(\mu^X-\nu)(d\theta,dz),
\quad r\in [t,T],\quad
\text{$\Prob_{t,\tiom}$-a.s.}
\end{align*}
Uniqueness for BSDEs concludes the proof.
\end{proof}
\begin{remark}\label{R:u0}
Fix $(s,\om)\in\bar{\Lambda}$. By 
Lemma~\ref{L:BSDE-DPP} and by 
Proposition~\ref{C:SMPwide},  
for every $t\in [s,T]$, and for
 $\Prob_{s,\om}$-a.e.~$\tiom\in\Omega$,
 \begin{align*}
 u^0(t,\tiom)&=
 \Mean_{t,\tiom} [Y^{s,\om}_t]=
 \Mean_{s,\om} [Y_t^{s,\om}\vert\cF^0_{t+}]
 (\tiom=Y_t^{s,\om}(\tiom).
\end{align*}
If $u^0\in C(\bar{\Lambda})$, then
 $u^0$ and $Y^{s,\om}$
 are 
$\Prob_{s,\om}$-in\-dis\-tin\-guish\-a\-ble
on $[s,T]$
because both processes
are right-continuous. 
\end{remark}
\begin{proof}[Proof of 
Theorem~\ref{T:Existence}]
By Proposition~\ref{P:Regularity}, $u^0\in C_b(\bar{\Lambda})$. Thus $u^0$ is bounded,
right-continuous, non-anticipating, 
and $\Prob_{s,\om}$-quasi-left-continuous
for every $(s,\om)\in\bar{\Lambda}$.
Keeping Remark~\ref{R:u0} in mind, one
can easily show that $u^0$ is a viscosity subsolution. To do so,
follow the lines of the corresponding part
of the proof of Theorem~\ref{T:Cons}
and replace in \eqref{E0:Cons}
\begin{align*}
&\mathcal{L} u\quad\text{by}\quad
f\left(
X,u,Z^{0,\om},\int U^{0,\om}(z)\eta(z)\,K(dz)
\right)
\intertext{and everywhere}
&\partial_\om u\quad\text{by}\quad Z^{0,\om},\quad
\mathcal{I} u\quad\text{by}\quad 
\int U^{0,\om}(z)\eta(z)\,K(dz),\quad\text{and}
\quad
\nabla^2_z u  
\quad\text{by}\quad
U^{0,\om}(z).
\end{align*}
Similarly, one can show that $u^0$ is a
viscosity supersolution.
\end{proof}
\section{Partial Comparison 
and Stability}\label{SS:PartCom}
Before we begin to prove the partial comparison
principle itself, we need to establish some auxiliary
results about BSDEs, reflected BSDEs (RBSDEs), and
optimal stopping.
\subsection{BSDEs with jumps and nonlinear expectations}
Given $(s,\om)\in\bar{\Lambda}$,
$L>0$, $\tau\in\mathcal{T}_s(\mathbb{F}^{s,\om})$,
and an $\mathcal{F}^{s,\om}_\tau$-measurable
random variable $\tilde{\xi}:\Omega\to\R$, 
denote by
\begin{align*}
(Y^{s,\om}(L,\tau,\tilde{\xi}),
Z^{s,\om}(L,\tau,\tilde{\xi}),
U^{s,\om}(L,\tau,\tilde{\xi}))
\end{align*}
the solution to the BSDE
\begin{align*}
Y_t&=\tilde{\xi}+\int_t^T 
\bfone_{\{r<\tau\}}\, L\left[
\abs{\sigma_r^\top Z_r}_1+
\int_{\R^d} U_r(z)^+\,\eta_r(z)\,K_r(dz)
\right]\,dr\\
&\qquad -\int_t^T Z_r\,dX^{c,s,\om}_r
-\int_t^T\int_{\R^d}
U_r(z)\,(\mu^X-\nu)(dr,dz),\quad
t\in [s,T],\quad\text{$\Prob_{s,\om}$-a.s.}
\end{align*}
\begin{remark}
Note that in the driver of the BSDE
above we use the positive part
$U(z)^+$ instead of the absolute value
$\abs{U(z)}$ in order for the comparison
principle for BSDEs to hold
(see \cite{BarlesEtAl97}).
\end{remark}
\begin{lemma}\label{L:gExpect}
We have
\begin{align}\label{E1:gExpect}
\overline{\mathcal{E}}^L_{s,\om} 
[\tilde{\xi}]=
\Mean_{s,\om} [Y_s^{s,\om}(L,\tau,\tilde{\xi})].
\end{align}
\end{lemma}

\begin{proof}
For the sake of readability, we omit
to write  $^{s,\om}(L,\tau,\tilde{\xi})$ in 
this proof.

Given a process $H$ and a random field $W$,
let $\Gamma=\Gamma^{H,W}$ be the solution to
\begin{align*}
\Gamma=1+(\Gamma_-H)\sint X^{c,s,\om}+
(\Gamma_-W)\ast(\mu^X-\nu)
\end{align*}
on $[s,T]$ with $\Gamma=1$ on $[0,s)$,
$\Prob_{s,\om}$-a.s. Since
$Z=\bfone_{\llb s,\tau\llb} Z$,
$U(z)=\bfone_{\llb s,\tau\llb} U(z)$, and
\begin{align*}
Y&=Y_s-\bfone_{\llb s,\tau\llb}\, L\left[\abs{\sigma^\top Z}_1+
\int_{\R^d} LU(z)^+\,\eta K(dz)\right]\sint t
\\&\qquad+
\sum_i Z^i\sint X^{i,c,s,\om}+U\ast (\mu^X-\nu),
\end{align*}
integration-by-parts (Lemma~\ref{L:IntegrByParts})
yields
\begin{equation}\label{E1a:gExpect}
\begin{split}
\Gamma Y&=Y_s+\bfone_{\llb s,\tau\llb}\Gamma
\left[-L
\abs{\sigma^\top Z}_1+
\sum_{i,j} H^i Z^j c^{ij}
\right]\sint t\\ &\qquad+
\bfone_{\llb s,\tau\llb}\Gamma\left[\int_{\R^d}
-LU(z)^+\,\eta(z)\,K(dz) +
\int_{\R^d} W(z)U(z)\,K(dz)
\right]\sint t\\ &\qquad +
\text{ martingale}\\
&=:Y_s+A^1+A^2+\text{ martingale}.
\end{split}
\end{equation}
Our goal is to choose $H$ and $W$ so that the 
drift term $A^1+A^2$ vanishes.
Let us first deal with $A^2$. If $W$ is defined
by
\begin{align}\label{E2:gExpect}
W(z)=L\eta(z).\bfone_{\{
U(z)> 0
\}},
\end{align}
then $A^2=0$ and $0\le W(z)\le L\eta(z)$.
Next, we deal with $A^1$. We need 
\begin{align}\label{E3:gExpect}
H^\top c Z=(H^\top\sigma)(\sigma^\top Z)=
L\abs{\sigma^\top Z}_1
\end{align}
to hold. To this end, put $\tilde{Z}=\sigma^\top Z$
and $\tilde{H}=\sigma^\top H$. Then
\eqref{E3:gExpect} is equivalent to 
\begin{align*}
\sum_i \tilde{H}^i \tilde{Z}^i=
L\sum_i \abs{\tilde{Z}^i}.
\end{align*}
Thus, if
\begin{align}\label{E4:gExpect}
\tilde{H}^i=L
\frac{\abs{\tilde{Z}^i}}{\tilde{Z}^i}.
\bfone_{\{\tilde{Z}^i\neq 0\}}\qquad\text{and}
\qquad
H=\begin{cases}
(\sigma^\top)^{-1}\tilde{H}
&\text{if $d>1$,}\\
\sigma^{-1}\tilde{H}.\bfone_{\{\sigma\neq 0\}}
&\text{if $d=1$,}
\end{cases}
\end{align}
then we get \eqref{E3:gExpect}, i.e., $A^1=0$,
and, moreover, we have
$\abs{\sigma^\top H}_\infty=
\abs{\tilde{H}}_\infty\le L$. Consequently,
for $H$ defined by \eqref{E4:gExpect}
and $W$ defined by \eqref{E2:gExpect},
we have
\begin{align}\label{E5:gExpect}
\overline{\mathcal{E}}^L_{s,\om}[\tilde{\xi}]\ge 
\Mean_{s,\om}\left[
\Gamma^{H,W}_\tau Y_\tau
\right]=\Mean_{s,\om} [Y_s].
\end{align}

On the other hand, for every process
$H$ with $\abs{\sigma^\top H}_\infty\le L$
and every random field $W$ with 
$0\le W(z)\le L\eta(z)$, we have
\begin{align*}
H^\top cZ&\le 
\abs{\sigma^\top H}_\infty 
\abs{\sigma^\top Z}_1 \le 
L\abs{\sigma^\top Z}_1\qquad\text{and}\qquad
W(z)U(z)\le L\eta(z) U(z)^+,
\end{align*}
which, by \eqref{E1a:gExpect}, yields
\begin{align*}
\Mean_{s,\om}\left[
\Gamma^{H,W}_\tau Y_\tau
\right]\le \Mean_{s,\om}[Y_s].
\end{align*}
This, together with \eqref{E5:gExpect},
 establishes \eqref{E1:gExpect}.
\end{proof}

\subsection{RBSDEs with jumps}
Our proof of the partial comparison principle
relies heavily on the theory of RBSDEs. See
\cite{ElKarouiEtAl97_RBSDEs},
\cite{HamadeneOuknine03},
\cite{CrepeyMatoussi08_RBSDEs},
and Chapter~14 of \cite{DelongBook}
for more details.

Fix a bounded, right-continuous, 
$\F$-adapted process 
$R:\bar{\Lambda}\to\R$ that is  
$\Prob_{s,\om}$-quasi-left-continous 
on $[s,T]$ for every $(s,\om)\in\bar{\Lambda}$.
Fix also $L\ge 0$.

For every $(s,\om)\in\bar{\Lambda}$
and $\ch\in\cT_s(\F^{s,\om})$, 
there exists, because of
 the martingale
representation property 
(Theorem III.4.29 in \cite{JacodShiryaevBook}),
 a unique solution 
\begin{align*} 
 (\bar{Y},\bar{Z},\bar{U},\bar{K})=
 (\bar{Y}^{s,\om}(L,\ch,R),
 \bar{Z}^{s,\om}(L,\ch,R),
 \bar{U}^{s,\om}(L,\ch,\R),
 \bar{K}^{s,\om}(L,\ch,\R))
\end{align*}
with $\bar{Y}=\bar{Y}_{\cdot\wedge\ch}$,
$\bar{Z}=\bfone_{\{\cdot\le\ch\}}\bar{Z}$,
$\bar{U}=\bfone_{\{\cdot\le\ch\}}\bar{U}$,
and $\bar{K}=\bar{K}_{\cdot\wedge\ch}$
 to the following RBSDE with lower barrier $R$
and random terminal time $\ch$
 (cf.~Remark~2.4~in
  \cite{CrepeyMatoussi08_RBSDEs}):
\begin{align*}
\begin{split}
\bar{Y}_t&=
R_\ch+\int_t^T \bfone_{\{r\le\ch\}}\,L\left[
\abs{\sigma^\top_r\bar{Z}_r}_1+
\int_{\R^d} \bar{U}_r(z)^+\,\eta_r(z)\,
K_r(dz)
\right]\,dr + 
\bar{K}_\ch-\bar{K}_{t}\\
&\quad -
\int_t^T \bar{Z}_r\,dX^{c,s,\om}_r -
\int_t^T \int_{\R^d} \bar{U}_r(z)
\,(\mu^X-\nu)(dr,dz),
t\in [s,T],\text{ $\Prob_{s,\om}$-a.s.,}
\end{split}\\
\bar{Y}_{t}&\ge R_{t\wedge\ch},
t\in [s,T],\text{ $\Prob_{s,\om}$-a.s.,}\\
\int_s^T &(\bar{Y}_{t}-R_{t})
\,d\bar{K}_{t}=0,\quad
\bar{K}^{s,\om}_s=0,\quad
\text{$\bar{K}$ is continuous
and nondecreasing.}
\end{align*}

\begin{lemma}\label{L:RBSDEs_DPP}
Fix 
$\tau\in\cT_s(\F^{s,\om})$
and $\ch\in\cH_{s,\om}$
with $\tau\le\ch$, $\Prob_{s,\om}$-a.s.
 Then, for
$\Prob_{s,\om}$-a.e.~$\tiom\in\Omega$,
the processes 
$\bar{Y}^{s,\om}(L,\ch,R)$ and
$\bar{Y}^{\tau,\tiom}(L,\ch,R)$ are
$\Prob_{\tau,
\tiom}$-indistinguishable on $[\tau(\tiom),T]$
and  $\bar{K}^{\tau,\tiom}(L,\ch,R)=\bar{K}^{s,\om}(L,\ch,R)-
\bar{K}^{s,\om}_{\tau(\tiom)}(L,\ch,R)$
 on $[\tau(\tiom),T]$.
\end{lemma}
\begin{proof}
Proceed as in the proof of Lemma~\ref{L:BSDE-DPP}
and note that Proposition~\ref{L:shiftMG_augment}
also applies to stopping times.
Utilizing uniqueness for RBSDEs will conclude
the proof.
\end{proof}

Given $\ch\in\cT_s(\F^{s,\om})$,
consider the optimal stopping times
\begin{align*}
\tau^\ast_{s,\om;\ch}&:=\inf\{t\ge s:
\bar{Y}_t^{s,\om}(L,\ch,R)=R_{t\wedge\ch}\},\\
\tau^{\ast\ast}_{s,\om;\ch}&:=\inf\{t\ge s:
\bar{K}^{s,\om}_t(L,\ch,R)>0\}.
\end{align*}
Note that, since $\bar{K}^{s,\om}(L,\ch,R)$ is continuous,
we have
$\tau^\ast_{s,\om;\ch}\le
\tau^{\ast\ast}_{s,\om;\ch}\wedge\ch$.

\begin{lemma}\label{L:PathSnell}
If $\ch\in\cT_s(\F^{s,\om})$, then
\begin{align}\label{E1:PathSnell}
\Mean_{s,\om}[\bar{Y}^{s,\om}_s(L,\ch,R)]=
\sup_{\tau\in\cT_s(\F^{s,\om})}
\overline{\mathcal{E}}^L_{s,\om}
[R_{\tau\wedge\ch}].
\end{align}
\end{lemma}
A corresponding result for RBSDEs without jumps
has been proven in \cite{Morlais13RBSDEs}.
We follow the approach in \cite{BayraktarYao12QBSDEs}, where 
quadratic RBSDEs without jumps are studied.
\begin{proof}[Proof of Lemma~\ref{L:PathSnell}]
Provided there is no danger of
confusion, we omit to write 
$^{s,\om}(L,\ch,R)$ in this proof.
Let us first fix a stopping time
$\tau\in\cT_s(\F^{s,\om})$. Note that
\begin{align*}
\bar{Y}_t&=
\bar{Y}_{\tau\wedge\ch}+
\int_t^T \bfone_{\{r\le\tau\wedge\ch\}}\,L\left[
\abs{\sigma^\top_r\bar{Z}_r}_1+
\int_{\R^d}\bar{U}_r(z)^+\,
\eta_r(z)\,K_r(dz)
\right]\,dr\\
&\qquad +\bar{K}_{\tau\wedge\ch}-\bar{K}_t
-\int_t^T \bfone_{\{r\le\tau\wedge\ch\}}
\,\bar{Z}_r\,
dX^{c,s,\om}_r\\
&\qquad -\int_t^T \int_{\R^d}
\bfone_{\{r\le\tau\wedge\ch\}}\,
\bar{U}_r(z)\,(\mu^X-\nu)(dr,dz),\,
t\in [s,T],\,\text{$\Prob_{s,\om}$-a.s.}
\end{align*}
Since $\bar{Y}_{\tau\wedge\ch}
\ge R_{\tau\wedge\ch}$
and $\bar{K}_{\tau\wedge\ch}-\bar{K}_t\ge 0$,
the comparison principle for BSDEs with jumps
(combine, e.g., the proofs of 
Theorem~4.2 of \cite{CrepeyMatoussi08_RBSDEs}
and Theorem~5.1 of 
\cite{BayraktarYao12QBSDEs}) yields
\begin{align*}
\bar{Y}^{s,\om}_t(L,\ch,R)\ge 
Y^{s,\om}_t(L,\tau\wedge\ch,R_\tau),
\quad s\le t\le T,
\qquad\text{$\Prob_{s,\om}$-a.s.}
\end{align*}

Consequently, by Lemma~\ref{L:gExpect},
\begin{align}\label{E2:PathSnell}
\Mean_{s,\om}[\bar{Y}_s]\ge 
\sup_{\tau\in\cT_s(\F^{s,\om})}
\overline{\mathcal{E}}^L_{s,\om}
[R_{\tau\wedge\ch}].
\end{align}

Next, consider the (optimal) stopping time
$\tau^\ast:=\tau^\ast_{s,\om;\ch}$. Since
$\bar{K}=0$ on $\llb s,\tau^\ast\rrb$,
$\tau^\ast\le\ch$,
and $\bar{Y}_{\tau^\ast}=
R_{\tau^\ast}$, we have
\begin{align*}
\bar{Y}^{s,\om}_{\cdot\wedge\tau^\ast}
(L,\ch,R)=
Y^{s,\om}
(L,\tau^\ast\wedge\ch,R_{\tau^\ast\wedge\ch}).
\end{align*}
Thus, by Lemma~\ref{L:gExpect},
$\Mean_{s,\om}[\bar{Y}_s]=
\overline{\mathcal{E}}^L_{s,\om}
[R_{\tau^\ast\wedge\ch}]$.
Together with \eqref{E2:PathSnell}, we get
\eqref{E1:PathSnell}.
\end{proof}

\begin{lemma}\label{L:PathSnellII}
If $\ch\in\cH_{s,\om}$, then
for $\Prob_{s,\om}$-a.e.~$\tiom\in\Omega$,
\begin{align*}
\bar{Y}^{s,\om}_{\tau^\ast_{s,\om;\ch}}
(\tiom;L,\ch,R)=
\sup_{
\tau\in\cT_{ 
\tau^{\ast}_{s,\om;\ch}(\tiom)
 }
(\F^{\tau^\ast_{s,\om;\ch},\tiom})
}\overline{\mathcal{E}}^L_{
\tau^{\ast}_{s,\om;\ch}(\tiom)
}
[R_{\tau\wedge\ch}].
\end{align*}
\end{lemma}

\begin{proof}
Write $\tau^\ast$ instead of
$\tau^{\ast}_{s,\om;\ch}(\tiom)$.
Let $\ttau$ be a stopping
time belonging to $\cT_s(\F^0_+)$ such that
$\ttau=\tau^\ast$, $\Prob_{s,\om}$-a.s.
Note that $\ttau\le \ch$, $\Prob_{s,\om}$-a.s.
Then, for $\Prob_{s,\om}$-a.e.~$\tiom\in\Omega$,
\begin{align*}
\bar{Y}^{s,\om}_{\tau^\ast}(\tiom;L,\ch,R)&=
\bar{Y}^{s,\om}_{\ttau}(\tiom;\L,\ch,R)\\
&=
\Mean_{s,\om}[
\bar{Y}^{s,\om}_{\ttau}(L,\ch,R)\vert\cF^0_{\ttau+}
](\tiom)\\
&=\Mean_{\ttau,\tiom}[
\bar{Y}^{s,\om}_{\ttau(\tiom)}(L,\ch,R)
]
&&\text{by
 \eqref{E2:cpd_rightLim}
  and Lemma~\ref{L:MGproblem}}\\
&=\Mean_{\ttau,\tiom}[
\bar{Y}^{\ttau,\tiom}_{\ttau(\tiom)}
(L,\ch,R)
]
&&\text{by Lemma~\ref{L:RBSDEs_DPP}}\\
&=\sup_{\tau\in\cT_{\ttau(\tiom)}
(\F^{\ttau,\tiom})}
\overline{\mathcal{E}}^L_{\ttau,\tiom}
[R_{\tau\wedge H}]
&&\text{by Lemma~\ref{L:PathSnell}}\\
&=\sup_{\tau\in\cT_{\tau^\ast(\tiom)}
(\F^{\tau^\ast,\tiom})}
\overline{\mathcal{E}}^L_{\tau^\ast,\tiom}
[R_{\tau\wedge H}].
\end{align*}
This concludes the proof.
\end{proof}

\subsection{Partial Comparison}
We will need the following modification of
the partial comparison principle. 
Theorem~\ref{T:PartialComp} can be proven
similarly. In order to formulate our result we need
the following definition. It might be helpful to recall
Definition~\ref{D:TheSets}.

\begin{definition}
Fix $t\in [0,T)$.
The space $\bar{C}_b^{1,2}(\bar{\Lambda}^t)$ is the set of all
universally measurable functionals $u:\bar{\Lambda}^t\to\R$ for which
there exist a sequence  $(\tau_n)_{n\in\N_0}$ of stopping times in $\mathcal{H}_t$
and a collection of functions $$(\vartheta_n(\pi_n;\cdot))_{n\in\N,\pi_n\in\Pi_n^t}$$
on $[t,T]\times\R^d$ such that the following holds:
\begin{enumerate}
\renewcommand{\labelenumi}{(\roman{enumi})}
\item The sequence $(\tau_n)$ is nondecreasing, $\tau_0=t$, $\tau_n<\tau_{n+1}$
if $\tau_n<T$, and, for every $\om\in\Omega$,
there exists an $m\in\N$ such that $\tau_m(\om)=T$.
\item For every $n\in\N$, $\vartheta_n=\vartheta_n(\pi_n;t,x)$ is universally measurable
 and, for every $\pi_n=(s_i,y_i)_{0\le i\le n-1}$, the function
$\vartheta_n(\pi_n;\cdot)$ is 
continuous on $[s_{n-1},T]\times\R^d$ and  belongs to
$C_b^{1,2}([s_{n-1},T)\times O_\eps(y_{n-1}))$
for some $\eps>0$.
\item For every $n\in\N$ and $\om\in\Omega$,
\begin{align*}
&\vartheta_n((\tau_i(\om),X_{\tau_i}(\om))_{0\le i\le n-1};\tau_n(\om),X_{\tau_n}(\om))=\\ &\qquad\qquad
\vartheta_{n+1}((\tau_i(\om),X_{\tau_i}(\om))_{0\le i\le n};\tau_n(\om),X_{\tau_n}(\om)).\bfone_{\{\tau_n<T\}}(\om)
+u(T,\om).\bfone_{\{\tau_n=T\}}(\om)
\end{align*}
\item We have the representation
\begin{align*}
u(s,\om)&=\sum_{n\ge 1} \vartheta_n
((\tau_i(\om),X_{\tau_i}(\om))_{0\le i\le n-1};s,\om_s).\bfone_{\llb\tau_{n-1},\tau_n\llb}(s,\om)\\&\qquad\qquad
+u(T,\om).\bfone_{\{T\}}(s).
\end{align*}
\end{enumerate}
\end{definition}

\begin{theorem}[Partial Comparison II]\label{T:PartialCompII}
Fix $(s,\om)\in\Lambda$.
Let $u^1$ be 
a viscosity subsolution of \eqref{E:PPIDE}
on $\Lambda^s$. Let $u^2\in\bar{C}_b^{1,2}(\bar{\Lambda}^s)$ 
with a corresponding sequence $(\tau_n)$ of stopping times
and  a corresponding collection $(\vartheta_n)$ of functionals
such that, for every $n\in\N$ and every $(r,\tiom)\in\llb \tau_{n-1}, 
\tau_n\llb$, we have, with
$\pi_n=(\ch^{t,\eps}_i(\tiom),X_{\ch^{t,\eps}_i}(\tiom))_{0\le i\le n-1}$,
\begin{align*}
&-\partial_t \vartheta_n(\pi_n;r,\tiom_r)-
\sum_{i=1}^d b^i_r(\tiom)\partial_{x^i}\vartheta_n(\pi_n;r,\tiom_r)-
\frac{1}{2}\sum_{i,j=1}^d c^{ij}_r(\tiom)\partial^2_{x^ix^j}\vartheta_n(\pi_n;r,\tiom_r)\\
&\qquad -\int_{\R^d}\left[\vartheta_n(\pi_n;r,\tiom_r+z)-
\vartheta_n(\pi_n;r,\tiom_r)-\sum_{i=1}^d z^i\partial_{x^i}
\vartheta_n(\pi_n;r,\tiom_r) \right]\,K_r(\tiom,dz)\\
&\qquad -f_r\Biggl(\tiom,\vartheta_n(\pi_n;r,\tiom_r),
\partial_x\vartheta_n(\pi_n;r,\tiom_r),\\
&\qquad\qquad\qquad
\int_{\R^d}\left[ 
\vartheta_n(\pi_n;r,\tiom_r+z)-
\vartheta_n(\pi_n;r,\tiom_r)
\right]\eta_r(\tiom)\,K_r(\tiom,dz)\Biggr)\ge 0.
\end{align*}
Suppose that $u^1_T\le u^2_T$,
$\Prob_{s,\om}$-a.s. Then $u^1(s,\om)\le 
u^2(s,\om)$.
\end{theorem}

\begin{proof}
Our proof follows  the 
approach of Lemma~5.7 in \cite{EKTZ11}.
However, due to our more general setting
and different definition of
$\bar{C}_b^{1,2}$,
details are somewhat more involved and therefore
we shall give a complete proof.

Without loss of generality, let $s=0$
and let $f$ be nonincreasing in $y$
(cf.~Remark~3.9 in \cite{EKTZ11}). For the sake of a contradiction,
assume that
\begin{align*}
c:=(u^1-u^2)(0,\om)>0.
\end{align*}
Define a process $R:\bar{\Lambda}\to\R$ by
\begin{align*}
R_t:=(u^1-u^2)_t+\frac{ct}{2T}.
\end{align*}
Note that $R$ is $\Prob_{t,
\tiom}$-quasi-left-continuous
 for every $(t,\tiom)\in\bar{\Lambda}$,
bounded, right-continuous,
and $\F$-adapted. Thus the results
about RBSDEs and optimal stopping in this 
section are applicable. 
Put
\begin{align*}
\ch:=\inf\{t\ge 0: (u^1-u^2)_t\le 0\}.
\end{align*}
Clearly, $\ch\le T$, $\Prob_{0,\om}$-a.s.
Put
\begin{align*}
\bar{Y}&:=\bar{Y}^{0,\om}(L,\ch,R)\qquad
\text{and}\qquad
\tau^\ast:=\inf\{t\ge 0:\bar{Y}_t=R_{t\wedge\ch}\}.
\end{align*}
Since $\bar{Y}_\ch=R_\ch$, we have
$\tau^\ast\le\ch$.
Note that
$\overline{\mathcal{E}}^L_{0,\om}[R_H]\le c/2$,
but, by Lemma~\ref{L:PathSnell},
\begin{align*}
\overline{\mathcal{E}}^L_{0,\om}[R_{\tau^\ast}]
=\Mean_{0,\om}[\bar{Y}_0]\le 
\Mean_{0,\om}[R_0]=c.
\end{align*}
Thus, $\Prob_{0,\om}(\tau^\ast<\ch)>0$.
Consequently,
 there exists an $\om^\ast\in\Omega$ such
that $t^\ast:=\tau^\ast(\om^\ast)<\ch(\om^\ast)$,
that $(t^\ast,\om^\ast)\in\llb\tau_{n-1},\tau_n\llb$ for
some $n\in\N$,
and that, 
according to Lemma~\ref{L:PathSnellII}
and Lemma~\ref{L:ShiftingHittingtime}, from
which we get
the existence of a hitting time
$\tilde{\ch}\in\cH_{t^\ast}$
with $\tilde{\ch}>t^\ast$,
$\tilde{\ch}(\tiom)=\ch(\tiom)$,
and $\tilde{\ch}=\ch$, $\Prob_{0,\om}$-a.s.,
we have
\begin{align*}
R_{t^\ast}(\om^\ast)=
\sup_{\tau\in\cT_{t^\ast}(\F^{t^\ast,\om^\ast})}
\overline{\mathcal{E}}^L_{t^\ast,\om^\ast}
[R_{\tau\wedge\tilde{\ch}}].
\end{align*}
Next, define a process $\tilde{R}:
\bar{\Lambda}^{t^\ast}\to\R$ by
$\tilde{R}_t:=R_t-R_{t^\ast}(\om^\ast)$.
Then
\begin{align*}
0=\tilde{R}_{t^\ast}(\om^\ast)=
\sup_{\tau\in\cT_{t^\ast}(\F^{t^\ast,\om^\ast})}
\overline{\mathcal{E}}^L_{t^\ast,\om^\ast}
[\tilde{R}_{\tau\wedge\tilde{\ch}}].
\end{align*}
Our goal is to establish a decomposition
$\tilde{R}=u^1-\varphi$
such that
$\varphi\in C_b^0(\bar{\Lambda}^{t^\ast})\cap
C^{1,2}_b(\llb t^\ast,\tau\rrb)$ for some
stopping time $\tau\ge t^\ast$.
Since
\begin{align*}
\tilde{R}_t=u^1_t-\left[u^2_t+
(u^1-u^2)(t^\ast,\om^\ast)
-\frac{c(t-t^\ast)}{2T}
\right],
\end{align*}
we get such a decomposition with $\tau=\tau_n$ by setting
$\varphi:=u^1-\tilde{R}$.
Then $\varphi\in\underline{\mathcal{A}}^L u^1
(t^\ast,\om^\ast)$,
and since $(u^1-u^2)(t^\ast,\om^\ast)\ge 0$
(, which follows from $t^\ast<\tilde{\ch}(\om^\ast)$),
we have
\begin{align*}
0&\ge \mathcal{L}\varphi(t^\ast,\om^\ast)-
f_{t^\ast}(\om^\ast,\varphi(t^\ast,\om^\ast),
\partial_\om\varphi(t^\ast,\om^\ast),
\mathcal{I}\varphi(t^\ast,\om^\ast))\\
&=\mathcal{L} u^2(t^\ast,\om^\ast)
+\frac{c}{2T}-
f_{t^\ast}(\om^\ast,\varphi(t^\ast,\om^\ast),
\partial_\om u^2(t^\ast,\om^\ast),
\mathcal{I} u^2(t^\ast,\om^\ast))\\
&> \mathcal{L} u^2(t^\ast,\om^\ast)
-
f_{t^\ast}(\om^\ast,u^2(t^\ast,\om^\ast),
\partial_\om u^2(t^\ast,\om^\ast),
\mathcal{I} u^2(t^\ast,\om^\ast)).
\end{align*}
But this is a contradiction to
$u^2$ being a classical supersolution.
\end{proof}

\subsection{Stability}

\begin{proof}[Proof of Theorem~\ref{T:Stab}]
Let $\eps^\prime>0$ be undetermined for the moment.
Assume that $u$ is not a viscosity 
$L$-supersolution of \eqref{E:PPIDE}.
Then there exist $(s_0,\om)\in\Lambda$
and $\varphi\in\overline{\mathcal{A}}^L u(s_0,\om)$
such that
\begin{align*}
c^\prime:=\mathcal{L}\varphi(s_0,\om)-
f_{s_0}(\om,\varphi(s_0,\om),
\partial_\om\varphi(s_0,\om),
\mathcal{I}\varphi(s_0,\om))<0.
\end{align*}
Without loss of generality, $s_0=0$.
Next, define processes $R$, 
$R^\eps:\bar{\Lambda}\to\R$, $\eps>0$, by
\begin{align*}
R_t:=\varphi_t-u_t-\eps^\prime t,\quad
R_t^\eps:=\varphi_t-u_t^\eps-\eps^\prime t.
\end{align*}
Also, put
\begin{align*}
\tau_1:=\inf\left\{t\ge 0: \mathcal{L}\varphi_t-
f_t(X,\varphi_t,\partial_\om\varphi_t,
\mathcal{I}\varphi_t)\ge\frac{c^\prime}{2}\right\}.
\end{align*}
Note that $\tau_1\in\cH_0$ with $\tau_1>0$,
$\Prob_{0,\om}$-a.s. Thus, since
$\varphi\in\overline{\mathcal{A}}^L u(0,\om)$,
there exists an $\ch\in\cH_{0,\om}$
with $\ch>0$, $\Prob_{0,\om}$-a.s., such that
$0=R_0>\overline{\mathcal{E}}_{0,\om}^L
\left[
R_{\tau_1\wedge\ch}
\right]$
because we have
$\overline{\mathcal{E}}^L_{0,\om}
[(-\eps^\prime)(\tau_1\wedge\ch)]<0$ for otherwise
$\overline{\mathcal{E}}^L_{0,\om}
[(-\eps^\prime)(\tau_1\wedge\ch)]=0=
\overline{\mathcal{E}}^L_{0,\om}
[0]$
would, by  the comparison principle for BSDEs
with jumps (cf.~Theorem~3.2.1 in
\cite{DelongBook}) together with 
Lemma~\ref{L:gExpect}, imply that
$(-\eps^\prime)(\tau_1\wedge\ch)=0$, 
$\Prob_{0,\om}$-a.s., which is a contradiction. 
Now, let $\eps$ sufficiently small so that
$R_0^\eps>
\overline{\mathcal{E}}^L_{0,\om}
[R^\eps_{\tau_1\wedge\ch}]$.
Put
\begin{align*}
\tau_2:=\tau_1\wedge\ch,\quad
\bar{Y}^\eps:=\bar{Y}^{0,\om}
(L,\tau_2,R^\eps),\quad
\tau^\eps:=\inf\{t\ge 0:
\bar{Y}^\eps_t=R^\eps_{t\wedge\tau_2}\},
\end{align*}
where we used the notation of
 Subsection~\ref{SS:PartCom} for RBSDEs.
 Then, $\Prob_{0,\om}(\tau^\eps<\tau_2)>0$
 because otherwise, by Lemma~\ref{L:PathSnell},
 $R_0^\eps\le\bar{Y}_0^\eps=
 \overline{\mathcal{E}}^L_{0,\om}
 [R^\eps_{\tau_2}]<R_0^\eps$.
That is, there exists an $\om^\eps\in\Omega$
such that $t^\eps:=\tau^\eps(\om^\eps)<
\tau_2(\om^\eps)$, that
$\tau_2\in\cH_{t^\eps}$ with
$\tau_2>t^\eps$,
$\Prob_{t^\eps,\om^\eps}$-a.s.~(,
which is possible by
 Lemma~\ref{L:ShiftingHittingtime}), and that
\begin{align*}
R^\eps_{t^\eps}(\om^\eps)=
\sup_{
\tau\in\cT_{t^\eps}\left(
\F^{t^\eps,\om^\eps}
\right)
}\overline{\mathcal{E}}^L_{0,\om}
\left[R^\eps_{\tau\wedge\tau_2}\right].
\end{align*}
Define $\tilde{R}^\eps:\bar{\Lambda}^{t^\eps}\to\R$
by $\tilde{R}^\eps_t
:=R^\eps_t-R_{t^\eps}(\om^\eps)$. Then,
\begin{align*}
\tilde{R}^\eps_t=
\varphi_t-\eps^\prime(t-t^\eps)-u_t^\eps-
[\varphi(t^\eps,\om^\eps)-u(t^\eps,\om^\eps)]
\end{align*}
and, with
$\varphi^\eps_t:=\varphi_t-\eps^\prime(t-t^\eps)
-(\varphi-u)(t^\eps,\om^\eps)$, $t\in [t^\eps,T]$,
we have
\begin{align*}
0=(\varphi^\eps-u^\eps)(t^\eps,\om^\eps)=
\sup_{
\tau\in\cT_{t^\eps}\left(
\F^{t^\eps,\om^\eps}
\right)
}
\overline{\mathcal{E}}^L_{t^\eps,\om^\eps}
\left[
(\varphi^\eps-u^\eps)_{\tau\wedge\tau_2}
\right].
\end{align*}
That is, $\varphi^\eps\in\overline{\mathcal{A}}^L
u^\eps(t^\eps,\om^\eps)$, and thus
\begin{align*}
0&\le\mathcal{L}^\eps\varphi^\eps(t^\eps,\om^\eps)-
f^\eps_{t^\eps}
(\om^\eps,\varphi^\eps(t^\eps,\om^\eps),
\partial_\om\varphi^\eps(t^\eps,\om^\eps),
\mathcal{I}^\eps\varphi^\eps(t^\eps,\om^\eps))\\
&=\mathcal{L}\varphi(t^\eps,\om^\eps)+\eps^\prime-
\Bigl[f^\eps_{t^\eps}
(\om^\eps,u^\eps(t^\eps,\om^\eps),
\partial_\om\varphi(t^\eps,\om^\eps),
\mathcal{I}^\eps\varphi(t^\eps,\om^\eps))\\
&\qquad\qquad -
f_{t^\eps}(\om^\eps,u(t^\eps,\om^\eps),
\partial_\om\varphi(t^\eps,\om^\eps),
\mathcal{I}\varphi(t^\eps,\om^\eps))\Bigr]\\
&\qquad\qquad
 -f_{t^\eps}(\om^\eps,u(t^\eps,\om^\eps),
\partial_\om\varphi(t^\eps,\om^\eps),
\mathcal{I}\varphi(t^\eps,\om^\eps))\\
&\le \frac{c^\prime}{2}+\eps^\prime
-\Bigl[f^\eps_{t^\eps}
(\om^\eps,u^\eps(t^\eps,\om^\eps),
\partial_\om\varphi(t^\eps,\om^\eps),
\mathcal{I}^\eps\varphi(t^\eps,\om^\eps))\\
&\qquad\qquad -
f_{t^\eps}(\om^\eps,u(t^\eps,\om^\eps),
\partial_\om\varphi(t^\eps,\om^\eps),
\mathcal{I}\varphi(t^\eps,\om^\eps))\Bigr].
\end{align*}
Hence, for $\eps^\prime=-c^\prime/8$ and
for sufficiently small $\eps$, uniform
convergence of $f^\eps$, $u^\eps$, $\eta^\eps$,
and $K^\eps$ yields
$0\le c^\prime/4<0$, which is 
a contradiction.
\end{proof}

\section{Comparison} In this subsection, 
Assumption~\ref{A:xi}, Assumption~\ref{A:L},
Assumption~\ref{A:f}, Assumption~\ref{A:jumps},
 and Assumption~\ref{A:pi}
are in force. Furthermore, without loss of generality,
assume that $f$ is nonincreasing in $y$
(cf.~Remark~3.9 in \cite{EKTZ11}).

Given $\delta\in (0,\infty]$, $t\in (-\infty,T]$, and $y\in\R^d$, put
\begin{align*}
O_\delta(y)&:=\{x\in\R^d:\abs{x-y}<\delta\},\\
\bar{O}_\delta(y)&:=\{x\in\R^d:\abs{x-y}\le\delta\},\\
\partial O_\delta(y)&:=\{x\in\R^d:\abs{x-y}=\delta\},\\
Q^\delta_{t,y}&:=(t,T)\times O_\delta(y),\\
\bar{Q}^\delta_{t,y}&:=[t,T]\times \bar{O}_\delta(y),\\
\partial Q^\delta_{t,y}&:=
\left((t,T]\times \partial O_\delta(y)\right)\cup
\left(\{T\}\times O_\delta(y)\right),\\
\partial \bar{Q}^\delta_{t,y}&:=
\left([t,T]\times \partial O_\delta(y)\right)\cup
\left(\{T\}\times O_\delta(y)\right),\\
Q_t^\infty&:=(t,T)\times\R^d,\\
\bar{Q}_t^\infty&:=[t,T]\times\R^d.
\end{align*}

\subsection{Hitting times}
Given $\eps>0$, $t\in [0,T]$, and $y\in\R^d$, define 
hitting times
\begin{align*}
\ch^{t,y,\eps}_0&:= t,\\
\ch^{t,y,\eps}_1&:=
\inf\{s\ge t: X_s\not\in O_\eps(y)\}\wedge T,\\
\ch^{t,y,\eps}_{i+1}&:=
\inf\{s\ge\ch^{t,y,\eps}_i:\abs{
X_s-X_{\ch^{t,y,\eps}_i}
}\ge\eps\}\wedge T.
\end{align*}
Also  put $\ch_j^{t,\eps}:=\ch_j^{t,X^t_t,\eps}$, that is,
\begin{align*}
\ch_0^{t,\eps}&=t,\\
\ch_1^{t,\eps}&=\inf\{s\ge t:\abs{X_s-X_t}\ge\eps\}\wedge T,\\
\ch_{j+1}^{t,\eps}&=\inf\{s\ge\ch_j^{t,\eps}:
\abs{X_s-X_{\ch_j^{t,\eps}}}\ge\eps\}\wedge T.
\end{align*}

\begin{lemma}\label{L:hittingLeftLimit}
Let $(G_n)_n$ be an increasing sequence
of non-empty open subsets of $\R^d$
with $G=\cup_n G_n$.
Let $Y$ be a $d$-dimensional,
c\`adl\`ag, and $\F^0_+$-adapted process
that is $\Prob$-quasi-left-continuous
for some probability measure $\Prob$ on
$(\Omega,\cF^0_T)$. Suppose that $Y_0\in G_1$,
$\Prob$-a.s. Consider the first-exit times
\begin{align}\label{E1:firstExit}
\tau_G&:=\inf\{t\ge 0: Y_t\in G^c\}\wedge T,\\
\label{E2:firstExit}
\tau_{G_n}&:=\inf\{t\ge 0:
Y_t\in G^c_n\}\wedge T,\quad n\in\N.
\end{align}
Then $\lim_n \tau_{G_n}=\tau_G$,
$\Prob$-a.s.
Moreover, $\lim_n Y_{\tau_{G_n}}=Y_{\tau_G}$,
$\Prob$-a.s.
\end{lemma}
\begin{proof}
Without loss of generality, we can assume that
that $\tau_G$ and $\tau_{G_n}$, $n\in\N$,
are $\F^0_+$-stopping times. Clearly,
$(\tau_{G_n})_n$ is increasing and
\begin{align*}
\tau:=\sup_n \tau_{G_n}\le \tau_G.
\end{align*}
Consider the sets $A:=\{\tau<\tau_G\}$ and
$A_n:=\{\tau_{G_n}<\tau_G\}$, $n\in\N$. We have to show
that $\Prob(A)=0$.

Let us first note that, for every $n\in\N$,
we have $\tau_{G_n}<\tau$ on $A$ because
otherwise there exists
an $\om\in A$ and an $n_0\in\N$ such that,
for every $n\ge n_0$,
$\tau_{G_n}(\om)=\tau(\om)$, which 
yields $Y_\tau(\om)\in\cap_{n\ge n_0} G_n^c=G^c$,
i.e., $\tau(\om)\ge\tau_G(\om)$ and thus
$\om\not\in A$.
Consequently, the stopping times
\begin{align*}
\tilde{\tau}&:=\tau.\bfone_A+T.\bfone_{A^c},\\
\tilde{\tau}_{G_n}&:=
\left(\tau_{G_n}.\bfone_{A_n}+T.\bfone_{A^c_n}
\right)\wedge [T(1-n^{-1})],\quad n\in\N,
\end{align*}
satisfy the following:
For every $n\in\N$,
$\tilde{\tau}_{G_n}<\tilde{\tau}$, i.e.,
$\tilde{\tau}$ is $\F^0_+$-predictable.
Thus, using $\Prob$-quasi-left-continuity of $Y$,
we get 
\begin{align}\label{E1:hittingLeftLimit}
\Prob(\tilde{\tau}<T)&=
\Prob(\{\Delta Y_{\tilde{\tau}}=0\}\cap 
\{\tilde{\tau}<T\})=
\Prob(\{\Delta Y_\tau =0\}\cap A).
\end{align}
Next, note that $\abs{\Delta Y_\tau}>0$ on $A$
because otherwise there exists an $\om\in A$
such that $\Delta Y_\tau(\om)=0$
and then, since, for every $m\in\N$,
$\{Y_{\tau_{G_n}}(\om)\}_{n\ge m}$ takes
values in $G_m^c$ and $G_m^c$ is closed,
we get $Y_\tau(\om)\in \cap_m G_m^c=G^c$,
i.e., $\tau(\om)\ge\tau_G(\om)$ and thus
$\om\not\in A$. Therefore,
$\Prob(\{\Delta Y_\tau=0\}\cap A)=0$
and, together with
\eqref{E1:hittingLeftLimit}, 
we get
$\Prob(A)=0$. 

I.e., we have shown that
$\lim_n {\tau_{G_n}}=\tau_G$,
$\Prob$-a.s.
Hence, by Proposition~I.2.26 of 
\cite{JacodShiryaevBook},
$\lim_n Y_{\tau_{G_n}}=Y_{\tau_G}$,
$\Prob$-a.s.
\end{proof}

In the following statement and its proof,
the random times $\tau_G$ and $\tau_{G_n}$,
$n\in\N$, resp., are defined by
\eqref{E1:firstExit} and
\eqref{E2:firstExit}, resp.
Also, we do not need the full strength 
of Assumption~\ref{A:L}, in particular,
$(B,C,\nu)$ is allowed to be random;
 only Part~(iii) of
 Assumption~\ref{A:L} is actually
used. 
\begin{lemma}\label{L:ItoHitting}
Let $(G_n)_n$ be an increasing 
sequence of open connected subsets of $\R^d$
with $G=\cup_n G_n$.
 Let $H$ be an open
subset of $\R^d$ with $G\subseteq H$.
Put $Q_n:=[0,T)\times G_n$, $n\in\N$,
$Q:=[0,T)\times G$, and
$R:=[0,T)\times H$.
Suppose that 
there exists an $\eps\in (0,1)$ such
that, for all $n\in\N$, 
\begin{align*}
\mathrm{dist}(G^c_{n+1},G_n)\le 
\frac{\eps}{n(n+1)}.
\end{align*}
Let $v\in C(\bar{R})\cap C^{1,2}(Q)$
and $x\in G_1$. Then
there exists an
$(\F^0_+,\Prob_{0,x})$-martingale $M$ such that
\begin{align*}
&v(\tau_G,X_{\tau_G})-
v(0,x)=\int_0^{\tau_G}\Biggl\{
\partial_t v(t,X_t)+
\sum_{i=1}^d b^i_t\partial_{x^i} v(t,X_t)\\
&\qquad +
\frac{1}{2}\sum_{i,j=1}^d
 c_t^{ij}\partial^2_{x^ix^j}
v(t,X_t)\\&\qquad+\int_{\abs{z}\le C_0^\prime}
\left[ 
v(t,X_t+z)-v(t,X_t)-\sum_{i=1}^d z^i
\partial_{x^i} v(t,X_t)
\right]\, K_t(dz)\Biggr\}
\, dt\\&\qquad 
+M_{\tau_G},\quad
\text{$\Prob_{0,x}$-a.s.}
\end{align*}
\end{lemma}

\begin{proof}
Let $(a_n)_n$ be a sequence of 
positive real numbers converging to $0$.
For every $n\in\N$, let
$v^n\in C^{1,2}(\bar{R})$ such that
$v=v^n$ on $Q_n$ and
$\abs{v-v^n}_R\le a_n$.
By Lemma~\ref{L:hittingLeftLimit},
$v(\tau_G,X_{\tau_G})=\lim_n
v^{n+1}(\tau_{G_n},X_{\tau_{G_n}})$,
$\Prob_{0,x}$-a.s. Moreover,  for 
every $n\in\N$, there exists,
by It\^o's formula an 
$(\F^0_+,\Prob_{0,x})$-martingale $M^n$
such that
\begin{align*}
&v^{n+1}(\tau_{G_n},X_{\tau_{G_n}})-
v(0,x)=\int_0^{\tau_{G_n}}\Biggl\{
\partial_t v(t,X_t)+
\sum_{i=1}^d b^i_t\partial_{x^i} v(t,X_t)\\
&\qquad +
\frac{1}{2}\sum_{i,j=1}^d
 c_t^{ij}\partial^2_{x^ix^j}
v(t,X_t)\\&\qquad+\int_{\abs{z}\le C_0^\prime}
\Biggl[ 
v(t,X_t+z).\bfone_{\left\{\abs{z}\le
 \frac{\eps}{n(n+1)}\right\}}+
v^{n+1}(t,X_t+z).\bfone_{\left\{\abs{z}> 
\frac{\eps}{n(n+1)}\right\}}
\\ &\qquad\qquad-v(t,X_t)-\sum_{i=1}^d z^i
\partial_{x^i} v(t,X_t)
\Biggr]\, K_t(dz)\Biggr\}
\, dt\\&\qquad 
+M^n_{\tau_{G_n}},\quad
\text{$\Prob_{0,x}$-a.s.}
\end{align*}
Now, if
\begin{align*}
a_{n+1}=\frac{1}{(L_{1/n}\vee n)(n+1)},
\quad n\in\N,
\end{align*}
then, by Part~(iii) of Assumption~\ref{A:L},
\begin{align*}
&\abs{\int_0^{\tau_{G_n}}
\int_{\frac{\eps}{n(n+1)}
<\abs{z}\le C_0^\prime} 
[v^{n+1}(t,X_t+z)-v(t,X_t+z)]\,K_t(dz)\,dt}\\
&\quad \le 
\int_0^{\tau_{G_n}} \int_{
\frac{\eps}{n(n+1)}<\abs{z}
\le C_0^\prime} a_{n+1} \,K_t(dz)\,dt\\
&\quad \le T\left[
\int_{\frac{\eps}{n(n+1)}<\abs{z}\le C_0^\prime}
a_{n+1} \, K_{t,1,1/n}(dz)+
\int_{\frac{\eps}{n(n+1)}<\abs{z}\le C_0^\prime}
a_{n+1}\, K_{t,2,1/n}(dz)
\right]\\
&\quad \le T\left[
\frac{n(n+1)a_{n+1}}{\eps})
\int_{\frac{\eps}{n(n+1)}<\abs{z}\le C_0^\prime}
\abs{z}  K_{t,1,1/n}(dz)+
a_{n+1} (L_{1/n}\vee n)
\right]\\
&\quad \le 
T\left[\frac{1}{(L_{1/n}\vee n)\eps} 
+\frac{1}{n+1}\right]\to 0
\end{align*}
as $n\to\infty$. This concludes the proof.
\end{proof}

\begin{remark}\label{R:hitting}
If $(b,c,K)$ is constant, $t\in [0,T]$, and $x\in\R^d$, then
\begin{align*} 
\Mean_{t,x}[\ch_1^{t,\eps}]=\Mean_{0,x} [t+[\ch_1^\eps\wedge (T-t)]].
\end{align*}
Indeed, using $(a_1\wedge a_2)-t=(a_1-t)\wedge(a_2-t)$, we get
\begin{align*}
\Mean_{t,x}[(\ch_1^{t,\eps}-t]&=
\Mean_{t,x}[(\inf\{s\ge t: |X_s-X_t|\ge\eps\}\wedge T)-t]\\
&=\Mean_{0,x}[(\inf\{s\ge t: |X_{s-t}-X_0|\ge\eps\}\wedge T)-t]\\
&=\Mean_{0,x}[((t+\inf\{r\ge 0: |X_r-X_0|\ge\eps\})\wedge T)-t]\\
&=\Mean_{0,x}[\inf\{r\ge 0:|X_r-X_0|\ge\eps\}\wedge (T-t)]\\
&=\Mean_{0,x}[\ch_1^\eps\wedge (T-t)].
\end{align*}
\end{remark}

\begin{proposition}\label{R:hittingTime}
Fix $(s,\om)\in\bar{\Lambda}$ and $\eps>0$.
Then the map $x\mapsto\ch_1^{s,x,\eps}(\om)$,
$\R^d\to [s,T]$, is universally measurable.
\end{proposition}
\begin{proof}
Note that we can express $x\mapsto\ch_1^{s,x,\eps}(\om)$ as
the deb\`ut of the set
\begin{align*}
A:=\{(t,x)\in [s,T]\times\R^d: \abs{\om_t-x}\ge\eps\}\cup\{(T,x):x\in\R^d\},
\end{align*} 
i.e., $\ch_1^{s,x,\eps}(\om)=D_A(x)$, where
$D_A:\R^d\mapsto [0,T]$ is defined by
\begin{align*}
D_A(x):=\inf\{t\in [s,T]: (t,x)\in A\}.
\end{align*}
Since $A\in\mathcal{B}([s,T])\otimes\mathcal{B}(\R^d)$ as inverse image
of a Borel set under a Borel measurable map,
we can deduce from Theorem~III.44 in \cite{DellacherieMeyer}
that $D_A$ is universally measurable.
\end{proof}
\begin{remark}\label{R2:hittingTime}
If $d=1$, then $\ch_1^{t,x,\eps}$ can be written
as  infimum of first-passage times that are monotone in $x$
and thus  $x\mapsto\ch_1^{s,x,\eps}(\om)$ is even Borel measurable.
\end{remark}

\begin{lemma} [Shifting of hitting times]\label{L:ShiftingHitting-new}
Let $\eps>0$, $t\in [0,T]$, $y\in\R$, $\om\in\Omega$,
and $s=\ch_1^{t,y,\eps}(\om)$.
Then
$\ch^{t,y,\eps}_{i+1}=
\ch^{s,\om_s,\eps}_i$,
$i\in\N_0$, $\Prob_{s,\om}$-a.s.
\end{lemma}

\begin{proof}
The case $i=0$ follows from Remark~\ref{R:ShiftMg}.
Next, let $i=1$. Then, by  Remark~\ref{R:ShiftMg},
\begin{align*}
\ch_2^{t,y,\eps}
=\inf\{r\ge s:\abs{
X_r-X_s}\ge\eps\}\wedge T
=\ch_1^{s,\om_s,\eps},\quad\text{$\Prob_{s,\om}$-a.s.}
\end{align*}
Finally, assume that the statement is true for some $i\in\N$.
Then
\begin{align*}
\ch^{t,y,\eps}_{i+2}
=\inf\{r\ge\ch_{i+1}^{s,\om_s,\eps}:
\abs{
X_r-X_{\ch_i^{t,\om_s,\eps}}
}\ge\eps\}\wedge T
=\ch_{i+1}^{s,\om_s,\eps},\quad\text{$\Prob_{s,\om}$-a.s.}
\end{align*}
Mathematical induction concludes the proof.
\end{proof}

\subsection{Regularity of path-frozen approximations}
We are going to 
define candidate solutions for so-called path-frozen 
integro-differential equations
by means stochastic representation.  To this end
fix  $\eps>0$ and $(s^\ast,\om^\ast)\in\Lambda$. 
Next, define a map
$g^{s^\ast,\om^\ast}=g:\Pi^{s^\ast}_\infty\to\R$ by
\begin{align*}
g(\pi_\infty):=
\xi(\om^\ast.\bfone_{[0,s^\ast)}+
\sum_{i\in\N_0} x_i.\bfone_{[t_i,t_{i+1})}+x_\infty.\bfone_{\{T\}})
\end{align*}
and a map $\tilde{f}^{s^\ast,\om^\ast}=\tilde{f}:\R\times
\Pi^{s^\ast}_\infty\times\R\times\R^d\times\R\to\R$ by
\begin{align*}
\tilde{f}_t(\pi_\infty,y,z,p):=
f_{(t\vee s^\ast)\wedge T}(\om^\ast.\bfone_{[0,s^\ast)}+
\sum_{i\in\N_0} x_i.\bfone_{[t_i,t_{i+1})}+x_\infty.\bfone_{\{T\}},y,z,p).
\end{align*}

Given $i\in\N$, $\pi_i=(s_0,y_0;\ldots;s_{i-1},y_{i-1})\in\Pi_i^{s^\ast}$
with $s^\ast=s_0$, $s=s_{i-1}$, and $y=y_{i-1}$, denote, for every $(t,x)\in [s,T]
\times\R^d$, by
\begin{align*}
\left(
\tilde{Y}^{s^\ast,\om^\ast,\eps;\pi_i,t,x},
\tilde{Z}^{s^\ast,\om^\ast,\eps;\pi_i,t,x},
\tilde{U}^{s^\ast,\om^\ast,\eps;\pi_i,t,x}
\right)=\left(
\tilde{Y}^{\pi_i,t,x},
\tilde{Z}^{\pi_i,t,x},
\tilde{U}^{\pi_i,t,x}
\right)
\end{align*}
the solution of the BSDE
\begin{align*}
&\tilde{Y}^{\pi_i,t,x}_r=
g(\pi_i;(\ch_j^{t,y-x,\eps},x+X_{\ch_j^{t,y-x,\eps}})_{j\in\N});x+X_T)\\
&\quad +\int_r^T\tilde{f}_{\tilde{r}}\Bigl(
(\pi_i;(\ch_j^{t,y-x,\eps},x+X_{\ch_j^{t,y-x,\eps}})_{j\in\N});x+X_T),
\\&\quad\qquad
\qquad\qquad
\tilde{Y}^{\pi_i,t,x}_{\tilde{r}},\tilde{Z}^{\pi_i,t,x}_{\tilde{r}},\int_{\R^d} \tilde{U}^{\pi_i,t,x}_{\tilde{r}}(z)\,\eta_{\tilde{r}}(z)\,
K_{\tilde{r}}(dz)
\Bigr)\,d\tilde{r}\\&\quad
-\int_r^T \tilde{Z}^{\pi_i,t,x}_{\tilde{r}}\,dX^{c,t,\bfnull}_{\tilde{r}}
-\int_r^T\int_{\R^d}\tilde{U}^{\pi_i,t,x}_{\tilde{r}}(z)\,
(\mu^X-\nu)(d\tilde{r},dz),\,r\in [t,T],\,\text{$\Prob_{t,\bfnull}$-a.s.,}
\end{align*}
and define $\theta_i^{s^\ast,\om^\ast,\eps}(\pi_i;\cdot)=
\theta_i(\pi_i;\cdot):[s,T]\times\R^d\to\R$ by
\begin{align*}
\theta_i(\pi_i;t,x):=\Mean_{t,\bfnull}[\tilde{Y}^{\pi_i;t,x}_t].
\end{align*}

\begin{lemma}[Dynamic programming]\label{L:DPP}
We have
\begin{align*}
\theta_i(\pi_i;t,x)&=
\Mean_{t,\bfnull}\Bigl[
\theta_{i+1}(
\pi_i;\ch^{t,y-x,\eps}_1,x+X_{\ch^{t,y-x,\eps}_1};
\ch^{t,y-x,\eps}_1,x+X_{\ch^{t,y-x,\eps}_1})
\\&\qquad\qquad+\int_t^{\ch^{t,y,\eps}_1}
\tilde{f}_r\Bigl(
(\pi_i;(\ch_j^{t,y-x,\eps},x+X_{\ch_j^{t,y-x,\eps}})_{j\in\N});x+X_T),
\\&\qquad\qquad
\qquad\qquad\qquad
\tilde{Y}^{\pi_i,t,x}_r,\tilde{Z}^{\pi_i,t,x}_r,
\int_{\R^d} \tilde{U}^{\pi_i,t,x}_r(z)\,\eta_r(z)\,
K_r(dz)
\Bigr)\,dr
\Bigr].
\end{align*}
If, additionally, $x\not\in O_\eps(y)$, then
\begin{align*}
\theta_i(\pi_i;t,x)=
\theta_{i+1}(\pi_i;t,x;t,x).
\end{align*}
\end{lemma}

\begin{proof}
Skip superscript $\eps$. Put $\tau:=\ch_1^{t,y-x}$.
Since
\begin{align*}
\theta_i(\pi_i;t,x)&=
\Mean_{t,0}\Bigl[
\tilde{Y}^{\pi_i,t,x}_\tau
+\int_t^\tau
\tilde{f}_r\Bigl(
(\pi_i;(\ch_j^{t,y-x,\eps},x+X_{\ch_j^{t,y-x,\eps}})_{j\in\N});x+X_T),
\\&\qquad\qquad
\qquad\qquad\qquad
\tilde{Y}^{\pi_i,t,x}_r,\tilde{Z}^{\pi_i,t,x}_r,
\int_{\R^d} \tilde{U}^{\pi_i,t,x}_r(z)\,\eta_r(z)\,
K_r(dz)
\Bigr)\,dr
\Bigr]
\end{align*}
it suffices to show that 
\begin{align}\label{E1:L:DPP}
\Mean_{t,\bfnull}[
\tilde{Y}^{\pi_i,t,x}_\tau]=
\Mean_{t,\bfnull}[
\theta_{i+1}(
\pi_i;\ch^{t,y-x,\eps}_1,x+X_{\ch^{t,y-x,\eps}_1};
\ch^{t,y,\eps}_1,x+X_{\ch^{t,y,\eps}_1})].
\end{align}
By Corollary~\ref{C:SMPwide},
\begin{align}\label{E2:L:DPP}
\Mean_{t,\bfnull}[
\tilde{Y}^{\pi_i,t,x}_\tau]&=
\int\int \tilde{Y}^{\pi_i,t,x}_{\tau(\om)}(\tiom)\,\Prob_{\tau,\om}(d\tiom)\,
\Prob_{t,\bfnull}(d\om).
\end{align}
For every $\om\in\Omega$, define a process $Y^\om$ on $[\tau(\om),T]$
by
\begin{align}\label{E3:L:DPP}
Y^\om_r:=\tilde{Y}^{\pi_i,t,x}(\om.\bfone_{[0,\tau(\om))}+
(X+\om_{\tau(\om)}).\bfone_{[\tau(\om),T]}).
\end{align}
Note that, by Lemma~\ref{L:ShiftingHitting-new}, 
for $\Prob_{\tau,\om}$-a.e.~$\tiom\in\Omega$,
\begin{align*}
&(\pi_i;(\ch_j^{t,y-x,\eps}(\tiom),x+X_{\ch_j^{t,y-x,\eps}}(\tiom))_{j\in\N},x+\tiom_T)\\
&\qquad =
(\pi_i;(\ch_j^{\tau(\om),\om_{\tau(\om)}}(\tiom),
x+X_{\ch_j^{\tau(\om),\om_{\tau(\om)}}}(\tiom))_{j\in\N_0},x+\tiom_T).
\end{align*}
Thus (cf.~Lemma~\ref{L:BSDE-ChangeMeasure} and 
Lemma~\ref{L:BSDE-DPP}), 
for $\Prob_{t,\bfnull}$-a.e.~$\om\in\Omega$,
there exists a pair $(Z^\om,U^\om)$
such that $(Y^\om,Z^\om,U^\om)$ is the solution to the BSDE
\begin{align*}
&Y^\om_r=g\left(\pi_i;\left(\ch_j^{\tau(\om),\om_{\tau(\om)}},
x+\om_{\tau(\om)}+X_{\ch_j^{\tau(\om),\om_{\tau(\om)}}}\right)_{j\in\N_0},
x+\om_{\tau(\om)}+X_T\right)\\
&\,+\int_r^T
\tilde{f}_{\tilde{r}}\Biggl(
\left(
\pi_i;\left(\ch_j^{\tau(\om),\om_{\tau(\om)}},
x+\om_{\tau(\om)}+X_{\ch_j^{\tau(\om),\om_{\tau(\om)}}}\right)_{j\in\N_0},
x+\om_{\tau(\om)}+X_T
\right),\\ &\,\qquad\qquad\qquad\qquad 
Y^{\om}_{\tilde{r}},Z^{\om}_{\tilde{r}},
\int_{\R^d} U^{\om}_{\tilde{r}}(z)\,\eta_{\tilde{r}}(z)\,K_{\tilde{r}}(dz)
\Biggr)\,d\tilde{r} \\ 
&\,+\int_r^T Z^{\om}_{\tilde{r}}\,dX^{c,\tau(\om),\bfnull}
+\int_r^T\int_{\R^d} U_{\tilde{r}}^{\om}(z)\,(\mu^X-\nu)(d\tilde{r},dz),\,
r\in [\tau(\om),T],\,\text{$\Prob_{\tau(\om),\bfnull}$-a.s.}
\end{align*}
Together with \eqref{E2:L:DPP} and \eqref{E3:L:DPP}, we obtain
\begin{align*}
\Mean_{t,\bfnull}\left[
\tilde{Y}^{\pi_i;t,x}_\tau
\right]&=
\int\int Y^{\om}_{\tau(\om)}(\tiom)\,
\Prob_{\tau(\om),\bfnull}(d\tiom)\,
\Prob_{t,\bfnull}(d\om)\\
& =\int\int 
Y^{
(\pi_i;\tau(\om),x+\om_{\tau(\om)}),
\tau(\om),x+\om_{\tau(\om)}
}_{\tau(\om)}(\tiom)\,
\Prob_{\tau(\om),\bfnull}(d\tiom)\,
\Prob_{t,\bfnull}(d\om)\\
& =\int
\theta_{i+1}(\pi_i;\tau(\om),x+\om_{\tau(\om)};
\tau(\om),x+\om_{\tau(\om)})\,
\Prob_{t,\bfnull}(d\om).
\end{align*}
Thus \eqref{E1:L:DPP} has been established.
\end{proof}

Fix $(s^\ast,\om^\ast)\in\bar{\Lambda}$
and $\eps\in(0,1)$. We will  write $g$ instead
of $g^{s^\ast,\om^\ast}$ and $\tilde{f}$ instead of
$\tilde{f}^{s^\ast,\om^\ast}$. 
The generic notation for an element of
$\Pi_i^{s^\ast}$ is
\begin{align*}
\pi_i=(s_0,y_0;\ldots;s_{i-1},y_{i-1}),\,
s=s_{i-1},\,y=y_{i-1}.
\end{align*}
Define a function
$h_i^{s^\ast,\om^\ast,\eps}=h_i^\eps:\Pi_i^{s^\ast}\times\R\times\R\to\R$ by
\begin{align*}
h_i^\eps(\pi_i;t,x):=
\theta_{i+1}^{s^\ast,\om^\ast,\eps} 
(\pi_i;(s\vee t)\wedge T,x;(s\vee t)\wedge T,x). 
\end{align*}
The following result is needed for the approximation of 
the functions $\theta_{i+1}(\pi_i;\cdot)$ by smooth functions.
\begin{lemma}\label{L:ContRestr}
Let
$\pi_i\in\Pi_i^{s^\ast}$. Then
 $(t,x)\mapsto h_i^\eps(\pi;t,x)$,
$\R\times\R^d\to\R$, is continuous.
\end{lemma}
\begin{remark}
If $\xi$ were only $d_{J_1}$-continuous, then,
in contrast to corresponding mappings in 
\cite{ETZ12a} and \cite{ETZ12b},
the mapping
$(y;t,x)\mapsto \xi(y.\bfone_{[0,t)}+
x.\bfone_{[t,T]})$
cannot be expected to be continuous.
For example, assume that $d=1$ and let
$(y^0;t^0,x^0):=(1;0,2)$ and
$(y^n;t^n,x^n):=(1;1/n,2)$.
Then $(y^n;t^n,x^n)\to(y^0;t^0,x^0)$
as $n\to\infty$, but
\begin{align*}
d_{J_1}
(y^n.\bfone_{[0,t^n)}+x^n.\bfone_{[t^n,T]},
y^0.\bfone_{[0,t^0)}+x^0.\bfone_{[t^0,T]})\ge 1.
\end{align*}
\end{remark}

\begin{remark}\label{R:notJ1}
If $\xi$ is just $d_{J_1}$-continuous,
then we cannot expect
$(t,x)\mapsto\theta_i^\eps(\pi_i;t,x)$
on $\overline{Q}^{2C_0^\prime}_{s,y}\setminus Q^\eps_{s,y}$
to be left-continuous in $t$. To
see this, assume that $d=1$ and let $t^0=T$ and
$t^n\uparrow t^0$ with $t^n<T$. Then,
given $\om\in\Omega$, we have, for sufficiently
large $n$ and with $y_j=0$,
$j=0$, $\ldots$, $i-1$, (which we can assume
without loss of generality,) 
\begin{align*}
&\vert
g(\pi_i;(t^n+[\ch_j^\eps(\om)\wedge (T-t^n)],
x+X_{\ch^\eps_j\wedge (T-t^n)}(\om))_{j\in\N_0};
x+X_{T-t^n}(\om))\\&\qquad\qquad-
g(\pi_i;(T,x+X_0(\om))_{j\in\N_0};x+X_0(\om))\vert\\
&\qquad =
\abs{\xi((x+\om_0).\bfone_{[t^n,T]})-
\xi((x+\om_0).\bfone_{\{T\}})}
\end{align*}
but
$d_{J_1}(z.\bfone_{[t^n,T]},z.\bfone_{\{T\}})\ge 1$
for $z\ne 0$. However
$d_{M_2}(z.\bfone_{[t^n,T]},z.\bfone_{\{T\}})\to 0$
as $n\to\infty$.
\end{remark}
\begin{proof}[Proof of Lemma~\ref{L:ContRestr}]
Let $(t^n,x_n)\to (t^0,x_0)$
in $\R\times\R^d$
as $n\to\infty$. For every $n\in\N$,
\begin{align*}
\abs{h_i^\eps(\pi_i;t^0,x_0)-
h_i^\eps(\pi_i;t^n,x_n)}&\le
\abs{h_i^\eps(\pi_i;t^0,x_0)-
h_i^\eps(\pi_i;t^0,x_n)}+\\ &\qquad
\abs{h_i^\eps(\pi_i;t^0,x_n)-
h_i^\eps(\pi_i;t^n,x_n)}\\
&=:A^x_n+A^t_n. 
\end{align*}
Note that, for every 
$(t,x)\in\R\times\R^d$,
with $t^\prime:=(t\vee s)\wedge T$,
\begin{equation}\label{E1:ContRestr}
\begin{split}
h_i^\eps(\pi_i;t,x)&=
\Mean_{t^\prime,\bfnull}\Biggl[
g\left(\pi_i;t^\prime,x;\left(\ch_j^{t^\prime,\eps},
x+X_{\ch_j^{t^\prime,\eps}}\right)_{j\in\N};x+X_T\right)\\
&\qquad\qquad +\int_{t^\prime}^T \tilde{f}_r
\Biggl(
\left(
\pi_i;t^\prime,x;\left(\ch_j^{t^\prime,\eps},
x+X_{\ch_j^{t^\prime,\eps}}\right)_{j\in\N};x+X_T
\right),\\
&\qquad\qquad\qquad
\tilde{Y}^{(\pi_i;t^\prime,x);t^\prime,x}_r,\tilde{Z}^{(\pi_i;t^\prime,x);t^\prime,x}_r,
\int_{\R^d} \tilde{U}^{(\pi_i;t^\prime,x);t^\prime,x}_r(z)\,\eta_r(z)\,K(dz)
\Biggr)\,dr
\Biggr]
\end{split}
\end{equation}
and also (cf.~Remark~\ref{R:hitting})
\begin{equation}\label{E2:ContRestr}
\begin{split}
&\Mean_{t^\prime,\bfnull}\left[
g\left(\pi_i;t^\prime,x;\left(\ch_j^{t^\prime,\eps},
x+X_{\ch_j^{t^\prime,\eps}}\right)_{j\in\N};x+X_T\right)\right]\\
&\qquad =
\Mean_{0,\bfnull}\left[
g\left(\pi_i;\left(t^\prime+[\ch_j^\eps\wedge (T-t^\prime)],
x+X_{\ch_j^\eps\wedge (T-t^\prime)}\right)_{j\in\N_0};
x+X_{T-t^\prime}\right)\right].
\end{split}
\end{equation}
as well as
\begin{equation}\label{E3:ContRestr}
\begin{split}
&\Mean_{t^\prime,\bfnull}\Biggl[
\int_{t^\prime}^T \tilde{f}_r
\Biggl(
\left(
\pi_i;t^\prime,x;\left(\ch_j^{t^\prime,\eps},
x+X_{\ch_j^{t^\prime,\eps}}\right)_{j\in\N};x+X_T
\right),\\
&\qquad\qquad\qquad
\tilde{Y}^{(\pi_i;t^\prime,x);t^\prime,x}_r,\tilde{Z}^{(\pi_i;t^\prime,x);t^\prime,x}_r,
\int_{\R^d} \tilde{U}^{(\pi_i;t^\prime,x);t^\prime,x}_r(z)\,\eta_r(z)\,K(dz)
\Biggr)\,dr
\Biggr]
\\
&\qquad =
\Mean_{0,\bfnull}\Biggl[
\int_0^{T-t^\prime} \tilde{f}_{r+t^\prime}
\Biggl(
\left(\pi_i;\left(t^\prime+[\ch_j^\eps\wedge (T-t^\prime)],
x+X_{\ch_j^\eps\wedge (T-t^\prime)}\right)_{j\in\N_0};
x+X_{T-t^\prime}\right),\\
&\qquad\qquad\qquad
\hat{Y}^{t^\prime,x}_r,\hat{Z}^{t^\prime,x}_r,
\int_{\R^d}\tilde{U}^{t^\prime,x}_r(z)\,\eta_r(z)\,K(dz)
\Biggr)\,dr
\Biggr],
\end{split}
\end{equation}
where $(\hat{Y}^{t^\prime,x},\hat{Z}^{t^\prime,x},\hat{U}^{t^\prime,x})$ 
is the solution of the BSDE
\begin{equation}\label{E4:ContRestr}
\begin{split}
&\hat{Y}^{t^\prime,x}_r=
g\left(\pi_i;\left(t^\prime+[\ch_j^\eps\wedge (T-t^\prime)],
x+X_{\ch_j^\eps\wedge (T-t^\prime)}\right)_{j\in\N_0};
x+X_{T-t^\prime}\right)\\
&\quad +\int_r^{T-t^\prime}\tilde{f}_{\tilde{r}+t^\prime}\Biggl(
\left(\pi_i;\left(t^\prime+[\ch_j^\eps\wedge (T-t^\prime)],
x+X_{\ch_j^\eps\wedge (T-t^\prime)}\right)_{j\in\N_0};
x+X_{T-t^\prime}\right),\\
&\quad\qquad
\qquad\qquad
\hat{Y}^{t^\prime,x}_{\tilde{r}},\hat{Z}^{t^\prime,x}_{\tilde{r}},
\int_{\R^d} \hat{U}^{t^\prime,x}_{\tilde{r}}(z)\,\eta_{\tilde{r}}(z)\,
K(dz)
\Biggr)\,d\tilde{r}\\&\quad
-\int_r^{T-t^\prime} \hat{Z}^{t^\prime,x}_{\tilde{r}}\,dX^{c,0,\bfnull}_{\tilde{r}}
-\int_r^{T-t^\prime}\int_{\R^d}\hat{U}^{t^\prime,x}_{\tilde{r}}(z)\,
(\mu^X-\nu)(d\tilde{r},dz),\,r\in [0,T-t^\prime],\,\text{$\Prob_{0,\bfnull}$-a.s.}
\end{split}
\end{equation}

Since $\xi$ is uniformly continuous
under $d_U$, since
$f$ is uniformly continuous under $\dist_\infty$ in $(t,\om)$ uniformly
in $(y,z,p)$, and since $d_{M_1}\le d_U$, one can show using
BSDE standard techniques 
(keeping \eqref{E1:ContRestr} in mind) 
that there exists a constant $C^\prime=C^\prime(t^0)>0$
such that 
\begin{align*}
A^x_n\le C^{\prime}\rho_0(\abs{x_0-x_n}).
\end{align*}

Put $s^n:=(t^n\vee s)\wedge T$.
To show convergence of $A^t_n$, let us
initially fix $\om\in\Omega$. Set
\begin{align*}
r_j&=r_j(\om):=\ch_j^\eps(\om),\,j\in\N_0,\\
\iota&=\iota(\om):=\max\{j\in\N_0:s^0+r_j\le T\}.
\end{align*}
We treat first the
case that $t^n\ge t^0$, whence $s^n\ge s^0$. Since
$\om$ is fixed, we can and will assume that,
without loss of generality,
\begin{align*}
s^n+r_\iota\le T.
\end{align*} 
Since $s^0\le T-r_\iota$, we have
$s_n\in[s^0,T-r_\iota]$. Since, for
every $n\in\N_0$ and $x\in\R$,
\begin{align*}
&g(\pi_i;(s^n+[r_j\wedge(T-s^n)],
\om_{r_j\wedge (T-s^n)})_{j\in\N_0};x+\om_{T-s^n})
\\ &\qquad=
\xi\Biggl(
\sum_{j=0}^{i-2} y_j.\bfone_{[s_j,s_{j+1})}
+y.\bfone_{[s,s^n)}
+\sum_{j=0}^{\iota-1}
(x+\om_{r_j}).\bfone_{[s^n+r_j,s^n+r_{j+1})}\\
&\qquad\qquad\qquad
+(x+\om_{r_\iota}).\bfone_{[s^n+r_\iota,T)}
+(x+\om_{T-s^n}).\bfone_{\{T\}}
\Biggr)\\
&\qquad=:\xi(\tiom(x,s^n)).
\end{align*}
we have, by Lemma~\ref{L:PathsCont},
\begin{align*}
\sup_x \abs{\xi(\tiom(x,s^n))-\xi(\tiom(x,s^0)}
&\le \sup_x \rho_0(d_{J_1}(\tiom(x,s^n),
\tiom(x,s^0)) \\
&\le \sup_x \rho_0(2(s^n-s^0)+
\abs{
\om_{T-s^n}-\om_{T-s^0}
}\to 0
\end{align*}
as $n\to\infty$ provided $\om$ is left-continuous at 
$T$, which, however is the case for
$\Prob_{0,\bfnull}$-a.e.~$\om$ because
$X$ is $\Prob_{0,\bfnull}$-quasi-left-continuous.
A corresponding result can be shown for the driver
$\tilde{f}(\cdots)$ of the BSDE \eqref{E4:ContRestr}.
Thus, keeping \eqref{E2:ContRestr} and \eqref{E3:ContRestr} in mind,
we can employ standard a-priori estimates for BSDEs
(cf.~Lemma~3.1.1 in \cite{DelongBook}) to deduce that,
for every $x\in\R^d$,
$h^\eps_i(\pi_i;t^n,x)\to
h^\eps_i(\pi_i;t^0,x)$ as $n\to\infty$ with $t^n\ge t$.
Hence, by Lemma~\ref{L:DCTpara},
$(t^n,x_n)\to (t^0,x_0)$  as $n\to\infty$ with $t^n\ge t^0$ implies
\begin{align*}
A^t_n\le\sup_x\abs{
h^\eps_i(\pi_i;t^n,x)-
h^\eps_i(\pi_i;t^0,x)
}\to 0
\end{align*}
as $n\to\infty$.

Now we treat the case $t^n\le t^0$. Again, fix
$\om=(\om^k)_{k\le d}\in\Omega$ and, in addition to the notation
introduced in the previous paragraph, set
\begin{align*}
\iota^n=\iota^n(\om):=
\max\{j\in\N_0: s^n+r_j\le T\}.
\end{align*}
Note that $\iota\le\iota^n$. Recall that
\begin{align*}
\tiom(x,s^0)&=(\tiom(x,s^0)^k)_{k\le d}=
\sum_{j=0}^{i-2} y_j.\bfone_{[s_j,s_{j+1})}
+y.\bfone_{[s,s^0)}
\\&\quad +\sum_{j=0}^{\iota-1}
(x+\om_{r_j}).\bfone_{[s^0+r_j,s^0+r_{j+1})}
+(x+\om_{r_\iota}).\bfone_{[s^0+r_\iota,T)}
+(x+\om_{T-s^0}).\bfone_{\{T\}}
\end{align*}
For any $n\in\N$, let
\begin{align*}
\bar{\om}(x,s^n)&=(\bar{\om}(x,s^n)^k)_{k\le d}:=
\sum_{j=0}^{i-2} y_j.\bfone_{[s_j,s_{j+1})}
+y.\bfone_{[s,s^n)}
\\&\quad+\sum_{j=0}^{\iota-1}
(x+\om_{r_j}).\bfone_{[s^n+r_j,s^n+r_{j+1})}\\
&\quad
+(x+\om_{r_\iota}).\bfone_{[s^n+r_\iota,
(s^n+r_{\iota+1})\wedge T)}
+\sum_{j=\iota+1}^{\iota^n-1}
(x+\om_{r_j}).\bfone_{
[s^n+r_j,s^n+r_{j+1})
}\\
&\quad +
\bfone_{\{\iota\}^c}(\iota^n).(x+\om_{r_{\iota^n}}).
\bfone_{[s^n+r_{\iota^n},T)}
+(x+\om_{T-s^n}).\bfone_{\{T\}}.
\end{align*}
Right-continuity of $\om$ yields 
$\om_{r_j}=\om_{r_j\wedge (T-s^n)}\to\om_{T-s^0}$
as $n\to\infty$ for $j\le\iota^n$.
Hence, since $d_{M_1}$ is a metric, by the triangle inequality 
together with Lemma~\ref{R:M1},
\begin{align*}
d_p(\bar{\om}(x,s^n),\tiom(x,s^0))
\le \max_{k\le d} d_{M_1}(\bar{\om}(x,s^n)^k,\tiom(x,s^0)^k)\to 0
\end{align*}
uniformly in $x$ as $n\to\infty$.
Thus, corresponding considerations
as in the previous paragraph yield
\begin{align*}
\sup_x \abs{\xi(\bar{\om}(x,s^n))-
\xi(\tiom(x,s^0))}\le 
\sup_x \rho_0(d_p
(\bar{\om}(x,s^n),\tiom(x,s^0))
\to 0
\end{align*}
and
\begin{align*}
\sup_x\abs{
h_i^\eps(\pi_i;t^n,x)-
h_i^\eps(\pi_i;t^0,x)}\to 0.
\end{align*}
as $n\to\infty$.

This concludes the proof.
\end{proof}
\subsection{Path-frozen integro-differential
equations}
Let $\eps\in (0,c_0^\prime)$.
Given $(s,y)\in [s^\ast,T]\times\R^d$,
let $K^{4C_0^\prime}_{s,y}:=[s,T]\times \prod_{i=1}^d 
[y^i-4C_0^\prime,y^i+4C_0^\prime]$.
Then, by the Weierstrass approximation
theorem, for any $\delta>0$,
there exists a polynomial
$h^{\eps,\delta}_i(\pi_i;\cdot)$ on $\R\times\R^d$
such that
\begin{align*}
\abs{
h^{\eps,\delta}_i(\pi_i;\cdot)-
h^{\eps}_i(\pi_i;\cdot)}_{K^{4C_0^\prime}_{s,y}}
<\delta.
\end{align*}
Since we can take multivariate Bernstein polynomials
as approximating functions 
(see, e.g., Appendix B of \cite{HeitzingerThesis}),
we can and will, by Assumption~\ref{A:pi}, assume
 that the mapping
$(\pi_i;t,x)\mapsto 
h^{\eps,\delta}_i(\pi_i;\cdot)$,
$\Pi_i^{s^\ast}\times\R\times\R^d\to\R$,
is continuous.
Put \begin{align*}
\overh_i^{\eps,\delta}
:=h_i^{\eps,\delta}+\delta.
\end{align*}

Now, we proceed similarly as in the approach of the proof of Lemma~5.4 
in the first arxiv version
of \cite{ETZ12b} to define inductively (in three steps) a functional
$\psi^{s^\ast,\om^\ast,\eps}\in \bar{C}_b^{1,2}(\bar{\Lambda}^{s^\ast})$,
which will have
 properties that are needed in the proof of Theorem~\ref{T:Comparison}
below.
To this end, let us, first of all, introduce
some notation.

\begin{definition}\label{D:AppendixC}
Let $0<\delta< c_0^\prime< C_0^\prime$
and $2C_0^\prime <\delta^\prime<\infty$.
(Recall
that $c_0^\prime$ is a lower bound and
$C_0^\prime$ is an upper bound of the jump size
of $X$.
See Remark~\ref{R:jumpSize} and Assumption~\ref{A:jumps}.)
Let $(t^\ast,y^\ast)\in (-\infty,T)\times\R^d$.
Put
$D:=O_\delta(y^\ast)$, $D^\prime:=O_{\delta^\prime}(y^\ast)$,
$Q:=(t^\ast,T)\times O_\delta(y^\ast)$,
$Q^\prime:=(t^\ast,T)\times O_{\delta^\prime}(y^\ast)$,
etc. 
 Fix $\alpha\in (0,1)$ and 
 $h\in C^\infty(\bar{Q})$.
 Set
\begin{align*}
\mathcal{C}_\alpha(h):=\{
w:\bar{Q}^\prime\to\R\quad&\text{such that
$w\in C^{2,\alpha}_{\loc}(Q)$}\\
&\text{with $\abs{\partial_t w}_Q$,
$\abs{\partial_x w}_Q$,
$\abs{\partial^2_{xx} w}_Q$ being bounded}\\
&\text{and that $w=h$ in $(t^\ast,T)\times(D^\prime\setminus D)$ and on
$\{T\}\times D^\prime$}
\}.
\end{align*}
Let $\check{f}=\check{f}(t,y,z,p):\bar{Q}\times\R\times\R^d\times\R\to\R$
be a function.
Define $\bar{\eta}:(-\infty,T]\to\R$ by $\bar{\eta}(t):=\eta_{t\vee 0}$.
Given $t\in (-\infty,T]$, define an operator $\mathbf{I}^h_t=\mathbf{I}_t$ on 
$\mathcal{C}_\alpha(h)$ 
by 
\begin{align*}
\mathbf{I}_t w(t,x):=\int_{c_0^\prime\le \abs{z}
\le C_0^\prime}
\left[h(t,x+z)\,\bar{\eta}(t)\right]\,K(dz)
-w(t,x)\bar{\eta}(t)K(\R^d).
\end{align*}
Define a mapping $\tilde{F}=\tilde{F}(t,x,y,z,w):\bar{Q}\times\R\times\R^d\times
\mathcal{C}_\alpha(h)\to\R$ by
\begin{align*}
\tilde{F}(t,x,y,z,w):=\check{f}(t,y,z,\mathbf{I}_t w(t,x)).
\end{align*}
Given $v\in\mathcal{C}_\alpha(h)$, put
\begin{align*}
\tilde{F}[v](t,x):=\tilde{F}(t,x,v(t,x),\partial_x v(t,x),v(\cdot,\cdot)).
\end{align*}
Define an operator $\mathbf{L}^h=\mathbf{L}$ on $\mathcal{C}_\alpha(h)$ by
\begin{align*}
\mathbf{L}w(t,x)&:=
-\partial_t w(t,x)-\sum_{i=1}^d \bar{b}^i)\,\partial_{x^i} w(t,x)-
\frac{1}{2}\sum_{i,j=1}^d \bar{c}^{i,j}\,\partial_{x^i x^j} w(t,x)\\
&\qquad
-\int_{\R^d}\left[
h(t,x+z)-\sum_{i=1}^d z^i\,\partial_{x^i} w(t,x)
\right]\,K(dz)\\
&\qquad +w(t,x)K(\R^d).
\end{align*}
\end{definition}

\begin{remark} Given the context of the preceding
definition, we have
\begin{align*}
\mathbf{I}^h_t w(t,x)&:=\int_{\R^d}
[w(t,x+z)-w(t,x)]\,\bar{\eta}(t)\,\bar{K}(t,dz),\\
\mathbf{L}^hw(t,x)&:=
-\partial_t w(t,x)-\sum_{i=1}^d b^i\,\partial_{x^i} w(t,x)-
\frac{1}{2}\sum_{i,j=1}^d c^{i,j}\,\partial_{x^i x^j} w(t,x)\\
&\qquad
-\int_{\R^d}\left[
w(t,x+z)-w(t,x)-\sum_{i=1}^d z^i\,\partial_{x^i} w(t,x)
\right]\,K(dz).
\end{align*}
\end{remark}

Given $y\in\R^d$
and $h\in C^\infty(\R\times\R^d)$,
 put $D_y:=O_\eps(y)$
 $D^\prime_y:=O_{2C_0^\prime}(y)$, and let 
the function space $\mathcal{C}_\alpha(h,y)$ be
defined as the space $\mathcal{C}_\alpha(h)$ with
$\bar{Q}^\prime=\bar{Q}^{2C_0^\prime}_{-1,y}$
and $Q=\bar{Q}^\eps_{-1,y}$.

\textit{Step~1.} Let $\pi_1=(s^\ast,y)$. 
Set $\delta_1:=\eps/4$.
Write $h=\overh_1^{\eps,\delta_1}(\pi_1;\cdot)$.
Write $\hat{f}(t,\cdot)=\tilde{f}_t((\pi_1;(T,y)_{j\in\N};y),\cdot)$
and let $\tilde{F}$ as well as $\tilde{F}[\cdot]$ be defined as in Definition~\ref{D:AppendixC}. By 
standard PDE theory,
there exists a function
$w_1^{s^\ast,\om^\ast,\eps}(\pi_1;\cdot)=
w_1(\pi_1;\cdot)=w_1\in\mathcal{C}_\alpha(h,y)$
such that
\begin{align*}
\mathbf{L}w_1-\tilde{F}[w_1]=0\,\,\text{in $Q^\eps_{-1}$, }
w_1=h\,\,\text{in $(-1,T)\times D^\prime_y\setminus D_y$, }
w_1=h\,\,\text{on $\{T\}\times D^\prime_y$.}
\end{align*}
Define a function
$v_1^{s^\ast,\om^\ast,\eps}(\pi_1;\cdot)=
v_1(\pi_1;\cdot)=v_1$ on $\bar{Q}^{2C_0^\prime}_{-1,y}$ by
\begin{align*}
v_1(\pi_1;t,x):=w_1(\pi_1;t,x)-w_1(\pi_1;\pi_1)+\theta_1(\pi_1;\pi_1)+\frac{\eps}{2}.
\end{align*}
Then $v_1(\pi_1;\cdot)\in\mathcal{C}_\alpha(h^\prime,y)$ 
for some $h^\prime\in C^\infty(\R\times\R^d)$
and
\begin{align}\label{E1:LStep1}
\mathbf{L} v_1-\tilde{F}[v_1]&\ge 0
 \,\,\text{in $Q^\eps_{-1,y}$,} 
\\ \label{E3:LStep1}
v_1(\pi_1;\pi_1)&=\theta_1(\pi_1;\pi_1)
+\frac{\eps}{2},
\\ \label{E2:LStep1}
\quad v_1&\ge h^\eps_1(\pi_1;\cdot)\,\,\text{in 
$[s^\ast,T)\times D^\prime_y\setminus D_y$ and
on  $\{T\}\times D^\prime_y$.}
\end{align}
To see that \eqref{E2:LStep1} is true, it suffices, by definition of $v_1$, to show
that $\eps/2+\theta_1(\pi_1;\pi_1)-w_1(\pi_1;\pi_1)\ge 0$. Indeed, noting
that $f$ is non-anticipating and using the dynamic programming principle
(Lemma~\ref{L:DPP}), we can employ the comparison principle for BSDEs
with jumps together with constant-translatibility
 of sublinear expectations 
 (see, e.g., \cite{PengGexpect}) to deduce that
$h\le h_i^{\eps,\delta}(\pi_1;\cdot)+2\delta_1$ implies
$w_1(\pi_1;\pi_1)\le \theta_1(\pi_1;\pi_1)+\eps/2$.
 Note that It\^o's formula together
 with quasi-left-continuity makes it possible
 to represent $w_1$ as BSDE
(see Lemmas~\ref{L:hittingLeftLimit} 
and \ref{L:ItoHitting}).

Define
\begin{align*}
\psi^{s^\ast,\om^\ast,\eps}(t,\om):=
v_1(s^\ast,\om_{s^\ast};
t,\om_t)+\sum_{j=1}^\infty \delta_j,
\quad s^\ast\le t\le\ch_1^{s^\ast,\eps}(\om).
\end{align*}
Then, by \eqref{E3:LStep1},
\begin{align*}
\theta_1^{s^\ast,\om^\ast,\eps}
(s^\ast,\om_{s^\ast};s^\ast,\om_{s^\ast})
<\psi^{s^\ast,\om^\ast,\eps}(s^\ast,\om)
<\theta_1^{s^\ast,\om^\ast,\eps}
(s^\ast,\om_{s^\ast};s^\ast,\om_{s^\ast})+\eps.
\end{align*}
Universal measurability of
 $(t,\om)\mapsto v_1(s^\ast,
\om_{s^\ast};t,\om_t)$
follows from Assumption~\ref{A:pi}.

\textit{Step~2.} 
Let $\pi_2=(s_0,y_0;s,y)\in\Pi_2^{s^\ast}$.
Set $\delta_2:=\eps/8$.
Write $h=\bar{h}^{\eps,\delta_2}_2(\pi_2;\cdot)$
and $\hat{f}(t,\cdot)=\tilde{f}_t((\pi_2;(T,y)_{j\in\N};y),\cdot)$
and let $\tilde{F}$ as well as $\tilde{F}[\cdot]$ be defined as in Definition~\ref{D:AppendixC}. By 
standard PDE theory,
there exists a function
$w_2^{s^\ast,\om^\ast,\eps}(\pi_2;\cdot)=
w_2(\pi_2;\cdot)=w_2\in\mathcal{C}_\alpha(h,y)$
such that
\begin{align*}
\mathbf{L}w_2-\tilde{F}[w_2]=0\,\,\text{in $Q^\eps_{-1}$, }
w_2=h\,\,\text{in $(-1,T)\times D^\prime_y\setminus D_y$, }
w_2=h\,\,\text{on $\{T\}\times D^\prime_y$.}
\end{align*}
Define a function
$v_2^{s^\ast,\om^\ast,\eps}(\pi_2;\cdot)=
v_2(\pi_2;\cdot)=v_2$ on $\bar{Q}^{2C_0^\prime}_{-1,y}$ by
\begin{align*}
v_2(\pi_2;t,x):=w_2(\pi_2;t,x)-w_2(\pi_2;s,y)+v_1(\pi_1;s,y)+\delta_1.
\end{align*}
Then $v_2
(\pi_2;\cdot)\in\mathcal{C}_\alpha(h^\prime,y)$
for some $h^\prime\in C^\infty(R\times\R^d)$
 and
\begin{align}\label{E1:LStep2}
\mathbf{L} v_2-\tilde{F}[v_2]&\ge 0
 \,\,\text{in $Q^\eps_{-1,y}$,} 
\\ \label{E3:LStep2}
v_2(\pi_2;s,y)&=v_1(\pi_1;s,y)+\delta_1,
\intertext{and, if $\eps\le\abs{y-y_0}\le 2C_0^\prime$, then}
\label{E2:LStep2}
\quad v_2&\ge h^\eps_2(\pi_2;\cdot)\,\,\text{in 
$[s,T)\times D^\prime_y\setminus D_y$ and
on  $\{T\}\times D^\prime_y$.}
\end{align}
To see that \eqref{E2:LStep2} is true, let
$(t,x)\in [s,T)\times D^\prime_y\setminus D_y$
or $(t,x)\in \{T\}\times D^\prime_y$.
Then, by \eqref{E2:LStep1} in Step~1,
\begin{align*}
v_2(\pi_2;t,x)&=h(t,x)+v_1(\pi_1;s,y)
-w_2(\pi_2;s,y)+\delta_1 \\
&\ge h_2^\eps(\pi_2;t,x)+v_1(\pi_1;s,y)
-w_2(\pi_2;s,y)+\delta_1\\
&= h_2^\eps(\pi_2;t,x)\\ &\qquad +
\bar{h}_1^{\eps,\delta_1}(s_0,y_0;s,y)+
\theta_1(s^\ast,y_0;s^\ast,y_0)
-w_1(s^\ast,y_0;s^\ast,y_0)+2\delta_1
\\ &\qquad-w_2(\pi_2;s,y)+\delta_1.
\end{align*}
That is, we have to show that
\begin{equation}\label{E5:LStep2}
\begin{split}
w_2(\pi_2;s,y)&\le
\bar{h}_1^{\eps,\delta_1}(s^\ast,y_0;s,y)
+\theta_1(s^\ast,y_0;s^\ast,y_0)
-w_1(s^\ast,y_0;s^\ast,y_0)
+3\delta_1.
\end{split}
\end{equation}
Note that, similarly as in Step~1, one can show that
$$w_2(\pi_2;s,y)\le\theta_2(\pi_2;s,y)+2\delta_2.$$
We also have $\theta_2(\pi_2;s,y)=h^\eps_1(s^\ast,y_0;s,y)$
because $\eps\le\abs{y-y_0}$. Thus 
\begin{align*}
w_2(\pi_2;s,y)\le \bar{h}_1^{\eps,\delta_1
}(s^\ast,y_0;s,y)+2\delta_2,
\end{align*}
and together with 
$2\delta_1\le
\theta_1(s^\ast,y_0;s^\ast,y_0)
-w_1(s^\ast,y_0;s^\ast,y_0)$
from Step~1 we get \eqref{E5:LStep2} and consequently
\eqref{E2:LStep2}.

Define
\begin{align*}
\psi^{s^\ast,\om^\ast,\eps}(t,\om)&:=
v_2^{s^\ast,\om^\ast,\eps}
(s^\ast,\om_{s^\ast};\ch_1^{s^\ast,\eps}(\om),
X_{\ch_1^{s^\ast,\eps}}(\om);
t,\om_t) \\ &\qquad\qquad
+\sum_{j=2}^\infty \delta_j,
\quad \ch_1^{s^\ast,\eps}(\om)< t\le
\ch_2^{s^\ast,\eps}(\om).
\end{align*}
Note that, by Step~1 and 
by definition of $v_2$,
\begin{align*}
\psi^{s^\ast,\om^\ast,\eps}
(\ch_1^{s^\ast,\eps},\om)=
v_2^{s^\ast,\om^\ast,\eps}
(s^\ast,\om_{s^\ast};
\ch_1^{s^\ast,\eps}(\om),
X_{\ch_1^{s^\ast,\eps}}(\om);
\ch_1^{s^\ast,\eps}(\om),
X_{\ch_1^{s^\ast,\eps}}(\om))
+\sum_{j=2}^\infty \delta_j.
\end{align*}
Universal measurability of
$$(t,\om)\mapsto v_2(s^\ast,\om_{s^\ast};
\ch_1^{s^\ast,\eps}(\om),
X_{\ch_1^{s^\ast,\eps}}(\om);
t,\om_t),$$
follows from Assumption~\ref{A:pi}
and standard BSDE error estimates.

\textit{Step~3 ($i\to{i+1}$).} Let $i\in\N$. 
Set $\delta_j:=\eps/2^{j+1}$, $j\in\N$.
For every 
$\pi_j=(s_0,y_0
;\ldots;s_{j-1},y_{j-1})\in\Pi_j^{s^\ast}$,
$j\in\N$, 
there exists, by standard PDE theory,
$w_j^{s^\ast,\om^\ast,\eps}(\pi_j;\cdot)
=w_j(\pi_j;\cdot)=w_j \in\mathcal{C}_\alpha
(\bar{h}^{\eps,\delta_j}_j(\pi_j;\cdot),y_{j-1})$
such that
\begin{align}
\mathbf{L}w_j-
\tilde{f}_t((\pi_j;(T,y_{j-1})_{k\in\N};y_{j-1}),w_j,\partial_x w_j,
\mathbf{I}_t w_j)&=0\,\,\text{in $Q^\eps_{-1,y_{j-1}}$, }\\
\label{Ew2:Step3}
w_j=\bar{h}^{\eps,\delta_j}_j(\pi_j;\cdot)\,\,\text{in $(-1,T)\times D^\prime_y\setminus D_{y_{j-1}}$, }
w_j&=\bar{h}^{\eps,\delta_j}_j(\pi_j;\cdot)\,\,\text{on $\{T\}\times D^\prime_{y_{j-1}}$.}
\end{align}
Define $v_j^{s^\ast,\om^\ast,\eps}(\pi_j;\cdot)
=v_j(\pi_j;\cdot)$ on $\bar{Q}_{-1,y_{j-1}}^{2C^\prime_0}$ 
recursively by
\begin{align}\label{Ev:Step3}
v_j(\pi_j;t,x):=w_j(\pi_j;t,x)+v_{j-1}(\pi_j)-w_j(\pi_j;s_{j-1},y_{j-1})+\delta_{j-1}.
\end{align}

Suppose that the following
induction hypothesis holds:
\begin{quotation}
For every $j\in\{1,\ldots,i\}$,
\begin{align*}
\mathbf{L}v_j-
\tilde{f}_t((\pi_j;(T,y_{j-1})_{k\in\N};y_{j-1}),v_j,\partial_x v_j,
\mathbf{I}_t v_j)&\ge 0\,\,\text{in $Q^\eps_{-1,y_{j-1}}$, }
\end{align*}
and, if 
$\eps\le\abs{y_{k+1}-y_k}\le 2$,
for every $k\in\{0,\ldots,j-1\}$, then
\begin{align*}
 v_j(\pi_j;\cdot)\ge h_j^\eps(\pi_j;\cdot)\quad
 \text{in $(s_{j-1},T)\times D^\prime_{y_{j-1}}\setminus D_{y_{j-1}}$
 and on $\{T\}\times D^\prime_{y_{j-1}}$.}
\end{align*}
\end{quotation}

Let $i\ge 2$. Fix $\pi_{i+1}=(s_0,y_0;\ldots;
s_i,y_i)\in\Pi^{s^\ast}_{i+1}$ with
 $s=s_i$ and $y=y_i$. Let
$\pi_j:=(s_0,y_0;\ldots;s_{j-1},y_{j-1})$, 
$j=1$, $\ldots$, $i-2$. Then $v_{i+1}(\pi_{i+1};\cdot)\in
\mathcal{C}_\alpha(h^\prime,y)$
for some $h^\prime\in C^\infty(\R\times\R^d)$
 and
\begin{align}\label{E1:LStep3}
\mathbf{L}v_{i+1}-
\tilde{f}_t((\pi_{i+1};(T,y_{})_{k\in\N};y_{}),v_{i+1},\partial_\om v_{i+1},
\mathbf{I}_t v_{i+1})&\ge 0\,\,\text{in $Q^\eps_{-1,y_{}}$, }
\\ \label{E3:LStep3}
v_{i+1}(\pi_{i+1};s,y)&=v_i(\pi_i)
+\delta_i,
\end{align}
and, if 
\begin{align}\label{E0:Step3} 
\eps\le \abs{y_{j+1}-y_j}\le 2,\quad
 j=0, \ldots, i-1,
\end{align}
then
\begin{align} \label{E2:LStep3}
\quad v_{i+1}&\ge 
h^\eps_{i+1}(\pi_{i+1};\cdot) 
&&\text{in $(s_{},T)\times D^\prime_{y_{}}\setminus D_{y_{}}$
 and on $\{T\}\times D^\prime_{y_{}}$.}
\end{align}
To see that \eqref{E2:LStep3} is true, let 
$(t,x)\in\bar{Q}^{2C_0^\prime}_{s,y}\setminus Q^\eps_{s,y}$.
Then
\begin{align*}
&v_{i+1}(\pi_{i+1};t,x) \\&\quad\ge 
h_{i+1}^\eps(\pi_{i+1};t,x)\\ &\quad\qquad
+v_i(\pi_i;s,y)-w_{i+1}(\pi_{i+1};s,y)
+\delta_i 
&&\text{ by \eqref{E0:Step3}}\\ &\quad \ge 
h_{i+1}^\eps(\pi_{i+1};t,x)\\ &\quad\qquad
+\bar{h}^{\eps,\delta_i}_i(\pi_i;s,y)
+\boxed{v_{i-1}(\pi_i)} \\ &\quad\qquad
-w_i(\pi_i;s_{i-1},y_{i-1})+\delta_{i-1}\\
&\quad\qquad
-w_{i+1}(\pi_{i+1};s,y)
+\delta_i  
&&\text{by \eqref{Ew2:Step3}, \eqref{Ev:Step3},
\eqref{E0:Step3}} \\ &\quad \ge
h_{i+1}^\eps(\pi_{i+1};t,x)\\ &\quad\qquad
+\bar{h}^{\eps,\delta_i}_i(\pi_i;s,y)
\\ &\quad\qquad
+\boxed{\bar{h}^{\eps,\delta_{i-1}}_{i-1}
(\pi_i)+v_{i-2}(\pi_{i-2})
-w_{i-1}(\pi_{i-1};s_{i-2},y_{i-2})
+\delta_{i-2} 
}
\\ &\quad\qquad
-w_i(\pi_i;s_{i-1},y_{i-1})+\delta_{i-1}\\
&\quad\qquad
-w_{i+1}(\pi_{i+1};s,y)
+\delta_i 
&&\text{by \eqref{Ew2:Step3}, \eqref{Ev:Step3},
\eqref{E0:Step3}} \\ &\quad =
h_{i+1}^\eps(\pi_{i+1};t,x)\\ &\quad\qquad
+\bar{h}^{\eps,\delta_i}_i(\pi_i;s,y)
+\bar{h}^{\eps,\delta_{i-1}}_{i-1}
(\pi_i) \\&\quad\qquad+
(\delta_i+\delta_{i-1}+\delta_{i-2})
\\&\quad\qquad
-(w_{i-1}(\pi_{i-1};s_{i-2},y_{i-2})
+w_i(\pi_i;s_{i-1};y_{i-1})
+w_{i+1}(\pi_{i+1};s,y)) \\ &\quad\qquad+
\boxed{v_{i-2}(\pi_{i-1})} \\ &\quad\ge\ldots
\\&\quad \ge 
h^\eps_{i+1}(\pi_{i+1};t,x) \\&\quad\qquad+
\left[
\bar{h}^{\eps,\delta_i}_i(\pi_{i+1})
+\ldots+
\bar{h}^{\eps,\delta_1}_1(\pi_2)+
\theta_1(\pi_1;\pi_1)
\right] \\&\quad\qquad+
\left[
(\delta_i+\ldots+\delta_1)+2\delta_1
\right]\\&\quad\qquad -
\left[
w_{i+1}(\pi_{i+1};s,y)+
w_i(\pi_i;s_{i-1},y_{i-1})+\ldots+w_1(\pi_1;s_0;y_0)
\right].
\end{align*}
I.e., we have to show that
\begin{equation}\label{E4:LStep3}
\begin{split}
&\left[
w_{i+1}(\pi_{i+1};s,y)+
w_i(\pi_i;s_{i-1},y_{i-1})+\ldots+w_1(\pi_1;s_0;y_0)
\right]\\
&\qquad\le
\left[
\bar{h}^{\eps,\delta_i}_i(\pi_{i+1})
+\ldots+
\bar{h}^{\eps,\delta_1}_1(\pi_2)+
\theta_1(\pi_1;\pi_1)
\right]\\
&\qquad\qquad
+\left[
(\delta_i+\ldots+\delta_1)+2\delta_1
\right].
\end{split}
\end{equation}
Again, similarly, as in Step~1, one can show that,
for every $j\in\{2,\ldots,i+1\}$, we have
$w_j(\pi_j;s_{j-1},y_{j-1})\le 
\theta_j(\pi_j;s_{j-1},y_{j-1})+2\delta_j$.
Also,  $\eps\le\abs{y_{j-1}-y_j}$,
$j=0$, $\ldots$, $i-1$, implies that,
for every $j\in\{2,\ldots,i+1\}$, we have
$\theta_j(\pi_j;s_{j-1},y_{j-1})=
h_{j-1}^\eps(\pi_{j-1};s_{j-1},y_{j-1})$, which yields
\begin{align*}
w_j(\pi_j;s_{j-1},y_{j-1})\le 
h_{j-1}^\eps(\pi_{j-1};s_{j-1},y_{j-1})+2\delta_j.
\end{align*}
Together with
$w_1(\pi_1;\pi_1)
\le \theta_1(\pi_1;\pi_1)+2\delta_1$,
from Step~1 and with 
$$2\delta_{i+1}+\ldots+2\delta_2=
(\delta_i+\ldots+\delta_1),$$
we get \eqref{E4:LStep3} and thus \eqref{E2:LStep3}.

Define
\begin{align*}
\psi^{s^\ast,\om^\ast,\eps}(t,\om):=
v_{i+1}(
(\ch_j^{s^\ast,\eps}(\om),
X_{\ch_j^{s^\ast,\eps}}(\om))_
{0\le j\le i};
t,\om_t)+\sum_{j=i+1}^\infty \delta_j,
\quad \ch_i^{s^\ast,\eps}(\om)< t\le
\ch_{i+1}^{s^\ast,\eps}(\om).
\end{align*}
Note that, by the induction hypothesis
by definition of $v_{i+1}$,
\begin{align*}
\psi^{s^\ast,\om^\ast,\eps}(\ch_i^{s^\ast,\eps},\om)=
v_{i+1}(
(\ch_j^{s^\ast,\eps}(\om),
X_{\ch_j^{s^\ast,\eps}}(\om))_
{0\le j\le i};
\ch_i^{s^\ast,\eps}(\om),
X_{\ch_i^{s^\ast,\eps}}(\om))+\sum_{j=i+1}^\infty \delta_j.
\end{align*}
As in Step~2,  universal measurability of 
\begin{align*}
(t,\om)\mapsto v_{i+1}(
(\ch_j^{s^\ast,\eps}(\om),
X_{\ch_j^{s^\ast,\eps}}(\om))_
{0\le j\le i};t,\om_t), 
\end{align*}
follows from Assumption~\ref{A:pi}
and standard BSDE error estimates.

By mathematical induction, we obtain the following
result.
\begin{lemma}
The mapping $\psi^{s^\ast,\om^\ast,\eps}:
\bar{\Lambda}\to\R$ defined in Step~1, Step~2,
and Step~3 belongs to
$\bar{C}_b^{1,2}(\bar{\Lambda})$.
\end{lemma}

\subsection{Proof of Comparison}

\begin{definition}
Let $(t,\om)\in\Lambda$.
Denote by $\overline{\mathcal{D}}(t,\om)$ 
(resp.~$\underline{\mathcal{D}}(t,\om)$) the set of all
$\varphi\in\bar{C}_b^{1,2}(\bar{\Lambda}^t)$ with corresponding
sequences $(\tau_n)$ of stopping times and corresponding collections
$(\vartheta_n)$ of functionals such that the following holds:
\begin{enumerate}
\renewcommand{\labelenumi}{(\roman{enumi})}
\item For every $n\in\N$ and every $(r,\tiom)\in\llb
\tau_{n-1},\tau_n\llb$, we have, 
with
$\pi_n=(\ch^{t,\eps}_i(\tiom),X_{\ch^{t,\eps}_i}(\tiom))_{0\le i\le n-1}$,
\begin{align*}
\mathbf{L}\vartheta_n(\pi_n;r,\tiom_r)-
f_r(\tiom,\vartheta_n(\pi_n;r,\tiom_r),
\partial_\om\vartheta_n(\pi_n;r,\tiom_r),
\mathbf{I}_r\vartheta_n(\pi_n;r,\tiom_r))\ge\,\,\text{(resp.~$\le$) $0$.}
\end{align*}
\item For $\Prob_{t,\om}$-a.e.~$\tiom\in\Omega$,
\begin{align*}
\varphi(T,\tiom)\ge\,\,\text{(resp.~$\le$) $\xi(\tiom)$.}
\end{align*}
\end{enumerate}
\end{definition}


\begin{proof}[Proof of Theorem~\ref{T:Comparison}]
\begin{align*}
Put
\overline{u}(t,\om):=\inf\{\varphi(t,\om):\varphi\in\overline{\mathcal{D}}(t,\om)\},\quad
\underline{u}(t,\om):=\sup\{\varphi(t,\om):\varphi\in\underline{\mathcal{D}}(t,\om)\}.
\end{align*}
We assert that $\overline{u}(t,\om)\le\underline{u}(t,\om)$. To show this,
we proceed nearly exactly as in the corresponding part of the proof of Proposition~7.5
in \cite{ETZ12a}. Define functionals $\overline{\psi}^{t,\om,\eps}$,
$\underline{\psi}^{t,\om,\eps}$ on $\bar{\Lambda}^t$ by
\begin{align*}
\overline{\psi}^{t,\om,\eps}_r:=\psi^{t,\om,\eps}_r+\rho(2\eps)[1+T-r],\quad
\underline{\psi}^{t,\om,\eps}_r:=\psi^{t,\om,\eps}_r-\rho(2\eps)[1+T-r].
\end{align*}
Note that $\overline{\psi}^{t,\om,\eps}$, $\underline{\psi}^{t,\om,\eps}
\in\bar{C}^{1,2}(\bar{\Lambda}^t)$ and the corresponding sequences of
stopping times are in both cases $(\ch^{t,\eps}_n)$ and the corresponding
collections of functionals are $(\overline{v}_n)$ and $(\underline{v}_n)$, resp.,
defined by 
\begin{align*}
\overline{v}_n(\cdot;r,\cdot):=v_n(\cdot;r,\cdot)+\rho(2\eps)[1+T-r],\quad
\underline{v}_n(\cdot;r,\cdot):=v_n(\cdot;r,\cdot)-\rho(2\eps)[1+T-r].
\end{align*}
Moreover, $\overline{\psi}^{t,\om,\eps}\in\overline{\mathcal{D}}(t,\om)$ because,
whenever $(r,\tiom)\in\llb\ch^{t,\eps}_{n-1},\ch^{t,\eps}_n \llb$ 
for some $n\in\N$, we have, with
$\pi_n=(\ch^{t,\eps}_i(\tiom),X_{\ch^{t,\eps}_i}(\tiom))_{0\le i\le n-1}$,
\begin{align*}
&\mathbf{L}\overline{v}_n(\pi_n;r,\tiom_r)-
f_r(\tiom,\overline{v}_n(\pi_n;r,\tiom_r),
\partial_x\overline{v}_n(\pi_n;r,\tiom_r),
\mathbf{I}_r\overline{v}_n(\pi_n;r,\tiom_r))\\
&\ge \mathbf{L}v_n(\pi_n;r,\tiom_r)
+\rho_0(2\eps)
-f_r(\tiom,v_n(\pi_n;r,\tiom_r),
\partial_x v_n(\pi_n;r,\tiom_r),
\mathbf{I}_r v_n(\pi_n;r,\tiom_r))\\
&\ge  \mathbf{L}v_n(\pi_n;r,\tiom_r)
-\tilde{f}_r((\pi_n;(T,\tiom_r)_{k\in\N};\tiom_r),\\ &\qquad\qquad
\qquad\qquad\qquad\qquad
v_n(\pi_n;r,\tiom_r),
\partial_x v_n(\pi_n;r,\tiom_r),
\mathbf{I}_r v_n(\pi_n;r,\tiom_r))\\&\ge 0
\end{align*}
and, similarly, $\overline{\psi}^{t,\om,\eps}_T\ge \xi$, 
$\Prob_{t,\om}$-a.s. Thus $\overline{u}(t,\om)\le\overline{\psi}^{t,\om,\eps}(t,\om)$
and, similarly, one can show that $\underline{\psi}^{t,\om,\eps}(t,\om)\le
\underline{u}(t,\om)$. Consequently,
$\overline{u}(t,\om)-\underline{u}(t,\om)\le 2\rho_0(2\eps)[1+T-t]$. Letting $\eps\downarrow 0$ yields $\overline{u}(t,\om)\le\underline{u}(t,\om)$.

Finally, by the partial comparison principle
(Theorem~\ref{T:PartialCompII}), $u^1(t,\om)\le\overline{u}(t,\om)$
and $\underline{u}(t,\om)\le u^2(t,\om)$. Our previous assertion yields then
$u^1(t,\om)\le u^2(t,\om)$.
\end{proof}

\appendix
\section{Martingale problems and 
regular conditioning}

The results in this appendix are actually valid in a more
general context than in our canonical setup and might be
of independent interest. In particular,
$(B,C,\nu)$ can be as general as in ¤III.2a.~of 
\cite{JacodShiryaevBook}, in which case standard conventions 
of \cite{JacodShiryaevBook} are in force.

First, we recall the definitions  of
\cite{SVBook} for conditional probability
distributions (c.p.d.) and regular
conditional probability distributions (r.c.p.d).
A c.p.d.~of a probability measure $\Prob$ on
$(\Omega,\cF_T^0)$ given a sub $\sigma$-field
$\cF\subseteq\cF_T^0$ is a collection
$\{\Prob_\om\}_{\om\in\Omega}$ of probability measures
satisfying the following:
\begin{enumerate}
\renewcommand{\labelenumi}{(\roman{enumi})}
\item For every $A\in\cF_T^0$, the map
$\om\mapsto\Prob_\om(A)$ is $\cF$-measurable.
\item For every $A\in\cF_T^0$ and every $B\in\cF$,
\begin{align*}
\Prob(A\cap B)=\int_A \Prob_\om(B)\,\Prob(d\om).
\end{align*}
\end{enumerate} 
If a c.p.d.~$\{\Prob_\om\}_{\om\in\Omega}$ given $\cF$ satisfies
$\Prob_{\om}(A(\om))=1$ for $\Prob$-a.e~$\om\in\Omega$, where
$A(\om):=\cap\{A\in\cF:x\in A\}$, then we call $\{\Prob_\om\}_{\om\in\Omega}$
an r.c.p.d.~given $\cF$.

The following two results are straight-forward
generalizations of Theorem~6.1.3 and
Theorem~6.2.2 in \cite{SVBook}. For unexplained notation 
we refer to \cite{JacodShiryaevBook}.
\begin{proposition}\label{T:shiftChar}
Let $X$ be an $(\Filt,\Prob)$-semimartingale
with characteristics $(B,C,\nu)$
after time $s\in [0,T]$,
 $\tau\in\mathcal{T}_s(\rawFilt)$
and $\{\Prob_\om\}_{\om\in\Omega}$ be a c.p.d.~of 
$\Prob$ given $\rawField{\tau}$.
Then, for $\Prob$-a.e.~$\om\in\Omega$,
the process $X$ is an $(\Filt,\Prob_\om)$-semimartingale
with characteristics 
$(p_{\tau(\om)}B,p_{\tau(\om)}C,p_{\tau(\om)}\nu)$
after time $\tau(\om)$.
\end{proposition}
\begin{proof}
By Theorem~II.2.2.1 in \cite{JacodShiryaevBook},
the processes
$X(h)-B-X_s$,
$M(h)^iM(h)^j-\tilde{C}^{ij}$, $i$, $j\le d$,
and $g\ast(p_s \mu^X)-g\ast\nu$,
$g\in\mathcal{C}^+(\R^d)$ (see
\cite{JacodShiryaevBook} for the definition
of $\mathcal{C}^+(\R^d)$),
are $(\Filt,\Prob)$-local martingales
 after time~$s$. 
 Hence, by Theorem~1.2.10 in \cite{SVBook}
(,which, after localization, is applicable
 by the same argument as Lemma~III.2.48
 in \cite{JacodShiryaevBook}
 in the proof of Theorem~III.2.40, p.~165,
 in \cite{JacodShiryaevBook}), there exists a 
 $\Prob$-null set $N\subset\Omega$ such that,
 for every $\om\in\Omega\setminus N$,
 the processes
 $X(h)-p_{\tau(\om)}B-X_{\tau(\om)}$,
 $M(h)^i M(h)^j-p_{\tau(\om)}\tilde{C}^{ij}
 -M(h)^i_{\tau(\om)}-M(h)^j_{\tau(\om)}$,
 $i$, $j\le d$,
 and $g\ast(p_{\tau(\om)}\mu^X)-
 g\ast(p_{\tau(\om)}\nu)$,
 $g\in\mathcal{C}^+(\R^d)$
 are local martingales.
 Hence, since the canonical decomposition of
 $X(h)$ after time $\tau(\om)$,
  $\om\in\Omega\setminus N$, is 
  \begin{align*}
  X(h)=X_{\tau(\om)}+M(h)-M(h)_{\tau(\om)}
  +B(h)-B(h)_{\tau(\om)},
  \end{align*}
 Theorem~II.2.21 in \cite{JacodShiryaevBook}
 concludes the proof. 
\end{proof}
\begin{cor}\label{C:strongMarkov}
Suppose that, for every 
$(s,\om)\in [0,T]\times\Omega$,
there exists a unique solution
$\Prob_{s,\om}$ of
the martingale problem for $(p_sB,p_sC,p_s\nu)$
starting at $(s,\om)$.
Then, for every $\tau\in\mathcal{T}_s(\rawFilt)$,
 the family 
$\{\Prob_{\tau(\tiom),\tiom}\}_{\tiom\in\Omega}$
is an r.c.p.d.~of $\Prob_{s,\om}$ given $\rawField{\tau}$.
\end{cor}
\begin{proof}
Let $\{\Prob_\tiom\}_{\tiom\in\Omega}$ be an r.c.p.d.~of
$\Prob_{s,\om}$ given $\rawField{\tau}$.
By Proposition~\ref{T:shiftChar}, for 
$\Prob_{s,\om}$-a.e.~$\tiom\in\Omega$,
$\Prob_\tiom$ is a solution of the martingale
problem for $$(p_{\tau(\tiom)}B,p_{\tau(\tiom)}C,
p_{\tau(\tiom)}\nu)$$
starting at $(\tau(\tiom),\tiom)$.
By uniqueness, 
$\Prob_\tiom=\Prob_{\tau(\tiom),\tiom}$.
\end{proof}

The next result is crucial.
\begin{theorem}[Proof communicated by
R.~Mikulevicius]\label{T:cpd_rightLim}
Let $\Prob$ be a probability measure on $(\Omega,\cF_T^0)$.
Let $\tau\in\mathcal{T}(\Filt)$.
Let $\{\Prob_\om\}_{\om\in\Omega}$
be a c.p.d.~of $\Prob$ given $\Field{\tau}$.
Then, for every $\om\in\Omega$,
\begin{align}\label{E:cpd_rightLim}
\Prob_\om(X_t=\om_t,\,0\le t\le\tau(\om))=1.
\end{align}
\end{theorem}
\begin{proof}
\textit{Step~1.} 
Fix a bounded 
$\cF^0_T\otimes\cF^0_{\tau+}$-measurable
function $H:\Omega\times\Omega\to\R$
and put $\bar{H}(\tiom):=H(\tiom,\tiom)$.
We claim that, for $\Prob$-a.e.~$\om\in\Omega$,
\begin{align}\label{E2:cpd_rightLim}
\Mean^\Prob [\bar{H}\vert\cF^0_{\tau+}]
(\om)=
\int H(\tiom,\om)\,\Prob_\om(d\tiom).
\end{align}
If $H$ is of the form 
$H(\tiom,\om)=G_1(\tiom)G_2(\om)$,
$G_1$ $\cF^0_T$-measurable,
$G_2$ $\cF^0_{\tau+}$-measurable, then,
for $\Prob$-a.e.~$\om\in\Omega$,
\begin{align*}
\Mean^\Prob[\bar{H}\vert\cF^0_{\tau+}](\om)&=
G_2(\om)\Mean^\Prob[G_1\vert\cF^0_{\tau+}](\om)
\\&=G_2(\om)\int G_1(\tiom)\,\Prob_\om(d\tiom)=
\int H(\om,\tiom)\,\Prob_\om(d\tiom).
\end{align*}
A monotone-class argument yields the claim.

\textit{Step~2.} Fix $t\in [0,T]$ and define 
$H:\Omega\times\Omega\to\R$ by
$H(\tiom,\om):=\bfone_{A}(\tiom,\om)$, where
\begin{align*}
A:=\{(\tiom,\om)\in\Omega\times\Omega:
\tiom_{t\wedge \tau(\om)}=
\om_{t\wedge\tau(\om)}\}.
\end{align*}
Since $H$ is
$\cF^0_T\otimes\cF^0_{\tau+}$-measurable
and $\bar{H}(\tiom)=1$, Step~1 yields that,
for $\Prob$-a.e.~$\om\in\Omega$,
\begin{align*}
1=\Mean^\Prob[\bar{H}\vert\cF^0_{\tau+}]
=\int \bfone_{A}(\tiom,\om)\,\Prob_\om(d\tiom)
=\Prob_\om(X_{t\wedge\tau(\om)}
=\om_{t\wedge\tau(\om)}). 
\end{align*}
Thus \eqref{E:cpd_rightLim} holds 
up to a $\Prob$-null set and on this null set
we can redefine $\Prob_\om$
 such that \eqref{E:cpd_rightLim} holds there, too
(cf.~p.~34 in \cite{SVBook}).
This concludes the proof.
\end{proof}
\begin{remark}
Note that the $\sigma$-field $\cF^0_{\tau+}$ is not
countably generated,
which yields non-existence of 
 r.c.p.d.~(see \cite{BlackwellDubins75},
where r.c.p.d.~in our sense are called
proper r.c.p.d.).
 Hence we cannot rely
on the corresponding proof in \cite{SVBook},
when $\cF^0_{\tau+}$ is replaced by
$\cF^0_\tau$ and $\tau\in\mathcal{T}(\F^0)$.
\end{remark}
The following result is an
adaption
of Lemma~2 in
 \cite{MikuleviciusPragarauskas92MGproblem}
 to our setting. Again, for unexplained notation, see \cite{JacodShiryaevBook}
 and also \cite{RogersWilliamsI}.
\begin{lemma}\label{L:MGproblem}
Let $(s,\om)\in\bar{\Lambda}$,
$\Prob$  be a solution of the
martingale problem for 
$(p_sB,p_s C,p_s\nu)$ starting
at $(s,\om)$, $\tau\in\mathcal{T}_s(\F^0_+)$,
and
$\{\Prob_\tiom\}_{\tiom\in\Omega}$ be
a c.p.d.~of $\Prob$ given $\cF^0_{\tau+}$.
Then, for $\Prob$-a.e.~$\tiom\in\Omega$,
the probability measure 
$\Prob_{\tiom}$ is a
solution of the martingale problem
for
$(p_{\tau(\tiom)}B,p_{\tau(\tiom)} C,
p_{\tau(\tiom)}\nu)$ starting
at $(\tau(\tiom),\tiom)$.
\end{lemma}

\begin{remark}\label{R:ShiftMg}

(i) Each $\F^0$-stopping time $\tau$ satisfies $\Prob_{\tau,\om} (\tau=\tau(\om))=1$ for every $\om$. This follows easily from Galmarino's
test (see \cite{DellacherieMeyer}).

(ii) Given a right-continuous ${\F}$-adapted process $Y$ such that $Y_t(\om)=Y_t(\om_{\cdot\wedge t})$
(this is sometimes nearly impossible to verify) and a closed  subset $E$ of $\R$,
the ${\F}$-stopping time
$\tau:=\inf\{t\ge 0:Y_t\in E\}\wedge T$
satisfies $\Prob_{\tau,\om} (\tau=\tau(\om))=1$. To see this, let 
$\om$ and $\tiom$ two paths that coincide on $[0,\tau(\om)]$.
First note that $\tau(\om)=T$ or, by right-continuity,  $Y_\tau(\om)=Y_{\tau(\om)}(\tiom)\in E$.
Moreover, if $0\le t<\tau(\om)$, then $Y_t(\om)=Y_t(\tiom)\not\in E$. Hence
 $\tau(\tiom)=\tau(\om)$.

(iii) If we assume that the set $E$ in the preceding paragraph is open instead of closed, then
a corresponding result does not necessarily hold. For example, let 
$T=2$, $\tau=\inf\{t\ge 0: |X_t|>1\}\wedge T$,
$\om\in\Omega$ be defined by $\om_t=t$, and $\tiom_t:=t.\bfone_{[0,1]}+(2-t).\bfone_{(1,T]}$.
Then $\tau(\om)=1$ but $\tau(\tiom)=T$.
\end{remark}
\begin{proof}[Proof of Lemma~\ref{L:MGproblem}]
First, note that, by Theorem~\ref{T:cpd_rightLim},
for $\Prob$-a.e.~$\tiom\in\Omega$, we have
$X_t=\tiom_t$, $0\le t\le \tau(\tiom)$.

In the next two steps, let $M=(M_t)_{t\ge s}$ be 
one of following processes:
\begin{align*}
&X(h)-p_sB-X_s,\\
&M(h)^iM(h)^j-p_s\tilde{C}^{ij},\, i,
j\le d,\\ 
&g\ast(p_s\mu^X)-g\ast(p_s\nu),\,
g\in\mathcal{C}^+(\R^d).
\end{align*}
Here, we can and will assume that
$\mathcal{C}^+(\R^d)$ is countable. 

\textit{Step~1.} 
 By Theorem~II.2.2.1 in \cite{JacodShiryaevBook}, 
$M$ is an $(\F^0_+,\Prob)$-local martingale.
Moreover, $M$ is $\F^0$-adapted.
Let $(\sigma_l)_l$ be a corresponding localizing
sequence of $\F^0$-stopping times
(cf.~the proof of Lemma~III.2.48 
in \cite{JacodShiryaevBook}).
Without loss of generality, let us assume that
$M^{\sigma_l}$ is bounded.
 Then,
for every $A\in\cF^0_{\tau+}$, $\l\in\N$, $r$, 
$r^\prime\in [0,T]$ with $r\le r^\prime$,
and $\eta\in\mathrm{b}\cF^0_r$,
we can apply the claim in Step~1 of the proof of 
Theorem~\ref{T:cpd_rightLim} to get
\begin{align*}
&\int_A \Mean^{\Prob_\tiom}
 [\eta (M^{\sigma_l}_{r^\prime\vee\tau(\tiom)}-
 M^{\sigma_l}_{r\vee\tau(\tiom)})]\,\Prob(d\tiom)\\
 &\qquad=\int_A \Mean^\Prob 
[\eta \Mean^\Prob[
M^{\sigma_l}_{r^\prime\vee\tau}
-M^{\sigma_l}_{r\vee\tau_k}
\vert \cF^0_{(r\vee\tau)+}
]\vert\cF^0_{\tau+}]\,\Prob(d\tiom)
=0.
\end{align*}

\textit{Step~ 2.} 
For every $n\in\N$, 
fix a countable dense
subset $J_n$ of $C^\infty_0(\R^{dn})$
 with respect to the
locally uniform topology.
For every $r\in [0,T]\cap(\Q\cup\{T\})$, denote
by $\Xi_r$ the set of all $\eta:\Omega\to\R$
of the form $\eta=f(X_{s_n},\ldots,X_{s_1})$
for some $n\in\N$, $s_1$, $\ldots$, 
$s_n\in [0,r]\cap(\Q\cup\{r\})$ with 
$s_1\le\ldots\le s_n$, and $f\in J_n$.
Put $\Xi:=\cap_r \Xi_r$.
Since $\Xi$ is countable, there exists, by
Step~1, a set
$\Omega_{M}\subset\Omega$ with
$\Prob(\Omega_{M})=1$ such that,
for every $l\in\N$,
 every $r$, $r^\prime\in [0,T]\cap(\Q\cup\{T\})$
with $r\le r^\prime$, every
$\eta\in\Xi_r$, and every $\tiom\in\Omega_{M}$,
\begin{align}\label{E2:MGproblem}
\Mean^{\Prob_\tiom}[\eta(
M^{\sigma_l}_{r^\prime\vee\tau(\tiom)}-
 M^{\sigma_l}_{r\vee\tau(\tiom)}
)]=0.
\end{align}
Since $\sigma(\Xi_r)=\cF^0_r$,
 \eqref{E2:MGproblem} holds
 also for every $\eta\in\mathrm{b}\cF^0_r$.
Right-continuity of $M$ implies that
$M^{\sigma_l}_{\tau(\tiom)\vee\cdot}$ is an 
$(\F^0_+,\Prob_\tiom)$-martingale after time
$\tau(\tiom)$. Hence $M$ is 
an $(\F^0_+,\Prob_\tiom)$-local martingale
after time  $\tau(\tiom)$.

\textit{Step~3}.
Since 
$\Prob(\cap_{M} \Omega_{M})=1$,
Step~2 and a second application of Theorem~II.2.2.1
in \cite{JacodShiryaevBook} conclude
the proof.
\end{proof}

\begin{proposition}\label{C:SMPwide}
For every $(s,\om)\in \bar{\Lambda}$,
$\eta\in\mathrm{b}\cF^{s,\om}_T$, and 
$\tau\in\mathcal{T}_s(\F^{s,\om})$,
\begin{align*}
\Mean_{\tau,X} [\eta]=
\Mean_{s,\om} [\eta\vert\cF^{s,\om}_\tau],
\quad\text{$\Prob_{s,\om}$-a.s.}
\end{align*}
\end{proposition}
\begin{proof}
Note that there exists
a $\ttau\in\mathcal{T}_s(\F^0_+)$ and
set $\Omega^\prime\subset\Omega$ such that
$\Prob_{s,\om}(\Omega^\prime)=1$ and
$\tau(\tiom)=\ttau(\tiom)$ for every
$\tiom\in\Omega^\prime$. 
Also note that there exists a
c.p.d.~$\{\Prob_{\tiom}\}_{\tiom\in\Omega}$
of $\Prob_{s,\om}$ given $\cF^0_{\ttau+}$,
which, by Theorem~\ref{T:cpd_rightLim},
satisfies
\begin{align*}
\Prob_{\tiom}(X_t=\tiom_t,\,
0\le t\le \ttau(\tiom))=1.
\end{align*}

\textit{Step 1.}
We will show that, for every $\eta\in\mathrm{b}\cF^0_T$,
$\Mean\left[\eta\vert\cF^{s,\om}_\tau\right]=
\Mean_{\ttau,X}\left[\eta\right]$, $\Prob$-a.s.

 First, we show 
that $\cF^{s,\om}_\tau=\cF^{s,\om}_\ttau$. To this end, let
$A\in\cF^{s,\om}_\tau$ and $t\in [s,T]$.
Then 
\begin{align*}
A\cap\{\ttau\le t\}=
[A\cap\{\tau\le t\}\cap\{\tau=\ttau\}]
\cup[A\cap\{\ttau\le t\}\cap\{\tau\neq\ttau\}]
\in\cF^{s,\om}_t.
\end{align*}
That is, $\cF^{s,\om}_\tau
\subseteq\cF^{s,\om}_\ttau$.
The other inclusion can be shown in the same way.
Consequently, using Lemma~\ref{L:MGproblem},
we have
\begin{align}\label{E:SMPwide}
\Mean\left[
\eta\vert\cF^{s,\om}_\tau
\right](\tiom)=
\Mean\left[
\eta\vert\cF^{s,\om}_\ttau
\right](\tiom)=^{\substack{(\ast)}}\,\,
\Mean\left[\eta\vert\cF^0_{\ttau+}\right](\tiom)
=\Mean^{\Prob_\tiom}[\eta]
=\Mean_{\ttau,\tiom}[\eta]
\end{align}
for $\Prob_{s,\om}$-a.e.~$\tiom\in\Omega$. Note that
the equality $(\ast)$ follows from the fact 
that first
for every $A\in\cF^{s,\om}_\ttau$
there exist, by Theorem~II.75.3 in
\cite{RogersWilliamsI},
sets $A^\prime\in\cF^0_{\ttau+}$
and $N\in \mathcal{N}_{s,\om}$
such that $A=A^\prime\cup N$ 
and then
\begin{align*}
\Mean_{s,\om}\left[
\bfone_A 
\Mean_{s,\om}\left[
\eta\vert\cF^{s,\om}_\ttau\right]
\right]&=
\Mean_{s,\om}\left[\bfone_A \eta\right]=
\Mean_{s,\om}\left[\bfone_{A^\prime}\eta\right]=
\Mean_{s,\om}\left[
\bfone_{A^\prime}\Mean_{s,\om}\left[
\eta\vert\cF^0_{\ttau+}
\right]\right]\\&=
\Mean_{s,\om}\left[\bfone_A\Mean_{s,\om}\left[
\eta\vert\cF^0_{\ttau+}
\right]\right]
\end{align*}
and that second $\Mean_{s,\om}\left[
\eta \vert\cF^0_{\ttau+}\right]$
is $\cF^{s,\om}_\ttau$-measurable.
Finally,
\begin{align*}
\Mean_{\tau,X} [\eta]=
\bfone_{\{\tau=\ttau\}}.\Mean_{\ttau,X}[\eta]+
\bfone_{\{\tau\neq\ttau\}}.\Mean_{\tau,X}[\eta],
\quad\text{$\Prob_{s,\om}$-a.s.}
\end{align*}
Thus, the map $\tiom\mapsto
\Mean_{\tau(\tiom),\tiom}[\eta]$
is $\cF^{s,\om}_\tau$-measurable and
$\Mean_{\tau,X}[\eta]=
\Mean_{s,\om}[\eta\vert\cF^{s,\om}_\tau]$,
$\Prob_{s,\om}$-a.s.

\textit{Step 2.} To finish the proof of the 
corollary, we can, without loss
of generality because of Step~1, assume
that $\eta=\bfone_A$ with $A\subset B\in\cF_T^0$
and $\Prob_{s,\om}(B)=0$. Using Step~1, we have
\begin{align*}
0\le\Mean_{\tau(\tiom),\tiom}\bfone_A\le 
\Mean_{\tau(\tiom),\tiom}\bfone_B=
\Mean_{s,\om}[\bfone_B\vert\cF^{s,\om}_\tau](\tiom)=
\Mean_{s,\om}[\bfone_A\vert\cF^{s,\om}_\tau](\tiom)=0
\end{align*}
for $\Prob_{s,\om}$-a.e.~$\tiom\in\Omega$.
\end{proof}
\begin{proposition}\label{L:shiftMG_augment}
Fix $(s,\om)\in\bar{\Lambda}$. 
Let $M$ be a bounded right-continuous
$\F^0_+$-adapted
$(\F^{s,\om},\Prob_{s,\om})$-martingale.
Let $\tau\in\mathcal{T}_s(\F^{s,\om})$.
Then, for $\Prob_{s,\om}$-a.e.~$\tiom\in\Omega$,
$M$ is an 
$(\F^{\tau,\tiom},\Prob_{\tau,\tiom})$-martingale
after time $\tau(\tiom)$. A corresponding
result holds for sub- and supermartingales.
\end{proposition}

\begin{proof}
Let $\ttau\in\mathcal{T}_s(\F^0_+)$ satisfy
$\ttau=\tau$, $\Prob_{s,\om}$-a.s.
For every $n\in\N$, fix a countable dense
subset $J_n$ of $C_b(\R^n)$ with respect to the
locally uniform topology.
For every $r\in [0,T]\cap(\Q\cup\{T\})$, denote
by $\Xi_r$ the set of all $\eta:\Omega\to\R$
of the form $\eta=f(X_{s_n},\ldots,X_{s_1})$
for some $n\in\N$, $s_1$, $\ldots$, 
$s_n\in [0,r]\cap(\Q\cup\{r\})$ with 
$s_1\le\ldots\le s_n$, and $f\in J_n$.
Put $\Xi:=\cap_r \Xi_r$.
Since $\Xi$ is countable, there exists, by
Step~1 of the proof of
Proposition~\ref{C:SMPwide}, a set
$\Omega^\prime\subset\Omega$ with
$\Prob(\Omega^\prime)=1$ 
 such that
for every $r$, $r^\prime\in [0,T]\cap(\Q\cup\{T\})$
with $r\le r^\prime$, every
$\eta\in\Xi_r$, and every $\tiom\in\Omega^\prime$,
\begin{align*}
\Mean_{\tau,\tiom}\left[
\eta M_{r^\prime\vee \tau(\tiom)}
\right]&=
\Mean_{\ttau,\tiom}\left[
\eta M_{r^\prime\vee \ttau(\tiom)}
\right] \\
&=\Mean_{s,\om}\left[
\eta M_{r^\prime\vee\ttau}\vert\cF^{0}_{\ttau+}
\right](\tiom) &&\text{by
 \eqref{E2:cpd_rightLim}}\\
&=\Mean_{s,\om}\left[
\eta M_{r^\prime\vee\ttau}\vert\cF^{s,\om}_{\ttau}
\right](\tiom) &&\text{by
 \eqref{E:SMPwide}}\\
 &=\Mean_{s,\om}\left[
 \eta\Mean_{s,\om}\left[
 M_{r^\prime\vee\ttau}\vert
 \cF^{s,\om}_{r\vee\ttau}
 \right]\vert\cF^{s,\om}_{\ttau}
 \right](\tiom)&&\text{since
  $\eta\in\mathrm{b}\cF^0_r\subset
 \mathrm{b}\cF_{r\vee\ttau}$
 }\\
 &=\Mean_{s,\om}\left[
 \eta M_{r\vee\ttau}\vert\cF^{s,\om}_{\ttau}
 \right](\tiom)\\
 &=\Mean_{s,\om}\left[
 \eta M_{r\vee\ttau}\vert\cF^{0}_{\ttau+}
 \right](\tiom) &&\text{by \eqref{E:SMPwide}}\\
 &=\Mean_{\ttau,\tiom}\left[
 \eta M_{r\vee\ttau(\tiom)}
 \right] &&\text{by
 \eqref{E2:cpd_rightLim}
 }\\
 &=\Mean_{\tau,\tiom}
 \left[\eta M_{r\vee\tau(\tiom)}
 \right]. 
\end{align*}
Since $\cF^0_r=\sigma(\Xi_r)$,
$\Mean_{\tau,\tiom}
[\eta M_{r^\prime\vee\tau(\tiom)}]=
\Mean_{\tau,\tiom}[\eta M_{r\vee\tau(\tiom)}]$
for every $\eta\in\mathrm{b}\cF^0_r$
hence, also for every
$\eta\in\mathrm{b}\cF^{\tau,\tiom}_r$,
because, by Proposition~III.4.32 in
\cite{JacodShiryaevBook}, we have
$\cF^0_r\vee\mathcal{N}_{\tau,\tiom}=
\cF^{\tau,\tiom}_r$.
Next, let $0\le \hat{s}\le s^\prime<T$,
and let 
$(r_n)$ and $(r^\prime_n)$ be sequences
in $[0,T]\cap\Q$ with $\hat{s}=\inf_n r_n$,
 $s^\prime=\inf_n r^\prime_n$, 
 and $r_n\le r^\prime_n$.
 Then, for every
 $k\in\N$, every $\tiom\in\Omega^\prime$,
 and every 
 $\eta\in\mathrm{b}\cF^{\tau,\tiom}_{\hat{s}}$,
 we have, by right-continuity of $M$ and
 the dominated convergence theorem,
 \begin{align*}
 \Mean_{\tau,\tiom}\left[
 \eta M_{s^\prime\vee\tau(\tiom)}
 \right]&=
 \lim_n \Mean_{\tau,\tiom}\left[
 \eta M_{r^\prime_n\vee\tau(\tiom)}\right]\\&=
  \lim_n \Mean_{\tau,\tiom}\left[
 \eta M_{r_n\vee\tau(\tiom)}\right]=
 \Mean_{\tau,\tiom}\left[
 \eta M_{\hat{s}\vee\tau(\tiom)}
 \right].
 \end{align*}
 Since $M$ is 
 $\F^0_+$-adapted, it is 
 also $\F^{\tau,\tiom}$-adapted 
 and thus $M_{\cdot\vee\tau(\tiom)}$ is an
 $(\F^{\tau,\tiom},\Prob_{\tau,\tiom})$-martingale 
 after time $\tau(\tiom)$. 
\end{proof}

\section{Skorohod's  topologies}\label{S:Skorohod}
In this appendix,
we  recall some definitions
and basic results from 
\cite{WhittBook}. Put $\D:=\D([0,T],\R)$.
The \emph{completed graph} of a
path $\om\in\D$
is defined as the set
\begin{align*}
\Gamma_\om:=\{
(t,x)\in [0,T]\times\R:\exists \alpha\in [0,1]:
x=\alpha\om_{t-}+(1-\alpha)\om_t
\},
\end{align*}
where $\om_{0-}:=\om_0$.
We equip $\Gamma_\om$ with a linear order $\le$
defined as follows: Given $(t,x)$, 
$(t^\prime,x^\prime)\in\Gamma_\om$, we write
$(t,x)\le (t^\prime,x^\prime)$ if either 
$t<t^\prime$ or both, $t=t^\prime$ as
well as $\abs{x-\om_{t-}}\le\abs{x^\prime-\om_{t-}}$
hold. A \emph{parametric representation} of
$\om$ is a mapping 
$(r,z):[0,1]\to\Gamma_\om$ that is continuous,
nondecreasing, and surjective. The
set of all parametric representations of $\om$
is denoted by $\Pi(\om)$. Define  
 $d_{M_1}:\D\times\D\to\R_+$ by
\begin{align*}
d_{M_1}(\om,\om^\prime):=
\inf_{
\substack{
(r,z)\in\Pi(\om)\\(r^\prime,z^\prime)
\in\Pi(\om^\prime)
}
}
\norm{r-r^\prime}_\infty\vee
\norm{z-z^\prime}_\infty.
\end{align*}
Note that $d_{M_1}$ is a metric (Theorem~12.3.1 in \cite{WhittBook}).

\begin{lemma}\label{R:M1}
Let  $0\le t^n\le t^0\le T$.
Put $\om^n=\bfone_{[t^n,T]}$, 
$\om=\bfone_{[t^0,T]}\in \D$.
Then $d_{M_1}(\om,\om^n)\le t^0-t^n$.
\end{lemma}
\begin{proof}
Fix $0<a<b<1$.
We distinguish between the cases $t^0<T$ and
$t^0=T$.

(i) Without loss of generality, let $T=2$, 
$t^0=1$, and $t^n=1-n^{-1}$. Then
$\om=\bfone_{[1,2]}$
and $\om^n=\bfone_{[1-n^{-1},2]}$.
Define $(r,z)\in\Pi(\om)$ and 
$(r^n,z^n)\in\Pi(\om^n)$ by
\begin{align*}
r(t)&:=\frac{t}{a}\bfone_{[0,a]}(t)
+\bfone_{(a,b]}(t)+
\left(
\frac{1-t}{1-b}\cdot 1+\frac{t-b}{1-b}\cdot 2
\right)\bfone_{(b,1]}(t),\\
r^n(t)&:=\frac{t}{a}(1-n^{-1})\bfone_{[0,a]}(t)
+(1-n^{-1})\bfone_{(a,b]}(t)
\\ &\qquad\qquad+\left(
\frac{1-t}{1-b}\cdot(1-n^{-1})+
\frac{t-b}{1-b}\cdot 2
\right)\bfone_{(b,1]}(t),\\
z(t)=z^n(t)&:=
\frac{t-a}{b-a}\bfone_{[a,b]}(t)+\bfone_{(b,1]}(t).
\end{align*}
Then $\norm{r-r^n}_\infty=n^{-1}$
and $\norm{z-z^n}_\infty=0$. Thus
$d_{M_1}(\om,\om^n)\le n^{-1}$.

(ii) Without loss of generality, let $T=t^0=1$
and $t^n=1-n^{-1}$. Then $\om=\bfone_{\{T\}}$
and $\om^n=\bfone_{[1-n^{-1},1]}$. 
Define $(r,z)\in\Pi(\om)$ and 
$(r^n,z^n)\in\Pi(\om^n)$ by
\begin{align*}
r(t)&:=\frac{t}{a}\bfone_{[0,a]}(t)+
\bfone_{(a,1]}(t),\\
r^n(t)&:=\frac{t}{a}(1-n^{-1})\bfone_{[0,a]}(t)+
(1-n^{-1})\bfone_{(a,b]}(t)\\&\qquad\qquad+
\left(
\frac{1-t}{1-b}\cdot(1-n^{-1})+
\frac{t-b}{1-b}\cdot 1
\right)\bfone_{(b,1]}(t),\\
z(t)=z^n(t)&:=
\frac{t-a}{b-a}\bfone_{[a,b]}(t)+
\bfone_{(b,1]}(t).
\end{align*}
Then $\norm{r-r^n}_\infty=n^{-1}$,
$\norm{z-z^n}_\infty=0$ and consequently
$d_{M_1}(\om,\om^n)\le n^{-1}$.
\end{proof}

The  
$d_{M_2}$-metric on $\mathbb{D}$ is defined by,
\begin{align*}
d_{M_2}(\om,\tiom):=
m_H(\Gamma_\om,\Gamma_{\tiom}),
\end{align*}
where $m_H$
is the \emph{Hausdorff distance}, that is,
given closed sets $A$, $B\subseteq [0,T]\times\R$,
\begin{align*}
m_H(A,B):=
\left[
\sup_{a\in A}\inf_{b\in B} \abs{a-b}
\right]\vee
\left[
\sup_{b\in B}\inf_{a\in A} \abs{b-a}
\right].
\end{align*}
The uniform topology on $\mathbb{D}$ and
the (usual) Skorohod $J_1$-topology are induced by the following metrics:
\begin{align*}
d_U(\om,\tiom)&:=\norm{\om-\tiom}_\infty,\\
d_{J_1}(\om,\tiom)&:=
\inf_\lambda \sup_{[s\in [0,T]}
\abs{\lambda(s)-s}\vee\abs{\om_{\lambda(s)}-
\om^\prime_s},
\end{align*}
 where $\lambda$ runs through the set of all strictly
increasing continuous functions from $[0,T]$ to $[0,T]$
satisfying $\lambda(0)=0$
and  $\lambda(T)=T$.

\begin{remark}
We have 
$d_{M_1}\le d_{J_1}\le d_U$ (Theorem~12.3.2 in \cite{WhittBook}).
\end{remark}

\section{Auxiliary results}
\begin{lemma}\label{L:XjumpsWhenUjumps}
Let $u\in C(\bar{\Lambda})$.
Then $\Delta u_t>0$ implies
$\Delta X_t>0$.
\end{lemma}
\begin{proof}
Fix $(t,\om)\in\Lambda$ and let $c:=\abs{\Delta u(t,w)}>0$. Then,
there exists an $m_0$ such that, for every $m\ge m_0$,
\begin{align*}
\abs{u(t,\om)-u(t-m^{-1},\om)}\ge c/2.
\end{align*}
Since $u\in C^0(\Lambda)$, there exists a $\delta=\delta(t,\om)>0$
(independent from $m$) such that, for every
$(t^\prime,\om^\prime)\in\Lambda$,
\begin{align*}
\dist_\infty((t,\om),(t^\prime,\om^\prime))<\delta\Rightarrow
\abs{u(t,\om)-u(t^\prime,\om^\prime)}<c/2.
\end{align*}
Let $t^\prime=t-m^{-1}$, $\om^\prime=\om$. Then
\begin{align*}
\dist_\infty((t,\om),(t^\prime,\om^\prime))&=
m^{-1}+\sup_{s\in[0,T]}\abs{\om_{s\wedge t}-\om_{s\wedge(t-m^{-1})}}\\
&=m^{-1}+\sup_{s\in[t-m^{-1},t]}\abs{\om_s-\om_{t-m^{-1}}}\\&>\delta
\end{align*}
if $m\ge m_0$. Let $m_1\ge m_0$ be large enough so that
$m^{-1}<\delta/2$ whenever $m\ge m_1$. Now, let $m\ge m_1$. Then
\begin{align*}
\sup_{s\in[t-m^{-1},t]}\abs{\om_s-\om_{t-m^{-1}}}>\delta/2.
\end{align*}
Letting $m\to\infty$ yields
$\abs{\Delta X_t(\om)}\ge \delta/2>0$.
\end{proof}

\begin{lemma}\label{L:DCTpara}
Let $(S,\frak{S},\mu)$ be a finite
measure space, $(P,\frak{P})$ be a measurable space,
and $\{f_n\}_{n\in\N_0}$ be a uniformly bounded
family of measurable functions
from $S\times P$ to $[0,\infty)$ such that,
for $\mu$-a.e.~$s\in S$,
\begin{align*}
\sup_{p\in P}
\abs{f_n(s,p)-f(s,p)}\to 0
\end{align*}
as $n\to\infty$. Then
\begin{align*}
\sup_{p\in P} \int
\abs{f_n(s,p)-f(s,p)}\,\mu(ds) \to 0
\end{align*}
as $n\to\infty$.
\end{lemma}

\begin{proof}
The proof follows the lines of a standard proof of
the Dominated Convergence Theorem. By Fatou's lemma,
\begin{align*}
&\limsup_{n\to\infty}
\sup_{p\in P}\int
\abs{f_n(s,p)-f(s,p)}\,\mu(ds)\\
&\qquad\le 
\limsup_{n\to\infty} \int
\esssup^\mu_{p\in P}  
\abs{f_n(s,p)-f(s,p)}\,\mu(ds)\\
&\qquad\le
\int \limsup_{n\to\infty}\esssup^\mu_{p\in P} 
\abs{f_n(s,p)-f(s,p)}\,\mu(ds)\\
&\qquad\le 
\int \limsup_{n\to\infty} \sup_{p\in P} 
\abs{f_n(s,p)-f(s,p)}\,\mu(ds)\\
&\qquad = 0.
\end{align*}
This concludes the proof.
\end{proof}

\begin{lemma}\label{L:PathsCont}
Let $t$, $t^n\in [0,T]$ with
$t\le t^n$. Let $\iota\in\N$,
$0=r_0<r_1<\ldots<r_\iota<T-t$
and $z_j\in\R^d$, $j=0$, $\ldots$,
$\iota$. Let $t^n+r_\iota<T$. Consider two paths
$\om$, $\om^n\in\Omega$ defined by
\begin{align*}
\om&:=\sum_{j=0}^{\iota-1}
z_j.\bfone_{[t+r_j,t+r_{j+1})}
+z_\iota.\bfone_{[t+r_\iota,T)},\\
\om^n&:=\sum_{j=0}^{\iota-1}
z_j.\bfone_{[t^n+r_j,t^n+r_{j+1})}
+z_\iota.\bfone_{[t^n+r_\iota,T)}.
\end{align*}
Then $d_{J_1}(\om^n,\om)\le 2(t^n-t)$.
\end{lemma}

\begin{proof}
Let $\lambda_n:[0,T]\to [0,T]$ be a strictly 
increasing and continuous function with
$\lambda_n(0)=0$,
$\lambda_n(t+r_j)=t^n+r_j$,
$j=0$, $\ldots$, $\iota$,
$\lambda_n(T)=T$, and
$\norm{\lambda_n-\mathrm{id}}_\infty\le 2(t_n-t)$.
Then, given $s\in [0,T)$, we have
$\om_s-\om^n_{\lambda_n(s)}=0$,
given $s\in[t+r_j,t+r_{j+1})$, 
$j=0$, $\ldots$, $\iota-1$, we have
\begin{align*}
\om_s-\om^n_{\lambda_n(s)}=
z_j-\om^n_{\lambda_n(t+r_j)}=
z_j-\om^n_{t^n+r_j}=0,
\end{align*}
and, given $s\in[t+r_\iota,T)$, we have
\begin{align*}
\om_s-\om^n_{\lambda_n(s)}=z_\iota-
\om^n_{\lambda_n(t+r_\iota)}=z_\iota-
\om^n_{t^n+r_\iota}=0.
\end{align*}
This concludes the proof.
\end{proof}

\begin{lemma}\label{L:SDE-DPP}
Fix $(t,\om)\in\bar{\Lambda}$ and $s\in [t,T]$.
For $\Prob_{t,\om}$-a.e.~$\tiom\in\Omega$,
\begin{align*}
dX_r^{c,s,\tiom}=dX_r^{c,t,\om},\,
s\le r\le T,\,\text{$\Prob_{s,\tiom}$-a.s.}
\end{align*}
\end{lemma}
\begin{proof}
Note that
\begin{align*}
X&=X_t+(B-B_t)+X^{c,t,\om}+z\ast(\mu-\nu)
&&\text{on $[t,T]$, $\Prob_{t,\om}$-a.s.,}\\
X&=X_s+(B-B_s)+X^{c,s,\tiom}+z\ast(\mu-\nu)
&&\text{on $[s,T]$, $\Prob_{s,\tiom}$-a.s.}
\end{align*}

Define a process $V$ on $[t,T]$ by
 $V:=X-B-z\ast(\mu-\nu)$ and put
\begin{align*}
\Omega^\prime:=\{
\om^\prime\in\Omega:
X_r^{c,t,\om}(\om^\prime)=
(V-X_t+B_t)(\om^\prime),\, t\le r\le T
\}.
\end{align*}
Then $\Prob_{t,\om}(\Omega^\prime)=1$.
Let $\tiom\in\Omega^\prime$. Put
\begin{align*}
\Omega^{\prime\prime}:=\{
\om^{\prime\prime}\in\Omega^\prime:
X_r^{c,s,\tiom}(\om^{\prime\prime})=
(V-X_s+B_s)(\om^{\prime\prime}),\, s\le r\le T
\}.
\end{align*}
Then $\Prob_{s,\tiom}(\Omega^{\prime\prime})=1$
and, for every $(r,\om^{\prime\prime})\in 
[s,T]\times
\Omega^{\prime\prime}$,
\begin{align*}
X^{c,s,\tiom}_r(\om^{\prime\prime})&=
(V_r-X_t+B_t)(\om^{\prime\prime})
+(X_t-X_s+B_s-B_t)(\om^{\prime\prime})\\
&=X_r^{c,t,\om}(\om^{\prime\prime})+
(X_t-X_s+B_s-B_t)(\om^{\prime\prime}),
\end{align*}
which concludes the proof.
\end{proof}

\begin{lemma}\label{L:IntegrByParts}
Fix $(s,\om)\in\bar{\Lambda}$. 
For $i=1$, $2$, 
let $S_s^i\in\R$,
 $\beta^i$ be a predictable
 process on $[s,T]$,
$H^i\in L^2_{\mathrm{loc}}(X^{c,s,\om},
\Prob_{s,\om})$, and  $W^i\in
 G_{\mathrm{loc}}(p_s\mu^X,\Prob_{s,\om})$.
 Then we have the following:
 
 (a) The process  $S^i$  defined by
\begin{align*}
S^i:=S^i_s+\beta^i\sint t
+H^i\sint X^{c,s,\om}
+W^i\ast (\mu^X-\nu)
\end{align*}
 is a special 
$(\F^{0}_+,\Prob^{s,\om})$-semimartingale
on $[s,T]$
with characteristics
$(\tilde{B}^i,\tilde{C}^i,\tilde{\nu}^i)$,
where $\tilde{B}^i=\beta^i\sint t$,
$\tilde{C}^{i}=\sum_{k,l=1}^d
(H^{ik}H^{il} c^{kl})\sint t$,
and $\tilde{\nu}^i$ is a random measure on
$[s,T]\times\R$ defined by
\begin{align*}
\tilde{\nu}^i([s,t]\times A,\tiom):=
\nu(\{(r,z)\in [s,T]\times\R^d:
(r,W^i(r,z,\tiom))\in [s,t]\times 
(A\setminus\{\bfnull\})\},\tiom)
\end{align*}
for every $(t,A,\tiom)\in [s,T]\times
\mathcal{B}(\R)\times\Omega$.

 (b) The $2$-dimensional process
$S:=(S^1,S^2)$ is a special
$(\F^{0}_+,\Prob^{s,\om})$-semi\-mar\-tingale
on $[s,T]$
with characteristics 
$(\tilde{B},\tilde{C},\tilde{\nu)}$, where
$\tilde{B}=(\tilde{B}^1,\tilde{B}^2)$,
$$\tilde{C}^{ij}= \sum_{k,l=1}^d 
(H^{ik}H^{jl} c^{kl})\sint t$$ for
$i$, $j=1$, $2$,
and $\tilde{\nu}$ is a random measure
on $[s,T]\times\R^2$ defined by
\begin{align*}
&\tilde{\nu}([s,t]\times A_1\times A_2,\tiom):=
\nu(\{(r,z)\in [t,T]\times\R^d:\\
&\qquad 
(r,W^1(r,z,\tiom),W^2(r,z,\tiom))\in 
[s,t]\times (A_1\setminus\{\bfnull\})
\times  (A_2\setminus\{\bfnull\})\},\tiom)
\end{align*}
for every $(t,A_1,A_2,\tiom)\in [s,T]\times
\mathcal{B}(\R)\times\mathcal{B}(\R)\times\Omega$.

(c) We have
\begin{align*}
S^1S^2&=S^1_s S^2_s+
\left[S^2\beta^1+S^1\beta^2+
\sum_{k,l=1}^d H^{1,k}H^{2,l}
c^{kl}
+\int_{\R^d}W^1 W^2\,K(dz)\right]\sint t\\
&\qquad +
\sum_{k=1}^d (S^2_-H^{1,k}+S^1_-H^{2,k})
\sint X^{k,c,s,\om}+
(S^1_-x_2+S^2_-x_1+x_1x_2)\ast(\mu^S-\tilde{\nu}).
\end{align*}
\end{lemma}
\begin{proof}
(a) Using Theorem~I.4.40 of \cite{JacodShiryaevBook},
we get
\begin{align*}
\left\langle \sum_{k=1}^d \gamma^{i,k}
\sint X^{k,c,s,\om}
\right\rangle&=
\sum_{k=1}^d \left\langle H^{ik}
\sint X^{k,c,s,\om},\sum_{l=1}^d H^{il}
\sint X^{l,c,s,\om}\right\rangle\\
&=\sum_{k,l=1}^d (H^{ik}H^{il})
\langle X^{k,c,s,\om},X^{l,c,s,\om}\rangle\\
&=\sum_{k,l=1}^d (H^{ik}H^{il}
c^{kl})\sint t.
\end{align*}
 Now it suffices to show that $x\ast\tilde{\nu}^i
=W^i\ast \nu$. 
Define a function $\pi_x:[s,T]\times\R\to\R$
by $\pi_x(t,x):=x$. Then, by Satz~19.1 in
\cite{BauerMI},
\begin{align*}
(x\ast\tilde{\nu}^i)_t&=
\int_s^t \int_{\R}
\pi_x(r,x)\,\tilde{\nu}^i(dr,dx,\tiom)\\
&=\int_s^t \int_{\R^d}
\pi_x(r,W^i(r,z,\tiom))\,\nu(dr,dz,\tiom)\\
&=(W^i\ast \nu)_t.
\end{align*}

(b) By Theorem~I.4.40
 in \cite{JacodShiryaevBook}, 
 \begin{align*}
 \langle S^{i,c},S^{j,c}\rangle&=
 \left\langle \sum_{k=1}^d
 H^{ik}\sint X^{k,c,s,\om},
 \sum_{l=1}^d
 H^{jl}\sint X^{k,c,s,\om}
 \right\rangle\\
 &=\sum_{k,l=1}^d 
 (H^{ik}H^{jl} c^{kl})\sint t.
\end{align*} 
 Next we
  show that 
 $x\ast\tilde{\nu}=(W^1\ast\nu,W^2\ast\nu)$.
 Define functions $\pi_i:[s,T]\times\R^2\to\R$
 by $\pi(t,x_1,x_2):=x_i$, $i=1$, $2$.
 Then, again  by Satz~19.1 in
\cite{BauerMI},
\begin{align*}
(x_i\ast\tilde{\nu})_t&=
\int_s^t\int_{\R^2} \pi_i(r,x_1,x_2)\,
\tilde{\nu}(dr,dx,\tiom)\\
&=\int_s^t\int_{\R^d} 
\pi_i(r,W^1(r,z,\tiom),W^2(r,z,\tiom))\,
\nu(dr,dz,\tiom)\\
&=(W^i\ast\nu)_t.
\end{align*}

(c) By (b) and by It\^o's formula based on local characteristics
(see, e.g., Section~2.1.2 
of \cite{JacodProtterBook}),
\begin{align*}
S^1 S^2&=S^1_sS^2_s+
S^2_-\sint\tilde{B}^1+S^1_-\sint\tilde{B}^2
+\tilde{C}^{1,2}\\
&\qquad +\left[
(S^1_-+x_1)(S^2_-+x_2)-S^1_- S^2_-
-S^2_- x_1-S^1_- x_2 
\right]\ast\tilde{\nu}\\
&\qquad + S^2_-\sint S^{1,c}+S^1_-\sint S^{2,c}\\
&\qquad +\left[
(S^1_-+x_1)(S^2_-+x_2)-S^1_- S^2_-
\right]\ast(\mu^S-\tilde{\nu}).
\end{align*}
To conclude, note that 
$S^j_-\sint \tilde{B}^i=(S^j\beta^i)\sint t$,
$\tilde{C}^{1,2}=\sum_{k,l=1}^d
(H^{1,k}H^{2,l}c^{kl})
\sint t$,
\begin{align*}
&\left[
(S^1_-+x_1)(S^2_-+x_2)-S^1_- S^2_-
-S^2_- x_1-S^1_- x_2 
\right]\ast\tilde{\nu}\\\qquad&=
 x_1 x_2\ast \tilde{\nu}\\\qquad&=
  (W^1 W^2)\ast\nu 
&&\text{by Satz~19.1 of \cite{BauerMI}}  
  \\\qquad&=
  (W^1 W^2)\ast (dr\otimes K(dz)),
\end{align*}
and
$S^j_-\sint S^{i,c}
=\sum_{k=1}^d(S^j_-\gamma^{i,k})\sint
 X^{k,c,s,\om}$.
\end{proof}
\bibliographystyle{amsplain}
\bibliography{PPIDE_bibliography}

\def\cprime{$'$} \def\cprime{$'$} \def\cprime{$'$} \def\cprime{$'$}
  \def\polhk#1{\setbox0=\hbox{#1}{\ooalign{\hidewidth
  \lower1.5ex\hbox{`}\hidewidth\crcr\unhbox0}}}
\begin{bibdiv}
\begin{biblist}

\bib{AubinHaddad02}{article}{
      author={Aubin, J.-P.},
      author={Haddad, G.},
       title={History path dependent optimal control and portfolio valuation
  and management},
        date={2002},
        ISSN={1385-1292},
     journal={Positivity},
      volume={6},
      number={3},
       pages={331\ndash 358},
         url={http://dx.doi.org/10.1023/A:1020244921138},
        note={Special issue of the mathematical economics},
      review={\MR{1932655 (2003h:49008)}},
}

\bib{BarlesEtAl97}{article}{
      author={Barles, Guy},
      author={Buckdahn, Rainer},
      author={Pardoux, Etienne},
       title={Backward stochastic differential equations and integral-partial
  differential equations},
        date={1997},
        ISSN={1045-1129},
     journal={Stochastics Stochastics Rep.},
      volume={60},
      number={1-2},
       pages={57\ndash 83},
      review={\MR{1436432 (98d:60110)}},
}

\bib{BauerMI}{book}{
      author={Bauer, Heinz},
       title={Ma\ss - und {I}ntegrationstheorie},
     edition={2. Auflage},
      series={de Gruyter Lehrbuch.},
   publisher={Walter de Gruyter \& Co.},
     address={Berlin},
        date={1992},
        ISBN={3-11-013626-0; 3-11-013625-2},
      review={\MR{1181881 (93g:28001)}},
}

\bib{BayraktarYao12QBSDEs}{article}{
      author={Bayraktar, Erhan},
      author={Yao, Song},
       title={Quadratic reflected {BSDE}s with unbounded obstacles},
        date={2012},
        ISSN={0304-4149},
     journal={Stochastic Process. Appl.},
      volume={122},
      number={4},
       pages={1155\ndash 1203},
         url={http://dx.doi.org/10.1016/j.spa.2011.12.013},
      review={\MR{2914748}},
}

\bib{BlackwellDubins75}{article}{
      author={Blackwell, David},
      author={Dubins, Lester~E.},
       title={On existence and non-existence of proper, regular, conditional
  distributions},
        date={1975},
     journal={Ann. Probability},
      volume={3},
      number={5},
       pages={741\ndash 752},
      review={\MR{0400320 (53 \#4155)}},
}

\bib{CarboneEtAl08_BSDE}{article}{
      author={Carbone, R.},
      author={Ferrario, B.},
      author={Santacroce, M.},
       title={Backward stochastic differential equations driven by c\`adl\`ag
  martingales},
        date={2008},
     journal={Theory of Probability \& Its Applications},
      volume={52},
      number={2},
       pages={304\ndash 314},
}

\bib{ConfortolaFuhrman13}{article}{
      author={Confortola, Fulvia},
      author={Fuhrman, Marco},
       title={Backward {S}tochastic {D}ifferential {E}quations and {O}ptimal
  {C}ontrol of {M}arked {P}oint {P}rocesses},
        date={2013},
        ISSN={0363-0129},
     journal={SIAM J. Control Optim.},
      volume={51},
      number={5},
       pages={3592\ndash 3623},
         url={http://dx.doi.org/10.1137/120902835},
      review={\MR{3105784}},
}

\bib{ContFournie10_CompteRendus}{article}{
      author={Cont, Rama},
      author={Fournie, David},
       title={A functional extension of the {I}to formula},
        date={2010},
        ISSN={1631-073X},
     journal={C. R. Math. Acad. Sci. Paris},
      volume={348},
      number={1-2},
       pages={57\ndash 61},
         url={http://dx.doi.org/10.1016/j.crma.2009.11.013},
      review={\MR{2586744 (2010m:60221)}},
}

\bib{ContFournie10_JFA}{article}{
      author={Cont, Rama},
      author={Fourni{\'e}, David-Antoine},
       title={Change of variable formulas for non-anticipative functionals on
  path space},
        date={2010},
        ISSN={0022-1236},
     journal={J. Funct. Anal.},
      volume={259},
      number={4},
       pages={1043\ndash 1072},
         url={http://dx.doi.org/10.1016/j.jfa.2010.04.017},
      review={\MR{2652181 (2011h:60112)}},
}

\bib{ContFournier13_AOP}{article}{
      author={Cont, Rama},
      author={Fourni{\'e}, David-Antoine},
       title={Functional {I}t\^o calculus and stochastic integral
  representation of martingales},
        date={2013},
        ISSN={0091-1798},
     journal={Ann. Probab.},
      volume={41},
      number={1},
       pages={109\ndash 133},
         url={http://dx.doi.org/10.1214/11-AOP721},
      review={\MR{3059194}},
}

\bib{CrandallLions83TAMS}{article}{
      author={Crandall, Michael~G.},
      author={Lions, Pierre-Louis},
       title={Viscosity solutions of {H}amilton-{J}acobi equations},
        date={1983},
        ISSN={0002-9947},
     journal={Trans. Amer. Math. Soc.},
      volume={277},
      number={1},
       pages={1\ndash 42},
         url={http://dx.doi.org/10.2307/1999343},
      review={\MR{690039 (85g:35029)}},
}

\bib{CrepeyMatoussi08_RBSDEs}{article}{
      author={Cr{\'e}pey, St{\'e}phane},
      author={Matoussi, Anis},
       title={Reflected and doubly reflected {BSDE}s with jumps: a priori
  estimates and comparison},
        date={2008},
        ISSN={1050-5164},
     journal={Ann. Appl. Probab.},
      volume={18},
      number={5},
       pages={2041\ndash 2069},
         url={http://dx.doi.org/10.1214/08-AAP517},
      review={\MR{2462558 (2009m:60136)}},
}

\bib{DellacherieMeyer}{book}{
      author={Dellacherie, Claude},
      author={Meyer, Paul-Andr{\'e}},
       title={Probabilities and potential},
      series={North-Holland Mathematics Studies},
   publisher={North-Holland Publishing Co.},
     address={Amsterdam},
        date={1978},
      volume={29},
        ISBN={0-7204-0701-X},
      review={\MR{521810 (80b:60004)}},
}

\bib{DelongBook}{book}{
      author={Delong, {\L}ukasz},
       title={Backward stochastic differential equations with jumps and their
  actuarial and financial applications},
      series={European Actuarial Academy (EAA) Series},
   publisher={Springer, London},
        date={2013},
        ISBN={978-1-4471-5330-6; 978-1-4471-5331-3},
         url={http://dx.doi.org/10.1007/978-1-4471-5331-3},
        note={BSDEs with jumps},
      review={\MR{3089193}},
}

\bib{dupirefunctional}{misc}{
      author={Dupire, Bruno},
       title={Functional {I}to calculus},
   publisher={SSRN},
        date={2010},
}

\bib{Ekren13}{misc}{
      author={Ekren, Ibrahim},
       title={Viscosity solutions of obstacle problems for fully nonlinear
  path-dependent {PDE}s},
   publisher={preprint},
        date={2013},
}

\bib{EKTZ11}{article}{
      author={Ekren, Ibrahim},
      author={Keller, Christian},
      author={Touzi, Nizar},
      author={Zhang, Jianfeng},
       title={On viscosity solutions of path dependent {PDE}s},
        date={2014},
        ISSN={0091-1798},
     journal={Ann. Probab.},
      volume={42},
      number={1},
       pages={204\ndash 236},
}

\bib{ETZ12a}{misc}{
      author={Ekren, Ibrahim},
      author={Touzi, Nizar},
      author={Zhang, Jianfeng},
       title={Viscosity solutions of fully nonlinear parabolic path dependent
  {PDE}s: Part {I}},
   publisher={preprint},
        date={2012},
}

\bib{ETZ12b}{misc}{
      author={Ekren, Ibrahim},
      author={Touzi, Nizar},
      author={Zhang, Jianfeng},
       title={Viscosity solutions of fully nonlinear parabolic path dependent
  {PDE}s: Part {II}},
   publisher={preprint},
        date={2012},
}

\bib{ElKarouiHuang97}{incollection}{
      author={El~Karoui, N.},
      author={Huang, S.-J.},
       title={A general result of existence and uniqueness of backward
  stochastic differential equations},
        date={1997},
   booktitle={Backward stochastic differential equations ({P}aris,
  1995--1996)},
      series={Pitman Res. Notes Math. Ser.},
      volume={364},
   publisher={Longman},
     address={Harlow},
       pages={27\ndash 36},
      review={\MR{1752673 (2000m:60069)}},
}

\bib{ElKarouiEtAl97_RBSDEs}{article}{
      author={El~Karoui, N.},
      author={Kapoudjian, C.},
      author={Pardoux, E.},
      author={Peng, S.},
      author={Quenez, M.~C.},
       title={Reflected solutions of backward {SDE}'s, and related obstacle
  problems for {PDE}'s},
        date={1997},
        ISSN={0091-1798},
     journal={Ann. Probab.},
      volume={25},
      number={2},
       pages={702\ndash 737},
         url={http://dx.doi.org/10.1214/aop/1024404416},
      review={\MR{1434123 (98k:60096)}},
}

\bib{HamadeneOuknine03}{article}{
      author={Hamad{\`e}ne, S.},
      author={Ouknine, Y.},
       title={Reflected backward stochastic differential equation with jumps
  and random obstacle},
        date={2003},
        ISSN={1083-6489},
     journal={Electron. J. Probab.},
      volume={8},
       pages={no. 2, 20},
         url={http://dx.doi.org/10.1214/EJP.v8-124},
      review={\MR{1961164 (2003m:60148)}},
}

\bib{HeitzingerThesis}{thesis}{
      author={Heitzinger, Clemens},
       title={Simulation and inverse modeling of semiconductor manufacturing
  processes},
        type={Dissertation},
 institution={Technische Universit\"at Wien},
        date={2002},
}

\bib{JacodProtterBook}{book}{
      author={Jacod, Jean},
      author={Protter, Philip},
       title={Discretization of processes},
      series={Stochastic Modelling and Applied Probability},
   publisher={Springer},
     address={Heidelberg},
        date={2012},
      volume={67},
        ISBN={978-3-642-24126-0},
         url={http://dx.doi.org/10.1007/978-3-642-24127-7},
      review={\MR{2859096}},
}

\bib{JacodShiryaevBook}{book}{
      author={Jacod, Jean},
      author={Shiryaev, Albert~N.},
       title={Limit theorems for stochastic processes},
     edition={Second},
      series={Grundlehren der Mathematischen Wissenschaften [Fundamental
  Principles of Mathematical Sciences]},
   publisher={Springer-Verlag},
     address={Berlin},
        date={2003},
      volume={288},
        ISBN={3-540-43932-3},
      review={\MR{1943877 (2003j:60001)}},
}

\bib{KimBook}{book}{
      author={Kim, A.~V.},
       title={Functional differential equations},
      series={Mathematics and its Applications},
   publisher={Kluwer Academic Publishers},
     address={Dordrecht},
        date={1999},
      volume={479},
        ISBN={0-7923-5689-6},
        note={Application of $i$-smooth calculus},
      review={\MR{1783365 (2001i:34108)}},
}

\bib{Lukoyanov00}{article}{
      author={Lukoyanov, N.~Yu.},
       title={Functional equations of {H}amilton-{J}acobi type, and
  differential games with hereditary information},
        date={2000},
        ISSN={0869-5652},
     journal={Dokl. Akad. Nauk},
      volume={371},
      number={4},
       pages={457\ndash 461},
      review={\MR{1776305 (2001e:49066)}},
}

\bib{Lukoyanov07}{article}{
      author={Lukoyanov, N.~Yu.},
       title={Viscosity solution of nonanticipating equations of
  {H}amilton-{J}acobi type},
        date={2007},
        ISSN={0374-0641},
     journal={Differ. Uravn.},
      volume={43},
      number={12},
       pages={1674\ndash 1682, 1727},
         url={http://dx.doi.org/10.1134/S0012266107120117},
      review={\MR{2435995 (2009d:35026)}},
}

\bib{Lukoyanov03}{article}{
      author={Lukoyanov, Nikolai},
       title={Functional {H}amilton-{J}acobi type equations in ci-derivatives
  for systems with distributed delays},
        date={2003},
        ISSN={1229-1595},
     journal={Nonlinear Funct. Anal. Appl.},
      volume={8},
      number={3},
       pages={365\ndash 397},
      review={\MR{2010832 (2004f:34124)}},
}

\bib{MikuleviciusPragarauskas92MGproblem}{article}{
      author={Mikulevicius, R.},
      author={Pragarauskas, H.},
       title={On the martingale problem associated with nondegenerate {L\'evy}
  operators},
    language={English},
        date={1992},
        ISSN={0363-1672},
     journal={Lithuanian Mathematical Journal},
      volume={32},
      number={3},
       pages={297\ndash 311},
         url={http://dx.doi.org/10.1007/BF00971435},
}

\bib{Mikulevicius1994classical}{article}{
      author={Mikulevicius, R.},
      author={Pragarauskas, H.},
       title={On classical solutions of boundary value problems for certain
  nonlinear integro-differential equations},
        date={1994},
     journal={Lithuanian Mathematical Journal},
      volume={34},
      number={3},
       pages={275\ndash 287},
}

\bib{Morlais13RBSDEs}{article}{
      author={Morlais, Marie-Amelie},
       title={Reflected backward stochastic differential equations and a class
  of non-linear dynamic pricing rule},
        date={2013},
        ISSN={1744-2508},
     journal={Stochastics},
      volume={85},
      number={1},
       pages={1\ndash 26},
         url={http://dx.doi.org/10.1080/17442508.2011.652115},
      review={\MR{3011909}},
}

\bib{PengGexpect}{incollection}{
      author={Peng, Shige},
       title={{$G$}-expectation, {$G$}-{B}rownian motion and related stochastic
  calculus of {I}t\^o type},
        date={2007},
   booktitle={Stochastic analysis and applications},
      series={Abel Symp.},
      volume={2},
   publisher={Springer, Berlin},
       pages={541\ndash 567},
         url={http://dx.doi.org/10.1007/978-3-540-70847-6_25},
      review={\MR{2397805 (2009c:60220)}},
}

\bib{PengICM}{inproceedings}{
      author={Peng, Shige},
       title={Backward stochastic differential equation, nonlinear expectation
  and their applications},
        date={2010},
   booktitle={Proceedings of the {I}nternational {C}ongress of
  {M}athematicians. {V}olume {I}},
   publisher={Hindustan Book Agency},
     address={New Delhi},
       pages={393\ndash 432},
      review={\MR{2827899}},
}

\bib{PhamZhang12}{article}{
      author={Pham, T.},
      author={Zhang, J.},
       title={Two person zero-sum game in weak formulation and path dependent
  bellman--isaacs equation},
        date={2014},
     journal={SIAM Journal on Control and Optimization},
      volume={52},
      number={4},
       pages={2090\ndash 2121},
      eprint={http://dx.doi.org/10.1137/120894907},
         url={http://dx.doi.org/10.1137/120894907},
}

\bib{Ren14}{misc}{
      author={Ren, Zhenjie},
       title={Viscosity solutions of fully nonlinear elliptic path dependent
  {PDE}s},
   publisher={preprint},
        date={2014},
}

\bib{RogersWilliamsI}{book}{
      author={Rogers, L. C.~G.},
      author={Williams, David},
       title={Diffusions, {M}arkov processes, and martingales. {V}ol. 1},
      series={Cambridge Mathematical Library},
   publisher={Cambridge University Press},
     address={Cambridge},
        date={2000},
        ISBN={0-521-77594-9},
        note={Foundations, Reprint of the second (1994) edition},
      review={\MR{1796539 (2001g:60188)}},
}

\bib{SVBook}{book}{
      author={Stroock, Daniel~W.},
      author={Varadhan, S. R.~Srinivasa},
       title={Multidimensional diffusion processes},
      series={Classics in Mathematics},
   publisher={Springer-Verlag},
     address={Berlin},
        date={2006},
        ISBN={978-3-540-28998-2; 3-540-28998-4},
        note={Reprint of the 1997 edition},
      review={\MR{2190038 (2006f:60005)}},
}

\bib{Subbotin80}{article}{
      author={Subbotin, A.~I.},
       title={Generalization of the fundamental equation of the theory of
  differential games},
        date={1980},
        ISSN={0002-3264},
     journal={Dokl. Akad. Nauk SSSR},
      volume={254},
      number={2},
       pages={293\ndash 297},
      review={\MR{587158 (82f:90104)}},
}

\bib{Subbotin84}{article}{
      author={Subbotin, A.~I.},
       title={Generalization of the main equation of differential game theory},
        date={1984},
        ISSN={0022-3239},
     journal={J. Optim. Theory Appl.},
      volume={43},
      number={1},
       pages={103\ndash 133},
         url={http://dx.doi.org/10.1007/BF00934749},
      review={\MR{745789 (85i:90164)}},
}

\bib{SubbotinBook}{book}{
      author={Subbotin, Andre{\u\i}~I.},
       title={Generalized solutions of first-order {PDE}s},
      series={Systems \& Control: Foundations \& Applications},
   publisher={Birkh\"auser Boston Inc.},
     address={Boston, MA},
        date={1995},
        ISBN={0-8176-3740-0},
        note={The dynamical optimization perspective, Translated from the
  Russian},
      review={\MR{1320507 (96b:49002)}},
}

\bib{WhittBook}{book}{
      author={Whitt, Ward},
       title={Stochastic-process limits},
      series={Springer Series in Operations Research},
   publisher={Springer-Verlag},
     address={New York},
        date={2002},
        ISBN={0-387-95358-2},
        note={An introduction to stochastic-process limits and their
  application to queues},
      review={\MR{1876437 (2003f:60005)}},
}

\bib{Xia00}{article}{
      author={Xia, Jianming},
       title={Backward stochastic differential equation with random measures},
        date={2000},
        ISSN={0168-9673},
     journal={Acta Math. Appl. Sinica (English Ser.)},
      volume={16},
      number={3},
       pages={225\ndash 234},
         url={http://dx.doi.org/10.1007/BF02679887},
      review={\MR{1779016 (2001h:60113)}},
}

\end{biblist}
\end{bibdiv}
\end{document}